\documentclass[11pt]{amsart}
\usepackage{qr-Bessel2}

\begin{document}
\title{Quadratic relations between Bessel moments}

\author[J. Fres\'an]{Javier Fres\'an}
\address[J. Fres\'an]{CMLS, \'Ecole polytechnique, F-91128 Palaiseau cedex, France}
\email{javier.fresan@polytechnique.edu}
\urladdr{http://javier.fresan.perso.math.cnrs.fr}

\author[C.~Sabbah]{Claude Sabbah}
\address[C.~Sabbah]{CMLS, CNRS, École polytechnique, Institut Polytechnique de Paris, 91128 Palaiseau cedex, France}
\email{Claude.Sabbah@polytechnique.edu}
\urladdr{http://www.math.polytechnique.fr/perso/sabbah}

\author[J.-D. Yu]{Jeng-Daw Yu}
\address[J.-D. Yu]{Department of Mathematics, National Taiwan University,
Taipei 10617, Taiwan}
\email{jdyu@ntu.edu.tw}
\urladdr{http://homepage.ntu.edu.tw/~jdyu/}

\thanks{This research was partly supported by the PICS project TWN 8094 from CNRS. The research of J.F. was also partly supported by the grant ANR-18-CE40-0017 of Agence Nationale de la Recherche.}

\begin{abstract} Motivated by the computation of some Feynman amplitudes, Broadhurst and Roberts recently conjectured and checked numerically to high precision a set of remarkable quadratic relations between the Bessel moments
\begin{displaymath}
\int_0^\infty\hspace*{-2mm}I_0(t)^i K_0(t)^{k-i}t^{2j-1}\,\mathrm{d}t \qquad (i, j=1, \ldots, \lfloor (k-1)/2\rfloor),
\end{displaymath}
where $k \geq 1$ is a fixed integer and $I_0$ and $K_0$ denote the modified Bessel functions. In this paper, we interpret these integrals and variants thereof as coefficients of the period pairing between middle de Rham cohomology and twisted homology of symmetric powers of the Kloosterman connection. Building on the general framework developed in \cite{F-S-Y20}, this enables us to prove quadratic relations of the form suggested by Broadhurst and Roberts, which conjecturally comprise all algebraic relations between these numbers. We also make Deligne's conjecture explicit, thus explaining many evaluations of critical values of $L$-functions of symmetric power moments of Kloosterman sums in terms of determinants of Bessel moments. 
\end{abstract}



\maketitle
{\let\\\relax\tableofcontents}
\mainmatter

\section{Introduction}

Let $I_0(t)$ and $K_0(t)$ denote the modified Bessel functions of order zero, which are solutions to the ordinary differential equation~\hbox{$((t\partial_t)^2-t^2)u=0.$}
Since this equation has an irregular singularity at infinity,
it does not come from geometry in the usual sense of encoding how periods vary in a family of algebraic varieties. However, certain integrals of monomials in $I_0(t)$ and $K_0(t)$ called \emph{Bessel moments} are themselves periods, as shown for example by the identity (\cf \cite[(8.11)]{Vanhove14})
\[
\int_0^\infty\hspace*{-2mm}I_0(t) K_0(t)^{\ell+1}t \,\de t=\frac{1}{2^\ell}\int_{x_i \geq 0} \frac{1}{(1+\sum_{i=1}^\ell x_i)(1+\sum_{i=1}^\ell \sfrac{1}{x_i})-1}\prod_{i=1}^\ell \frac{\de x_i}{x_i}.
\]

In a series of papers and conference talks \cite{Broadhurst16, Broadhurst17, Broadhurst17b, BKLF, B-R18, Roberts17}, Broadhurst and Roberts put forward a program to understand the motivic origin of the Bessel moments
\begin{equation}\label{eqn:BesselmomentsMabc}
\int_0^\infty\hspace*{-2mm}I_0(t)^a K_0(t)^b t^c \,\de t.
\end{equation} An important insight of theirs was to look at counterparts of these integrals over finite fields, pursuing the analogy between the Bessel differential equation and the Kloosterman $\ell$-adic sheaf. Roughly speaking, $I_0(t)$ and $K_0(t)$ correspond to the eigenvalues of Frobenius, and out of them
one forms the $k$-th symmetric power moments of Kloosterman sums.
The generating series of these moments over finite extensions of $\FF_p$ is a polynomial with integer coefficients. After removing some trivial factors and handling primes of bad reduction, a global $L$-function~$L_k(s)$ is built with the reciprocals of these polynomials as local factors. Back at the beginning, Broadhurst, partly in joint work with Mellit \cite{B-M16} and Roberts, Bloch-Kerr-Vanhove \cite{BKV}, and Y.\,Zhou \cite{wick} proved or numerically checked in many cases that the critical values of these $L$\nobreakdash-functions agree up to rational factors and powers of $\pi$ with certain determinants of the Bessel moments~\eqref{eqn:BesselmomentsMabc}.

For technical reasons, we shall make the change of variables~$z=\sfrac{t^2}{4}$ and consider the associated rank-two vector bundle with connection on $\Gm$, which is called the \emph{Kloosterman connection} and denoted by $\Kl_2$.
Motives associated with symmetric powers of the Kloosterman connection were introduced in~\cite{F-S-Y18}. Namely, for each integer $k \geq 1$, we constructed a motive $\Motive_k$
over the rational numbers, which is pure of weight $k+1$, has rank $k'=\flr{(k-1)/2}$ (\resp $k'-1$)
if $k$ is not a multiple of $4$ (\resp if $k$ is a multiple of $4$),
and is endowed with a self-duality pairing
\begin{equation}\label{eqn:motivicpairing}
\Motive_k \otimes \Motive_k \to \QQ(-k-1)
\end{equation} that is symplectic if $k$ is even
and orthogonal if $k$ is odd. By design, the $L$-function of this motive coincides with the above $L$-function $L_k(s)$. The main result of that paper was the computation of the Hodge numbers of $\Motive_k$, which led to a proof that~$L_k(s)$ extends meromorphically to the complex plane and satisfies the expected functional equation.

In this paper, we investigate the period realizations of the motives $\Motive_k$.
By design, the de Rham realization of $\Motive_k$ is isomorphic to the middle de Rham cohomology of the~$k$-th symmetric power $\Sym^k \Kl_2$, which is defined as the image
\[
\coH^1_{\dR, \rmid}(\Gm, \Sym^k \Kl_2)=\image\bigl[\coH^1_{\dR,\cp}(\Gm, \Sym^k \Kl_2)\ra\coH^1_{\dR}(\Gm, \Sym^k \Kl_2)\bigr]
\] of compactly supported de Rham cohomology under the natural map to usual de Rham cohomology,
and comes with a perfect intersection pairing
$\DRpairing^\rmid_k$ realizing \eqref{eqn:motivicpairing}. Extending a computation from \cite[Prop.\,4.14]{F-S-Y18},
we exhibit a basis of middle de Rham cohomology in Section~\ref{sec:DRpairing}, which is natural in that it is adapted to the Hodge filtration,
and we present an explicit formula
to compute the matrix of $\DRpairing^\rmid_k$ on this basis. If $k$ is not a multiple of $4$, the basis is simply given by the classes~$\omega_i=[z^iv_0^k\rd z/z]$ for $1 \leq i \leq k'$, where $v_0$ is a specific section of $\Kl_2$.

Besides,
we shall prove that
the dual of the Betti realization of $\Motive_k$ is isomorphic to the middle twisted homology of $\Sym^k \Kl_2$, which is defined as the image
\[
\coH_1^\rmid(\Gm,\Sym^k\Kl_2)=\image\bigl[\coH^\rrd_1(\Gm,\Sym^k\Kl_2) \ra \coH^{\rmod}_1(\Gm, \Sym^k\Kl_2) \bigr]
\] of rapid decay homology under the natural map to moderate growth homology. Elements of these homology groups are represented by linear combinations of twisted chains $c \otimes e$,
where $c$ is a path and $e$ is a horizontal section of $\Sym^k \Kl_2$ that decays rapidly (\resp has moderate growth) on a neighborhood of $c$. These conditions ensure that de Rham (\resp compactly supported de Rham) cohomology classes can be integrated along them, thus giving rise to a period pairing
\[
\Ppairing^\rmid_{k}: \coH^\rmid_1(\Gm,\Sym^k\Kl_2) \otimes \coH^1_{\dR, \rmid}(\Gm, \Sym^k \Kl_2) \to \CC.
\]
This middle homology comes with a natural $\QQ$-structure and, likewise to middle de Rham cohomology, a perfect intersection pairing $\Bpairing^\rmid_k$ realizing the transpose of \eqref{eqn:motivicpairing}.
By analyzing the asymptotic behaviors
of products of modified Bessel functions,
we exhibit in Section \ref{sec:Bpairing} rapid decay homology classes $\alpha_i$ for~$0 \leq i \leq k'$ whose images in middle homology are non-zero for~$i \geq 1$.

Relying on the general results from the companion paper~\cite{F-S-Y20}, in particular the compatibility of the Betti and de Rham intersection pairings with the period pairing, we prove the following theorem. For simplicity, we only state it here when $k$ is not a multiple of $4$, postponing the full statements to Theorem~\ref{th:Smid}, Proposition~\ref{prop:Bk}, Theorem \ref{th:Bettibasis}, Theorem~\ref{th:middlequad}, and Corollary~\ref{cor:PmidasBessel}.

\begin{thm}\label{thm:maintheo-intro} Assume $k$ is not a multiple of $4$.
\begin{enumerate}
\item With respect to the basis $\{\omega_i\}_{1 \leq i \leq k'}$, the matrix of the de Rham intersection pairing $\DRpairing^\rmid_k$ is a lower\nobreakdash-right triangular matrix with coefficients in $\QQ$ and $(i, j)$ anti\nobreakdash-diagonal entries
\[
\begin{cases}
(-2)^{k'}\,\dfrac{k'!}{k!!}&\text{if $k$ is odd},\\[5pt]
\dfrac{(-1)^{k'+1}}{2^{k'}(j-i)}\cdot\dfrac{(k-1)!!}{(k'+1)!}&\text{if $k$ is even}.
\end{cases}
\]

\item The middle homology classes $\{\alpha_i\}_{1 \leq i \leq k'}$ form a basis and the matrix of the Betti intersection pairing $\Bpairing_k^\rmid$  on this basis is given by
\[
\Bpairing_k^\rmid= \begin{pmatrix}
(-1)^{k-i}\,\dfrac{(k-i)!(k-j)!}{k!}\, \dfrac{\Bern{k-i-j+1}}{(k-i-j+1)!}
\end{pmatrix}_{1 \leq i,j \leq k'},
\]
where $\Bern{n}$ denotes the $n$-th Bernoulli number.

\item With respect to the bases $\{\alpha_i\}_{1 \leq i \leq k'}$ and $\{\omega_j\}_{1 \leq j \leq k'}$, the matrix of the period pairing $\Ppairing^\rmid_k$ consists of the Bessel moments
\[
\Ppairing^\rmid_k=\begin{pmatrix}\dpl (-1)^{k-i}\,2^{k+1-2j}(\pii)^i\int_0^\infty I_0(t)^i K_0(t)^{k-i} t^{2j-1} \,\de t\end{pmatrix}_{1 \leq i, j \leq k'}.
\]

\item The following quadratic relations hold:
\[
\Ppairing_k^\rmid\cdot(\DRpairing_k^\rmid)^{-1}\cdot{}^t\Ppairing_k^\rmid=(-2\pii)^{k+1}\,\Bpairing_k^\rmid.
\]
\end{enumerate}
\end{thm}

Quadratic relations of the shape $\Ppairing_k^\BR\cdot \Dpairing_k^\BR\cdot {}^t\Ppairing_k^\BR = \Bpairing_k^\BR$ were conjectured by Broadhurst and Roberts in \cite{B-R18}. As we explain in Section \ref{subsec:BR}, their matrices $\Ppairing_k^\BR$ and~$\Bpairing_k^\BR$ coincide with ours up to different normalizations, but we were unfortunately unable to prove that, again up to normalization,
the inverse of~$\DRpairing_k^\rmid$ satisfies the recursive formulas
defining their matrix $\Dpairing_k^\BR$.
Nevertheless, we checked numerically that both matrices agree for $k \leq 22$,
which is the limit for reasonable computation time with Maple.

Grothendieck's period conjecture predicts that the transcendence degree of the field of periods of $\Motive_k$ agrees with the dimension of its motivic Galois group. Since the Betti intersection pairing is motivic,~this is a subgroup of the general orthogonal group~$\mathrm{GO}_{k'}$ if $k$ is odd and of the general symplectic group~$\mathrm{GSp}_{k'}$ (\resp $\mathrm{GSp}_{k'-1}$) if $k$ is even and not a multiple of $4$ (\resp if~$k$ is a multiple of $4$). Broadhurst and Roberts conjecture that this inclusion is an equality, which would mean that for fixed $k$ the quadratic relations from Theorem~\ref{thm:maintheo-intro} conjecturally exhaust all algebraic relations between the Bessel moments. 

Finally, in Section \ref{subsec:Deligne} we make Deligne's conjecture explicit for the critical values of~$L_k(s)$ by identifying the periods that are expected to agree with them up to a rational factor with some determinants of Bessel moments already considered by Broadhurst and Roberts. Prior to that, we identify in Section \ref{sect:Lfunctions} the period structure of the motive $\Motive_k$ with the period structure attached to the middle cohomology of $\Sym^k\Kl_2$ by means of Theorem \ref{thm:maintheo-intro}. For that purpose, the appendix develops the necessary tools in a general setting of exponential mixed Hodge structures, complementing thereby the appendix of \cite{F-S-Y18}.

\begin{notation}\label{nota:k'}
We refer the reader to \cite{F-S-Y20} for the general setting of de Rham cohomology and twisted homology of vector bundles with connection, as well as the intersection forms and period pairings on these spaces.
Throughout this article, we use the following notation and conventions:
\begin{itemize}
\item
Given an integer $k\geq1$, we set
\[
k'=\flr{\psfrac{k-1}{2}}\ \text{(\ie $k=2k'+1$ for odd $k$ and $k=2(k'+1)$ for even $k$}).
\]
\item
Since the case where $k$ is a multiple of $4$ plays a special role throughout, we use the simplified common notation:
\[
\lcr1,k'\rcr=\begin{cases}
\{1,\dots,k'\}&\text{if }\knotfour,\\
\{1,\dots,k'\}\moins\{k/4\}&\text{if }\kfour,
\end{cases}
\]
so that
\[
\#\lcr1,k'\rcr=\begin{cases}
k'&\text{if }\knotfour,\\
k'-1&\text{if }\kfour.
\end{cases}
\]
We will consider square matrices indexed by $i,j\in\lcr1,k'\rcr$ that, when~$k$ is a multiple of $4$, are obtained from a $k' \times k'$ matrix by deleting the row and column of index $k/4$.

\item For integers $m \leq 0$, the factorial $m!$ and double factorial $m!!$ are given the value $1$.

\item
For each integer $n \geq 0$, we denote by $\Bern{n}$ the $n$-th Bernoulli number, \ie the $n$-th coefficient of the power series expansion
\[
\frac{x}{e^x-1}=\sum_{n=0}^\infty \Bern{n}\frac{x^n}{n!}.
\]

\item
The base torus is denoted by $\Gmz$,
and is regarded as included in the affine line with coordinate $z$.
The coordinate $1/z$ is denoted by $w$.
We also consider the degree two morphism $\rho_2:\Gmt\to\Gmz$ which, at the ring level, is defined by $z\mto t^2/4$.
\end{itemize}
\end{notation}

\subsubsection*{Acknowledgments} It is our pleasure to thank David Broadhurst and David Roberts for useful correspondence on the subject of their conjecture.

\subsubsection*{Added on the revised version}
Recently, Y.\,Zhou also obtained in \cite{Zhou20} the existence of quadratic relations as conjectured by Broadhurst and Roberts, with a different interpretation of the matrix $\Dpairing_k$ however. The methods are completely different from those of the present article and rely on the previous works of the author.

\section{Pairings on the \texorpdfstring{$\Kl_2$}{Kl2} connection and its moments}\label{sect:pairingsKl2}

In this section, we explain the algebraic duality pairing on $\Sym^k \Kl_2$ that gives rise to the de Rham intersection pairing.
On the other hand, we endow the associated local system of flat sections
$\Sym^k \Kl_2^\nabla$
with a $\QQ$-structure and a topological duality pairing that will give rise to the Betti intersection pairing.

\subsection{The \texorpdfstring{$\Kl_2$}{Kl2} connection}\label{subsec:Kl2connection}

We first recall the definition of the $\Kl_2$ connection, referring the reader to \cite[\S4.1]{F-S-Y18} for more details. We denote by $\Gmx$ (\resp $\Gmz$) the torus $\Gm$ over the complex numbers with coordinate~$x$ (\resp $z$), and we define $f:\Gmx\times\Gmz\to\Afu$ as $f(x,z)=x+z/x$.

Let $\pi: \Gmx\times\Gmz \to \Gmz$ denote the projection to the second factor and $E^f$ the rank\nobreakdash-one vector bundle with connection $(\cO_{\Gmx\times\Gmz},\rd+\rd f)$ on $\Gmx\times\Gmz$. We define $\Kl_2$ as the pushforward (in the sense of $\cD$-modules) $\mathcal{H}^1\pi_+E^f$: this is a free $\cO_{\Gm}$-module of finite rank endowed with a connection having a regular singularity at the origin and an irregular one at infinity.
Since the varieties we work with are all affine,
it will be convenient to identify coherent sheaves with their global sections. To the sheaf $\mathcal{H}^1\pi_+E^f$
is then associated the module $\coH^1\pi_+E^f$ of global sections.
Fixing the generator $\rd x/x$ of relative differentials and denoting by $\partial_x$ the partial derivative with respect to the variable $x$, we then have
\[
\Kl_2=\coH^1\pi_+E^f=\coker\bigl[\CC[x,x^{-1},z,z^{-1}]\To{x\partial_x+(x-z/x)}\CC[x,x^{-1},z,z^{-1}]\bigr].
\]
It follows that $\Kl_2$ is the free $\CC[z,z^{-1}]$-module generated by the class $v_0$ of $\rd x/x$ and the class~$v_1$ of $\rd x$. The connection $\nabla$ on $\Kl_2$ satisfies
\[\arraycolsep3.5pt
z\nabla_{\partial_z}(v_0,v_1)=(v_0,v_1)\cdot\begin{pmatrix}0&z\\1&0\end{pmatrix},
\]
so that $v_0$ is a solution to the differential equation $((z\partial_z)^2-z)v=0$.

Let $j:\Gmx\hto\PP^1$ denote the inclusion. We write $j_\dag$ for the adjoint by duality of the push\-forward $j_+$, and similarly for $\pi$. The same argument as in \cite[App.\,2, Prop.\,(1.7) p.\,217]{Malgrange91} shows that the natural map $(j\times\id)_\dag E^f\to (j\times\id)_+E^f$ is an isomorphism. Projecting to $\Gmz$, we~deduce that
\begin{equation}\label{eq:dagplus}
\coH^1\pi_\dag E^f\to\coH^1\pi_+E^f
\end{equation}
is an isomorphism as well. Let us make this explicit. We set $x'=1/x$. By an argument similar to that of \hbox{\cite[Cor.\,3.5]{F-S-Y20},} we can represent an element of $\coH^1\pi_\dag E^f$ as a pair $(\wh\psi,\eta\rd x/x)=(\wh\psi,-\eta\rd x'/x')$, where
\begin{itemize}
\item
$\eta\in\CC[x,x^{-1},z,z^{-1}]$,
\item
$\wh\psi=(\wh\psi_0,\wh\psi_\infty)$, with
\[
\wh\psi_0\in\CC[z,z^{-1}]\lcr x\rcr[x^{-1}]\quand\wh\psi_\infty\in\CC[z,z^{-1}]\lcr x'\rcr[x'^{-1}],
\]
\end{itemize}
are such that the following holds:
\begin{equation}\label{eqn:definecs}
(x\partial_x+(x-z/x))\wh\psi_0=\iota_{\wh0}\eta,\quad (x'\partial_{x'}+(zx'-1/x'))\wh\psi_\infty=-\iota_{\wh\infty}\eta.
\end{equation}
Here, $\iota_{\wh0} : \CC[x,x^{-1},z,z^{-1}] \hookrightarrow \CC[z,z^{-1}]\lcr x\rcr[x^{-1}]$ denotes the natural inclusion,
and similarly for $\iota_{\wh\infty}$. On these representatives, the natural morphism \eqref{eq:dagplus} is given by $(\wh\psi,\eta\rd x/x)\mto\eta\rd x/x$. Checking that \eqref{eq:dagplus} is an isomorphism amounts to checking that, for any $\eta$ as above, there exists a unique~$\wh\psi$
such that the equations \eqref{eqn:definecs} hold.
Setting $\wh\psi_0=\sum_{n\geq n_0}\psi_{0,n}(z)x^n$
and $\wh\psi_\infty=\sum_{n\geq n_\infty}\psi_{\infty,n}(z)x^{\prime n}$,
\hbox{$\iota_{\wh0}\eta=\sum_n\eta_{0,n}x^n$,}
$\iota_{\wh\infty}\eta=\sum_n\eta_{\infty,n}x^{\prime n}$,
we determine $\psi_{0, n}(z)$ and $\psi_{\infty, n}(z)$ inductively by
\begin{equation}\label{eq:psi0infty}
\psi_{0,n+1}=z^{-1}(k\psi_{0,n}+\psi_{0,n-1}-\eta_{0,n}),
\quad
\psi_{\infty,n+1}=n\psi_{\infty,n}+z\psi_{\infty,n-1}+\eta_{\infty,n},
\end{equation}
thus showing explicitly that \eqref{eq:dagplus} is an isomorphism.

\begin{example}\label{exam:vdag}
The element of $\coH^1\pi_\dag E^f$ corresponding to $v_0$ (\resp $v_1$) is $(\wh\varphi,\rd x/x)$ (\resp $(\wh\psi,\rd x)$), where the elements $\wh\varphi$ and $\wh\psi$ determined by \eqref{eq:psi0infty} satisfy (note that $\eta=1$ \resp $\eta=x=1/x'$)
\[
\begin{cases}
\varphi_{0,\leq0}=0,\\
\varphi_{0,1}=-z^{-1},\\
\varphi_{\infty,\leq0}=0,\\
\varphi_{\infty,1}=1,
\end{cases}
\qquad
\begin{cases}
\psi_{0,\leq1}=0,\\
\psi_{\infty,<0}=0,\\
\psi_{\infty,0}=1,\\
\psi_{\infty,1}=0.
\end{cases}
\]
\end{example}

\subsection{Algebraic duality on \texorpdfstring{$\Kl_2$}{Kl2} and its moments}

Set
\[
D=\{0, \infty\} = \PP^1\setminus\Gm.
\]
Starting from the tautological pairing \hbox{$E^f \otimes E^{-f} \to(\cO_{\Gmx\times\Gmz},\rd)$},
we deduce a natural pairing
\[
\langle\cbbullet,\cbbullet\rangle:\coH^1\pi_\dag E^f\otimes\coH^1\pi_+E^{-f}\to\coH^2\pi_\dag\cO_{\Gmx\times\Gmz}\xrightarrow[\sim]{\textstyle~\res_\Div~}\CC[z^{\pm 1}], 
\]
where the isomorphism $\res_\Div$ stands for the residue along $D$
as in \cite[\S3.c]{F-S-Y20} (see also the proof of Lemma \ref{lemma:alg-pair-Kl2} below).
Let $\iota:\Gmx\times\Gmz\to\Gmx\times\Gmz$ denote the involution $(x,z)\mto(-x,z)$. Then $\iota^+E^{-f}=\nobreak E^f$, and this defines a canonical isomorphism $\mu:\coH^1\pi_+E^f\isom\coH^1\pi_+E^{-f}$ since $\pi\circ\iota=\pi$. Let us set
\[
(v_0^-,v_1^-)=\iota^*(v_0,v_1)=(\rd x/x,-\rd x),
\]
which we consider as a basis of $\coH^1\pi_+E^{-f}$.
Then the matrix of $z\nabla_{\partial_z}$ on $\coH^1\pi_+E^{-f}$
is equal to $\begin{smallpmatrix}0&z\\1&0\end{smallpmatrix}$, and the above isomorphism reads $\mu(v_0,v_1)=(v_0^-,v_1^-)$.

\begin{lemma}\label{lemma:alg-pair-Kl2}
The induced pairing
\[
\langle\cbbullet,\cbbullet\rangle_\alg:\coH^1\pi_+ E^f\otimes\coH^1\pi_+E^f\to\CC[z^{\pm 1}]
\]
defined by $\langle\cbbullet,\cbbullet\rangle_\alg=\langle\eqref{eq:dagplus}^{-1}\cbbullet,\mu\cbbullet\rangle$ satisfies
\[
\langle v_0,v_0\rangle_\alg=\langle v_1,v_1\rangle_\alg=0,\quad\langle v_0,v_1\rangle_\alg=-\langle v_1,v_0\rangle_\alg=1.
\]
\end{lemma}

In other words, we get a skew-symmetric perfect pairing on $\Kl_2$:
\begin{equation}\label{eq:pairingKl}
\langle\cbbullet,\cbbullet\rangle_\alg:(\Kl_2,\nabla)\otimes(\Kl_2,\nabla)\to(\cO_{\Gm},\rd),
\end{equation}
which amounts to a canonical isomorphism $\lambda_\alg:\Kl_2\to\Kl_2^\vee$ with the dual connection endowed with the dual basis $(v_0^\vee,v_1^\vee)$, by setting
\begin{align*}
\Kl_2&\To{\lambda_\alg}\Kl_2^\vee\\
(v_0,v_1)&\Mto{\hphantom{\lambda_\alg}}(-v_1^\vee,v_0^\vee).
\end{align*}

\begin{proof}
We compute with the notation of Example \ref{exam:vdag}. We find, on the one hand,
\begin{align*}
\langle v_0,v_0\rangle_\alg=\langle (\wh\varphi,v_0),v_0^-\rangle&=\res_\Div\wh\varphi\,\frac{\rd x}{x}=\varphi_{0,0}-\varphi_{\infty,0}=0,\\
\langle v_1,v_1\rangle_\alg=\langle (\wh\psi,v_1),v_1^-\rangle&=-\res_\Div\wh\psi\,\rd x=-\psi_{0,-1}+\psi_{\infty,1}=0,
\end{align*}
and, on the other hand,
\begin{align*}
\langle v_0,v_1\rangle_\alg=\langle (\wh\varphi,v_0),v_1^-\rangle&=-\res_\Div\wh\varphi\,\rd x=-\varphi_{0,-1}+\varphi_{\infty,1}=1,\\
\langle v_1,v_0\rangle_\alg=\langle (\wh\psi,v_1),v_0^-\rangle&=\res_\Div\wh\psi\,\frac{\rd x}{x}=\psi_{0,0}-\psi_{\infty,0}=-1.\qedhere
\end{align*}
\end{proof}

For each $k\geq1$, let $\symgp_k$ be the symmetric group
acting on the tensor power $\Kl_2^{\otimes k}$
by the natural permutation.
Let $\Sym^k\Kl_2$ be the symmetric power
regarded as the $\symgp_k$-invariant part of $\Kl_2^{\otimes k}$. We consider the basis $\bmu=(u_a)_{0\leq a\leq k}$ of
$\Sym^k\Kl_2$ given~by
\[
u_a=v_0^{k-a}v_1^a
= \frac{1}{|\symgp_k|} \sum_{\sigma\in\symgp_k}
\sigma (v_0^{\otimes k-a}\otimes v_1^{\otimes a}),
\] in which the connection reads
\begin{equation}\label{eq:ua}
z\partial_z u_a=(k-a)u_{a+1}+azu_{a-1}
\quad(0\leq a\leq k)
\end{equation} with the convention $u_{k+1}=0$. The pairing \eqref{eq:pairingKl} extends to $\Sym^k\Kl_2$, which is thus endowed with the following $(-1)^k$-symmetric pairing (compatible with the connection):
\begin{equation}\label{eq:selfduality}
\langle u_a,u_b\rangle_\alg=
\begin{cases}
(-1)^a\dfrac{a!\,b!}{k!}&\text{if }a+b=k,\\
0&\text{otherwise}.
\end{cases}
\end{equation}

\Subsection{The \texorpdfstring{$\QQ$}{Q}-structure of \texorpdfstring{$\Kl_2^\nabla$}{Kl2d} and its moments}\label{subsec:QKl2}

\subsubsection*{The $\QQ$-structure for a fixed non-zero $z$}
We start by considering the $\QQ$\nobreakdash-struc\-ture
on the fiber of the sheaf of analytic flat sections~$\Kl_2^\nabla$
at some $z\in\Gmz$. We consider the function $f_z:\Gmx\to\Afu$, defined as $f_z(x)=x+z/x$,  where~$z$ is a~fixed non-zero complex number, and $E^{f_z}=(\cO_{\Gmx},\rd+\rd f_z)$. Let $\PP^1_x$ be the projective closure of~$\Gmx$ and let $\wt\PP^1$ be the real oriented blow-up along $D=\{0,\infty\}$, which is topologically a closed annulus. We denote by $\wtj:\Gmx^\an\hto\wt\PP^1$ the inclusion of the open annulus into the closed one. On~$\wt\PP^1$,
the de~Rham complexes with rapid decay $\DR^\rrd(E^{f_z})$
and with moderate growth $\DR^\rmod(E^{f_z})$
have cohomology in degree zero only (\cf \cite[Th.\,2.30]{F-S-Y20}), and the natural map $\DR^\rrd(E^{f_z})\to\DR^\rmod(E^{f_z})$ is a quasi-isomorphism. Indeed, the function~$e^{-f_z}$ on~$\Gmx^\an$ has moderate growth 
near a point of the boundary $\partial\wt\PP^1$
if and only if it has rapid decay there. Above $x=0$, this amounts to $\arg x\in\arg z+(-\pi/2,\pi/2)\bmod2\pi$. Above $x=\infty$, this amounts to $\arg x\in(-\pi/2,\pi/2)\bmod2\pi$. We denote by $\wt\PP^1_\rrd$ the open set which is the union of $\Gmx^\an$ and these two boundary open intervals, so that we have natural open inclusions
\[
\Gmx^\an\Hto{a_z}\wt\PP^1_\rrd\Hto{b_z}\wt\PP^1.
\]
Then multiplication by $e^{-f_z}$ yields an isomorphism of sheaves of $\CC$-vector spaces
\begin{equation}\label{eq:comparison}
b_{z,!}\,a_{z,*}\CC_{\Gmx^\an}\isom\cH^0\DR^\rrd(E^{f_z})=\cH^0\DR^\rmod(E^{f_z}).
\end{equation}

\begin{defi}
The $\QQ$-subsheaf $\cH^0\DR^\rrd(E^{f_z})_\QQ \subset \cH^0\DR^\rrd(E^{f_z})$ is the image of $b_{z,!}\,a_{z,*}\QQ_{\Gmx^\an}$ under the above isomorphism.
\end{defi}

The Betti $\QQ$-structure on $\coH^1_\dR(\Gmx,E^{f_z})$ (\cf\cite[\S2.d]{F-S-Y20}) is defined by means of~\eqref{eq:comparison} as
\[
\coH^1_\dR(\Gmx,E^{f_z})_\QQ
=\coH^1(\wt\PP^1,b_{z,!}\,a_{z,*}\QQ_{\Gmx^\an})
=\coH^1_\rc(\wt\PP^1_\rrd,a_{z,*}\QQ_{\Gmx^\an}).
\]

We denote by (recall that $z\neq0$ is fixed)
\begin{itemize}
\item
$c_0^x$ the unit circle in $\Gmx^\an$ starting at $1$ and oriented counterclockwise;
\item
$c_z^x$ a~smooth oriented path in $\Gmx^\an$ starting in a direction $\arg x$ contained in $\arg z+(-\pi/2,\pi/2)\bmod2\pi$ at $x=0$, intersecting~$c_0^x$ transversally only once, so that the local intersection number (Kronecker index) $(c_z^x,c_0^x)$ is equal to one, and abutting to $x=\infty$ in a direction of~$\arg x$ contained in $(-\pi/2,\pi/2)\bmod2\pi$. The precise choice of $c_z^x$ will be made later. We also consider the path $c_z^{-x}$ obtained from $c_z^x$
by applying the involution $\iota\colon x\mto-x$.
\end{itemize}
They define twisted cycles $\alpha^x_z=-c_z^x\otimes e^{-f_z}$
and $\beta^x_z=c_0^x\otimes e^{-f_z}$
in rapid decay homology $\coH_1^\rrd(\Gmx,E^{f_z})$ (\cf\eg\cite[\S2.d]{F-S-Y20}).
Similarly, we set
\[
(\alpha^x_z)^\vee=-c_z^{-x}\otimes e^{f_z}\quand(\beta^x_z)^\vee=-c_0^x\otimes e^{f_z},
\]
which define twisted cycles in~$\coH_1^\rrd(\Gmx,E^{-f_z})$. As for the de Rham cohomology, the involution $\iota$ induces an isomorphism
$\coH_1^\rrd(\Gmx,E^{f_z}) \to \coH_1^\rrd(\Gmx,E^{-f_z})$
sending a rapid decay chain $s\mto a(s)\otimes e^{-f_z}$ to the rapid decay chain $s\mto -a(s)\otimes e^{f_z}$,
and thus inducing the corresponding Betti intersection pairing
\begin{align*}
\coH_1^\rrd(\Gmx,E^{f_z})\otimes\coH_1^\rrd(\Gmx,E^{f_z})
&\isom \coH_1^\rrd(\Gmx,E^{f_z})\otimes\coH_1^\rrd(\Gmx,E^{-f_z}) \\
&\to \coH_0(\Gmx, \CC) = \CC.
\end{align*}
This pairing is easily computed and has matrix~$\bigl(\begin{smallmatrix}0&1\\1&0\end{smallmatrix}\bigr)$, thus showing that $(\alpha^x_z,\beta^x_z)$ is a $\QQ$-basis of $\coH_1^\rrd(\Gmx,E^{f_z})_\QQ$.

For applying \cite[Prop.\,2.23]{F-S-Y20}, we use the topological duality pairing $\langle\cbbullet,\cbbullet\rangle_\topo$ on $\coH_1^\rrd(\Gmx,E^{f_z})$, which preserves the $\QQ$-structure since it is induced by Poincaré-Verdier duality. The following relation holds (\cf \hbox{\cite[(3.10)]{F-S-Y20}}):
\[
\langle\cbbullet,\cbbullet\rangle_\topo=2\pii\,\langle\cbbullet,\cbbullet\rangle_\alg.
\]
We let $(v_i^{\vee,\topo}=\frac{1}{2\pii}\,v_i^\vee)$ denote the dual basis of $(v_i)$ with respect to $\langle\cbbullet,\cbbullet\rangle_\topo$.

\begin{prop}\label{prop:Kl2Q}
The $\QQ$-vector space $\coH^1_\dR(\Gmx,E^{f_z})_\QQ$ is the $\QQ$-span of
\begin{align*}
e_0&=\Bigl(\frac{1}{2\pii}\int_{c_0^x}e^{-f_z}\,\rd x\Bigr)\cdot v_0-\Bigl(\frac{1}{2\pii}\int_{c_0^x}e^{-f_z}\,\frac{\rd x}{x}\Bigr)\cdot v_1
\quad\text{and} \\
e_1&=-\Bigl(\frac{1}{2\pii}\int_{c_z^x}e^{-f_z}\,\rd x\Bigr)\cdot v_0+\Bigl(\frac{1}{2\pii}\int_{c_z^x}e^{-f_z}\,\frac{\rd x}{x}\Bigr)\cdot v_1.
\end{align*}
\end{prop}

\begin{proof}
From \cite[Prop.\,2.23]{F-S-Y20}, we deduce that $\coH^1_\dR(\Gmx,E^{f_z})_\QQ$ is the $\QQ$-span of
\begin{equation}\label{eq:perEf}
\Ppairing_1^{\rrd,\rmod}(\beta^x_z,v_0^{\vee,\topo})v_0+\Ppairing_1^{\rrd,\rmod}(\beta^x_z,v_1^{\vee,\topo})v_1
\quad\text{and}\quad
\Ppairing_1^{\rrd,\rmod}(\alpha^x_z,v_0^{\vee,\topo})v_0+\Ppairing_1^{\rrd,\rmod}(\alpha^x_z,v_1^{\vee,\topo})v_1,
\end{equation}
where $\Ppairing_1^{\rrd,\rmod}: \coH_1^\rrd(\Gmx,E^{f_z}) \otimes \coH^1_\dR(\Gmx,E^{f_z}) \to \CC$
denotes the period pairing from \cite[\S2.d]{F-S-Y20}.
We conclude with the identification
\[
v_0^{\vee,\topo}=\frac1{2\pii}\,v_1=\frac1{2\pii}\,[\rd x]\quad\text{and}\quad v_1^{\vee,\topo}=-\frac1{2\pii}\,v_0=-\frac1{2\pii}\,[\rd x/x].
\]
For example, the integral $\frac1{2\pii}\int_{c_0^x}e^{-f_z}\,\rd x$ is identified with the period
\[
\Ppairing_1^{\rrd,\rmod}(\beta^x_z,v_0^{\vee,\topo}).\qedhere
\]
\end{proof}

\subsubsection*{The $\QQ$-structure on $\Kl_2^\nabla$}
We first recall some basic properties of the modified Bessel functions of order zero
\begin{equation}\label{eq:BesselIK}
\begin{aligned}
I_0(t)
&= \frac{1}{2\pii}\oint \exp\Big(-\frac{t}{2}\big(y+\sfrac{1}{y}\big)\Big)
\frac{\de y}{y},\\
K_0(t)
&= \frac{1}{2}\int_0^\infty \exp\Big(-\frac{t}{2}\big(y+\sfrac{1}{y}\big)\Big)
\frac{\de y}{y}
\qquad(|\arg t| < \sfrac{\pi}{2}),
\end{aligned}
\end{equation}
which are annihilated by the modified Bessel operator $(t\pt_t)^2 - t^2$. The function $I_0(t)$ is entire and satisfies $I_0(t)=I_0(-t)$. The function $K_0(t)$ extends analytically to a multivalued function on $\CC^\times$ satisfying the rule $K_0(e^{\pii}t) = K_0(t) - \pii I_0(t)$.

We have the following estimates as $t \to 0$ in any bounded ramified sector, by taking the real determination of $\log(t/2)$
when $t\in \RR_{>0}$:
\begin{align*}
I_0(t) &= 1 + O(t^2), \\
K_0(t) &= - \left(\gamma + \log\Psfrac{t}{2}\right) + O(t^2\log t),
\end{align*}
where $\gamma = 0.5772\dots$ is the Euler constant. As a consequence, in such sectors,
\begin{equation}\label{eq:BesselIKzero}
I_0(t)^i K_0(t)^{k-i} = (-1)^{k-i}\left(\gamma + \log\Psfrac{t}{2}\right)^{k-i}
+ O(t^2\log^{k-i}t).
\end{equation}
On the other hand, we have the asymptotic expansions at infinity (\cf \cite[\S7.23]{Watson22})
\begin{equation}\label{eq:BesselIKinfty}
\begin{aligned}
I_0(t)
&\sim e^t\frac{1}{\sqrt{2\pi t}}
\sum_{n=0}^\infty \frac{((2n-1)!!)^2}{2^{3n} n!} \frac{1}{t^n},
&& |\arg t| < \frac{1}{2}\pi \\
K_0(t)
&\sim e^{-t} \sqrt{\frac{\pi}{2t}}
\sum_{n=0}^\infty (-1)^n\frac{((2n-1)!!)^2}{2^{3n} n!} \frac{1}{t^n},
&& |\arg t| < \frac{3}{2}\pi \\
I_0(t)K_0(t)
&\sim \frac{1}{2t}
\sum_{n=0}^\infty \frac{((2n-1)!!)^3}{2^{3n}n!}\frac{1}{t^{2n}}.
\end{aligned}
\end{equation}
The latter is the unique formal solution in $\sfrac{1}{t}$ of the second symmetric power $(t\pt_t)^3 - 4t^2(t\pt_t) - 4t^2$ of the modified Bessel operator up to a scalar. Let us also note that these asymptotic expansions can be derived termwise. The Wronskian is given by $I_0(t)K_0'(t) - I_0'(t)K_0(t) = \sfrac{-1}{t}$.

We now assume that $z$ varies in $\CC\moins\RR_{\leq 0}$
and we choose a square root $t/2$ of $z$ satisfying~\hbox{$\mathrm{Re}(t)>0$,} that is, $\arg t\in(-\pi/2,\pi/2)\bmod2\pi$.
Due to formulas \eqref{eq:perEf},
since the integration paths can be made to vary in a locally constant way,
we conclude that $e_0, e_1$ are sections of $\Kl_2^\nabla$ on this domain, and their coefficients on the basis $v_0,v_1$ are holomorphic there. We will express them in terms of the modified Bessel functions.
We set $x=(t/2)y$. On the one hand, we have
\[
\frac{1}{2\pii}\int_{c_0^x}e^{-f_z}\,\frac{\rd x}{x}=\frac{1}{2\pii}\int_{(t/2)c_0^y}\exp\Bigl(-\frac t2\,(y+1/y)\Bigr)\,\frac{\rd y}{y}
=\frac{1}{2\pii}\int_{c_0^y}\exp\Bigl(-\frac t2\,(y+1/y)\Bigr)\,\frac{\rd y}{y}=I_0(t).
\]
On the other hand, we now regard $y$ as varying in $\RR_{>0}$
and set $c_z^x=(t/2)\RR_{>0}$,
so that, for $x\in c_z^x$, both $\arg(z/x)$ and $\arg(x)$ belong to $(-\pi/2,\pi/2)\bmod2\pi$. Then, similarly,
\[
\frac{1}{2\pii}\int_{c_z^x}e^{-f_z}\,\frac{\rd x}{x}
=\frac{1}{2\pii}\int_{\RR_{>0}}\exp\Bigl(-\frac t2\,(y+1/y)\Bigr)\,\frac{\rd y}{y}
=\frac{1}{\pii}K_0(t).
\]
By flatness of $e_0,e_1$, we obtain therefore
\begin{equation}\label{eq:baree}
e_0=(t/2)I'_0(t)v_0-I_0(t)v_1
\quad\text{and}\quad
e_1=\frac{1}{\pii}(-(t/2)\,K'_0(t)v_0+K_0(t)v_1).
\end{equation}
The pairing $\langle \cbbullet,\cbbullet\rangle_\topo=2\pii\langle \cbbullet,\cbbullet\rangle_\alg$ being flat, it induces a non-degenerate pairing on the constant sheaf $\Kl_2^\nabla{}_{|\CC\moins\RR_{\leq 0}}$ and we have there
\[
\langle e_0,e_1\rangle_\topo=2\pii\,\Bigl((t/2)I'_0(t)\cdot\frac{1}{\pii}K_0(t)-I_0(t)\cdot(t/2)\frac{1}{\pii}K'_0(t)\Bigr)
=t(I_0'(t)K_0(t)-I_0(t)K'_0(t))=1,
\]
according to the Wronskian relation. The other pairings are deduced from this one by skew-symmetry. We also obtain
\begin{equation}\label{eq:ve}
v_0=(2K_0(t)e_0+2\pii I_0(t)e_1),\quad v_1=t(K'_0(t)e_0+\pii I'_0(t)e_1).
\end{equation}

In order to cross the cut $z\in\RR_{<0}$,
we note that the coefficients of $e_0$, regarded as functions of~$z\in\CC$, are entire, while those of $e_1$ are multivalued holomorphic, and the monodromy operator $T$ defined by analytic continuation along the path $\theta\mto e^{\theta}z$ ($\theta\in[0,2\pii]$) acts on $e_1$ as $T(e_1)=e_1+e_0$. This shows that the $\QQ$-structure of $\Kl_2^\nabla{}_{|\CC\moins\RR_{\leq 0}}$ extends to a $\QQ$-structure of $\Kl_2^\nabla$, which will be denoted by $(\Kl_2^\nabla)_\QQ$. Moreover,
\[
\langle \cbbullet,\cbbullet\rangle_\topo: (\Kl_2^\nabla)_\QQ \otimes(\Kl_2^\nabla)_\QQ\to\QQ
\]
is a non-degenerate skew-symmetric pairing,
and the multivalued flat sections $e_0$ and $e_1$ satisfy
$\langle e_0,e_1\rangle_\topo=1$.

\subsubsection*{The $\QQ$-structure on $\Sym^k\Kl_2$}
We naturally endow $\Sym^k\Kl_2$ with the pairing
\[
\langle \cbbullet,\cbbullet\rangle_\topo=(2\pii)^k\langle \cbbullet,\cbbullet\rangle_\alg
\]
and the $\QQ$-structure $(\Sym^k\Kl_2)^\nabla_\QQ=\Sym^k((\Kl_2^\nabla)_\QQ)$. The monomial sections
\begin{equation}\label{eq:e0ae1b}
e_0^{k-a}e_1^a = \frac{1}{|\symgp_k|} \sum_{\sigma\in\symgp_k}
\sigma (e_0^{\otimes k-a}\otimes e_1^{\otimes a})
\qquad
(0\leq a\leq k)
\end{equation}
form a basis of multivalued flat sections of the subsheaf $(\Sym^k\Kl_2)^\nabla_\QQ$
of $(\Kl_2^{\otimes k})^\nabla_\QQ$ and satisfy
\begin{equation}\label{eq:selfdualityQ}
\langle e_0^{k-a}e_1^a,e_0^{k-b}e_1^b\rangle_\topo=
\begin{cases}
(-1)^a\dfrac{a!\,b!}{k!}&\text{if }a+b=k,\\
0&\text{otherwise}.
\end{cases}
\end{equation}

\begin{lemma}\label{lem:coefev}
The coefficients of the flat sections $e_0^{k-a}e_1^a$ on the meromorphic basis $(u_b)_b$ of~$\Sym^k\Kl_2$ have moderate growth at the origin.
Moreover, they have moderate growth (\resp rapid decay)
in a small sector centered at infinity and containing $\RR_{>0}$
if and only if $a\!\leq\!k/2$ (\resp $a\!<\!k/2$).
\end{lemma}

\begin{proof}
Since $I_0'(t)$ (\resp $K'_0(t)$) has an asymptotic expansion similar to that of $I_0(t)$ (\resp $K_0(t)$) at the origin and at infinity in the specified domains, the first statement follows from \eqref{eq:baree} and the definition \eqref{eq:BesselIK} of $I_0$ and $K_0$. Besides, the asymptotic expansion of $I_0$ and $K_0$ at infinity~\eqref{eq:BesselIKinfty} implies the second statement by calculating the power of~$e^t$
in the products of $I_0$ and~$K_0$.
\end{proof}

\section{De Rham pairing for \texorpdfstring{$\Sym^k\Kl_2$}{SymKL}}\label{sec:DRpairing}

The main result of this section is the computation of the matrices of the de~Rham pairings
\begin{align}\label{eq:dRcduality}
\coH^1_{\dR,\rc}(\Gm,\Sym^k\Kl_2)\otimes \coH^1_{\dR}(\Gm,\Sym^k\Kl_2)&\To{\DRpairing}\CC\\
\label{eq:dRmidduality}
\coH^1_{\dR,\rmid}(\Gm,\Sym^k\Kl_2)\otimes \coH^1_{\dR,\rmid}(\Gm,\Sym^k\Kl_2)&\To{\DRpairing_\rmid}\CC
\end{align}
with respect to suitable bases, taking into account the self-duality pairing induced by \eqref{eq:selfduality}. Since the latter is $(-1)^k$-symmetric, \eqref{eq:dRmidduality} is $(-1)^{k+1}$-symmetric. We first make clear the bases in which we compute the matrices.

\subsection{Bases of the de~Rham cohomology}\label{subsec:bases}

Let
\[
\iota_{\wh0}: \Sym^k\Kl_2 \to (\Sym^k\Kl_2)_{\wh0}\quad\text{and}\quad\iota_{\wh\infty}: \Sym^k\Kl_2 \to (\Sym^k\Kl_2)_{\wh\infty}
\]
denote the formalization of $\Sym^k\Kl_2$ at zero and infinity respectively, and $\wh\nabla$ the induced connection. We can represent elements of $\coH^1_{\dR,\rc}(\Gm,\Sym^k\Kl_2)$ as pairs $(\wh m,\eta)$ as follows (\cf \cite[Cor.\,3.5]{F-S-Y20}):
\begin{itemize}
\item
$\wh m=(\wh m_0,\wh m_\infty)$ is a pair of formal germs in $(\Sym^k\Kl_2)_{\wh0}\oplus(\Sym^k\Kl_2)_{\wh\infty}$, and
\item
$\eta$ belongs to $\Gamma(\Gm,\Omega^1_{\Gm}\otimes\Sym^k\Kl_2)$,
\end{itemize}
such that, denoting by $\wh\eta=(\iota_{\wh0}\eta,\iota_{\wh\infty}\eta)$ the formal germ of $\eta$ in
\[
\bigl[\Omega^1_{\PP^1,\wh0}\otimes(\Sym^k\Kl_2)_{\wh0}\bigr]\oplus\bigl[\Omega^1_{\PP^1,\wh\infty}\otimes(\Sym^k\Kl_2)_{\wh\infty}\bigr],
\]
$\wh m$ and $\eta$ are related by $\wh\nabla\wh m=\wh\eta$.

We can regard $\coH^1_{\dR,\rmid}(\Gm,\Sym^k\Kl_2)$ as the image of the natural morphism
\[
\coH^1_{\dR,\rc}(\Gm,\Sym^k\Kl_2)\to \coH^1_{\dR}(\Gm,\Sym^k\Kl_2)
\]
sending a pair $(\wh m,\eta)$ to $\eta$. Recall
that $\coH^1_{\dR,\rmid}(\Gm,\Sym^k\Kl_2)$ has dimension~$k'$ if~$k$ is not a multiple of~$4$, and $k'-1$ otherwise (\cf \cite[Prop.\,4.12]{F-S-Y18}). According to \cite[Rem.\,3.6]{F-S-Y20}, there exists a basis of $\coH^1_{\dR,\rc}(\Gm,\Sym^k\Kl_2)$ consisting of
\begin{itemize}
\item
pairs $(\wh m_i,0)_i$ where $(\wh m_i)_i$ is a basis of $\ker\wh\nabla$
in $(\Sym^k\Kl_2)_{\wh0}\oplus(\Sym^k\Kl_2)_{\wh\infty}$,
and
\item
a set of pairs $(\wh m_j,\eta_j)_j$,
of cardinality $\dim\coH^1_{\dR,\rmid}(\Gm,\Sym^k\Kl_2)$, related as above such that $(\eta_j)_j$ are linearly independent in $\coH^1_{\dR}(\Gm,\Sym^k\Kl_2)$.
\end{itemize}
Furthermore, such a family $(\eta_j)_j$ is a basis of~$\coH^1_{\dR,\rmid}(\Gm,\Sym^k\Kl_2)$.

We set (recall that $u_0=v_0^k$)
\begin{equation}\label{eq:omegai}
\omega_i=[z^iu_0\rd z/z]\in \coH^1_{\dR}(\Gm,\Sym^k\Kl_2),\quad 0\leq i\leq k'.
\end{equation}
It was proved in \cite[Prop.\,4.14]{F-S-Y18} that
\[
\Basis_k=\{\omega_i\mid 0\leq i\leq k'\}
\]
is a basis of $\coH^1_{\dR}(\Gm,\Sym^k\Kl_2)$.
This property also follows from the combination of Lemma \ref{lem:sol0},
Proposition \ref{prop:solinfty}, and Theorem \ref{th:Smid} below (see Remark \ref{rem:detDRpairing}).

We first determine which linear combinations of elements of~$\Basis_k$ belong to $\coH^1_{\dR,\rmid}(\Gm,\Sym^k\Kl_2)$. For each $i=0,\dots,k'$, we look for the existence~of
\[
\wh m_i=(\wh m_{i,0},\wh m_{i,\infty}) \in (\Sym^k\Kl_2)_{\wh0}\oplus(\Sym^k\Kl_2)_{\wh\infty}
\]
such that
$\wh\nabla\wh m_i= (\iota_{\wh0}(\omega_i), \iota_{\wh\infty}(\omega_i))$.

\begin{lemma}[Solutions at $z=0$]\mbox{}\label{lem:sol0}
\begin{enumerate}
\item
The subspace $\ker\wh\nabla\subset (\Sym^k\Kl_2)_{\wh0}$ has dimension one and a basis is given by $\wh m_{0,0} = e_0^k$.
\item
There exists $\wh m_{i,0}$ if and only if $i\geq1$ and, in such case, there exists a unique $\wh m_{i,0}$ belonging to $z\CC\lcr z\rcr\cdot\bmu$.\end{enumerate}
In fact, for any $j\geq 1$, there exists a unique $\wh{m}_{j,0} \in z\CC\lcr z\rcr\cdot\bmu$
with $\wh\nabla\wh{m}_{j,0} = \iota_{\wh0}(z^ju_0\de z/z)$.
\end{lemma}

\begin{proof}
Set $(V,\nabla)=(\Sym^k\Kl_2,\nabla)$. We first claim that $\CC\lcr z\rcr\cdot\bmu$ is equal to $V_{\wh0,0}$ (the $0$-th step of the formal Kashiwara-Malgrange filtration at the origin, similar to that considered in the proof of~\hbox{\cite[Prop.\,3.2]{F-S-Y20}).} Indeed, it is standard to show that there exists a formal (in fact convergent) base change $P(z)=\id+zP_1+\cdots$ such that the matrix of $\wh\nabla$ in the basis $\bmu'=\bmu\cdot P(z)$ is constant
and equal to a lower standard Jordan block with eigenvalue 0.
It follows that $\bmu'$ is a $\CC\lcr z\rcr$-basis of $V_{\wh0,0}$, and the claim follows, as well as the first point of the lemma.

Let us consider the second point. Setting $V_{\wh0,-1}=zV_{\wh0,0}$, we have recalled in \loccit\ that $z\partial_z:V_{\wh0,-1}\to V_{\wh0,-1}$ is bijective. The ``if'' part and its supplement follow. It remains to check that $\iota_{\wh0}(u_0\rd z/z)$ does not belong to the image of~$\wh\nabla$. It amounts to the same to replace $u_0$ with $u'_0$ defined above,
and since $u'_0$ is the primitive vector of the matrix of $\wh\nabla$, the assertion follows.
\end{proof}

We now look for the solutions of $\wh\nabla\wh m_{i,\infty}=\iota_{\wh\infty}(\omega_i)$ for $i=1,\dots,k'$.
For this purpose, we introduce the constants $\gamma_{k,i}$
as follows.
Let us assume that $\kfour$. Recall that we have set $w=1/z$ on $\Gm$.
Write the asymptotic expansion as
\begin{equation}\label{eq:I0K0}
2^k(I_0(t)K_0(t))^{k/2} \sim w^{k/4}\sum_{j=0}^\infty \gamma_{k,k/4+j}w^j,
\end{equation}
so that we can define $\gamma_{k,i}$ by the residue
\begin{equation}\label{eq:gammaki}
\gamma_{k,i}=\res_{w=0} \frac{2^k(I_0(t)K_0(t))^{k/2}}{w^{i+1}}.
\end{equation}
We have $\gamma_{k,i} = 0$ if $i<k/4$, $\gamma_{k,k/4} = 1$, and $\gamma_{k,i} > 0$ for all $i\in k/4+\NN^*$, \eg
\begin{equation}\label{eq:valuesofgammaki}
\gamma_{k,1+k/4} = \frac{k}{2^6},
\quad
\gamma_{k,2+k/4} = \frac{k(k+52)}{2^{13}},
\quad
\gamma_{k,3+k/4} = \frac{k(k^2+156k+13184)}{2^{19}\cdot 3},
\quad\dots
\end{equation}
For what follows, it will be convenient to set
\begin{equation}\label{eq:gammakiconv}
\gamma_{k,i}=0\quad(i\in k/4+\ZZ)\text{ if $\knotfour$}.
\end{equation}

\begin{prop}[Solutions at $z=\infty$]\label{prop:solinfty}
Let us fix $i\in\{1,\dots,k'\}$.
\begin{enumerate}
\item\label{prop:solinfty1}
For $\knotfour$, the equation $\wh\nabla\wh m_{i,\infty}=\iota_{\wh\infty}(\omega_i)$ has a unique solution.
\item\label{prop:solinfty2}
For $\kfour$ and $i\neq k/4$,
the equation
\[
\wh\nabla\wh m_{i,\infty}=\iota_{\wh\infty}(\omega_i-\gamma_{k,i}\omega_{k/4})
\] has a solution (in fact a one-dimensional vector space of solutions)
where $\gamma_{k,i}$ is given by \eqref{eq:gammaki}.
Moreover, the subspace $\ker\wh\nabla\subset (\Sym^k\Kl_2)_{\wh\infty}$ is generated by the formal expansion
$\wh m_{k/4,\infty}$ of $2^k(\pii)^{k/2}(e_0e_1)^{k/2}$.
\end{enumerate}
In fact, for any $j\geq 1$,
there exists $\wh{m}_{j,\infty}$
satisfying
$\wh\nabla\wh{m}_{j,\infty}
= \iota_{\wh\infty}(z^j-\gamma_{k,j}z^{k/4})u_0\de z/z$.
\end{prop}

The second assertion of \ref{prop:solinfty}\eqref{prop:solinfty2} is easy to check from the formal structure at infinity of $\Sym^k\Kl_2$ (\cf\cite[Prop.\,4.6\,(3)]{F-S-Y18}).
We set
\begin{align*}
\wh m_0&=(\wh m_{0,0},0)
	\text{ (\ie $\wh m_{0,\infty}=0$) \quad for all $k$},\\
\wh m_{k/4}&=(0,\wh m_{k/4,\infty})
	\text{ (\ie $\wh m_{k/4,0}=0$) \quad for $\kfour$},
\end{align*}
and (\cf Notation \ref{nota:k'})
\begin{equation}\label{eq:omegaprimei}
\omega'_i=
\begin{cases}
0&\text{if $i=0$ and $i=k/4$},\\
\omega_i-\gamma_{k,i}\omega_{k/4}&\text{if $i\in\lcr1,k'\rcr$},
\end{cases}
\end{equation}
so that $\omega'_i=\omega_i$ if $\knotfour$ and $1\leq i\leq k'$, or if $\kfour$ and $1\leq i<k/4$. Once we know that $\Basis_k$ is a basis of $\coH^1_\dR(\Gm,\Sym^k\Kl_2)$, using the convention~\eqref{eq:gammakiconv} we derive the following from~\cite[Rem.\,3.6(4)]{F-S-Y20}:

\begin{cor}
The following set of $k'+1$ elements
\[
\Basis_{k,\rc}=(\wh m_i,\omega'_i)_{0\leq i\leq k'}=
\begin{cases}
\{(\wh m_0,0),(\wh m_1,\omega_1),\dots,(\wh m_{k'},\omega_{k'})
& \text{if }\knotfour,\\
\{(\wh m_0,0),(\wh m_1,\omega'_1),\dots,(\wh m_{k/4},0),\dots,(\wh m_{k'},\omega'_{k'})
& \text{if }\kfour,
\end{cases}
\]
is a basis of $\coH^1_{\dR,\rc}(\Gm,\Sym^k\Kl_2)$.
\end{cor}

Let us consider the subset $\Basis_{k,\rmid}$ of the $k'$ (\resp $(k'-1)$) dimensional subspace $\coH^1_{\dR,\rmid}(\Gm,\Sym^k\Kl_2)$ of
$\coH^1_\dR(\Gm,\Sym^k\Kl_2)$ if $\knotfour$
(\resp if $\kfour$):
\begin{equation}\label{eqn:BasisdRmid}
\Basis_{k,\rmid}=
\begin{cases}
\{\omega_i\mid i\in\lcr1,k'\rcr\}&\text{if }\knotfour,\\
\{\omega'_i\mid i\in\lcr1,k'\rcr\}&\text{if }\kfour.
\end{cases}
\end{equation}

\begin{cor}\label{cor:BasisdRmid}
The set $\Basis_{k,\rmid}$ is a basis of $\coH^1_{\dR,\rmid}(\Gm,\Sym^k\Kl_2)$.
\end{cor}

\begin{remark}
Part of this result can also be proved as a consequence of \cite[Prop.\,4.21\,(2) \& Th.\,1.8]{F-S-Y18}.
However, the present proof does not rely on Hodge theory.
\end{remark}

\begin{proof}[Proof of Proposition \ref{prop:solinfty}]
Recall that $e_0,e_1$ are defined by~\eqref{eq:baree} and $\omega_i$ by~\eqref{eq:omegai}. Let us set $\ov e_1=\pii\, e_1$. Then $v_0=2(K_0 e_0+I_0\ov e_1)$ and
\[
\omega_i = -\sum_{a=0}^k \binom{k}{a}\,\frac{2^kI_0^aK_0^{k-a}}{w^i}\,e_0^{k-a}\ov e_1^a\,\frac{\de w}{w}.
\]
We have to examine if there exist $\xi_{i,a}$ in some extension of $\CC\lpr w\rpr$
such that
\[ w\partial_w \xi_{i,a} = -w^{-i}(2^kI_0^aK_0^{k-a})
\quad\text{and}\quad
\sum_{a=0}^k \binom{k}{a}\xi_{i,a} e_0^{k-a}\ov e_1^a\in(\Sym^k\Kl_2)_{\wh\infty}\quad\text{if $\knotfour$}, \]
and a similar property in the case $\kfour$.
Then one takes
\[
\wh{m}_{i,\infty} = \sum_{a=0}^k \binom{k}{a}
\xi_{i,a}e_0^{k-a}\ov e_1^a.
\]

We write
\[
\frac{2^kI_0^aK_0^{k-a}}{w^i}
\sim \begin{cases}
\sqrt{\pi}^{k-2a}\,e^{-2(k-2a)/\sqrt{w}}w^{k/4-i}
\cdot F_a, & a \neq \sfrac{k}{2} \\[5pt]
w^{k/4-i}\biggl(\sum_{n=0}^\infty \dfrac{((2n-1)!!)^3}{2^{5n}n!}w^n\biggr)^{k/2},
& a = \sfrac{k}{2}
\end{cases}
\]
with $F_a \in 1+\sqrt{w}\,\QQ\lcr \sqrt{w}\,\rcr$.
When $a \neq \sfrac{k}{2}$
there exists a unique $\xi_{i,a}$
with the expansion
\begin{equation}\label{eq:xiia}
\xi_{i,a}
\sim \frac{\sqrt{\pi}^{k-2a}}{(k-2a)}\,e^{-2(k-2a)/\sqrt{w}}w^{k/4-i+1/2}
\cdot G_{i,a}
\end{equation}
for some $G_{i,a} \in -1+\sqrt{w}\,\QQ\lcr \sqrt{w}\,\rcr$.
Moreover,
when expressed as a combination of monomials $v_0^{k-b}v_1^b$,
such $\xi_{i,a}e_0^{k-a}\ov e_1^a$
has no exponential factor
and, if $\sigma$ denotes the action $w^{1/4}\mto\sfi\,w^{1/4}$
so that $\CC\lpr w\rpr = \CC\lpr w^{1/4}\rpr^\sigma$,
one has
$\sigma(\xi_{i,a}e_0^{k-a}\ov e_1^a) = \xi_{i,k-a}e_0^a\ov e_1^{k-a}$.
When $\ktwofour$ and $a=k/2$,
the exponents of $w$ in the expansion of $2^kw^{-i-1}I_0^aK_0^{k-a}$
are in $\sfrac{1}{2}+\ZZ$
and one takes $\xi_{i,k/2}$ satisfying
\[
\xi_{i,k/2} \sim \frac{w^{k/4-i}}{(k/4-i)} \cdot G_i
\]
with $G_i \in-1+w\QQ\lcr w\rcr$.
Then the factor $\xi_{i,k/2}(e_0\ov e_1)^{k/2}$
has no exponential part
in its expression in terms of $v_0,v_1$
and is invariant under $\sigma$.
Finally, when $\kfour, k\geq 8$ and $a=k/2$,
the residue
\[
\gamma_{k,i}=\res_{w=0} \frac{2^kI_0^{k/2}K_0^{k/2}}{w^{i+1}},
\]
vanishes if and only if $i<k/4$. Therefore, for $i\geq k/4$ there exists $\xi_{i,k/2} \in \QQ\lpr w\rpr$ such that
\[
w\partial_w\xi_{i,k/2} = -(w^{-i} - \gamma_{k,i} w^{-k/4})
(I_0K_0)^{k/2} \,\de w.
\]
In this case,
$\xi_{i,k/2}(e_0\ov e_1)^{k/2}$ has no exponential factor
in combinations of $v_0,v_1$
and is invariant under $\sigma$.
The proof is complete.
\end{proof}

\begin{remark}
In the case where $\knotfour$, the
$\wh{m}_{i,\infty}$ are uniquely determined
since there is no horizontal formal section at $z=\infty$.
When $\kfour$, the
$\wh{m}_{i,\infty}$ are unique
up to adding a multiple of the formal horizontal section
$\wh{m}_{k/4,\infty}$ defined in Proposition \ref{prop:solinfty}\eqref{prop:solinfty2}. To normalize the choice,
we take the series $\xi_{i,k/2}$ above to have no constant term.
This normalization then fixes the computations of periods below.
\end{remark}

\subsection{Computation of the de~Rham pairing}\label{subsec:DRpairing}
We aim at computing the matrix $\DRpairing_k^\rmid$ of the pairing
\[
\coH^1_{\dR,\rmid}(\Gm,\Sym^k\Kl_2)\otimes \coH^1_{\dR,\rmid}(\Gm,\Sym^k\Kl_2)\to\CC
\]
induced by the self-duality pairing \eqref{eq:selfduality}, with respect to the basis $\Basis_{k,\rmid}$. By \cite[Prop.\,3.12]{F-S-Y20}, the matrix $\DRpairing_k^\rmid$ is equal to the sum of the matrices having $(i,j)$ entries respectively
\begin{itemize}
\item
if $\knotfour$ and $i,j=1,\dots,k'$,
\[
\res_{z=0}\langle\wh m_{i,0},z^ju_0\rd z/z\rangle_\alg\quad\text{and}\quad -\res_{w=0}\langle\wh m_{i,\infty},w^{-j}u_0\rd w/w\rangle_\alg,
\]
\item
a similar formula following Proposition \ref{prop:solinfty}\eqref{prop:solinfty2} if $\kfour$.
\end{itemize}
For $i\in\{1,\dots,k'\}$ we set
\[
\wh m_{i,\infty}=\sum_{a=0}^k\mu_{a,i}(w)u_a, \qquad \mu_{a,i}(w)=\sum_{\ell\gg-\infty}\mu_{a,i,\ell}w^{\ell}.
\]
We can already note that, according to Lemma \ref{lem:sol0}, $\langle\wh m_{i,0},z^ju_0\rd z/z\rangle_\alg$ has no residue (and a similar assertion if $\kfour$) so $\DRpairing_k^\rmid$ is determined by the residues at infinity. It follows from~\eqref{eq:selfduality} that, if $\knotfour$, we have
\begin{equation}\label{eq:Sprimekmid}
\DRpairing_{k;i,j}^\rmid=(-1)^{k+1}\mu_{k,i,j},\quad i,j=1,\dots,k'.
\end{equation}
If $\kfour$, $\DRpairing_k^\rmid$ is the $(k'-1)\times(k'-1)$-matrix given by the formula
\begin{equation}\label{eq:Sprimekmidfour}
\DRpairing_{k;i,j}^\rmid=(-1)^{k+1}
\begin{cases}
\mu_{k,i,j}&\text{if }i\text{ or }j<k/4,\\
(\mu_{k,i,j}-\gamma_{k,j}\mu_{k,i,k/4})&\text{if }i\text{ and }j>k/4,
\end{cases}
\ i,j\in\lcr1,k'\rcr.
\end{equation}
In other words, in the matrix given by \eqref{eq:Sprimekmid} we delete the row $i=k/4$ and the column $j=k/4$ and we add to it the matrix having entry $(i,j)$ equal to
$(-1)^k\gamma_{k,j}\mu_{k,i,k/4}$
for $i,j>k/4$. According to~\hbox{\cite[Cor.\,3.14]{F-S-Y20}} and \eqref{eq:selfduality}, the matrix $\DRpairing_k^\rmid$ is $(-1)^{k+1}$-symmetric.

\begin{thm}\label{th:Smid}
The matrix $\DRpairing_k^\rmid$
is lower-right triangular
(\ie the entries $(i,j)$ are zero if $i+j\leq k'$)
and the anti-diagonal entry on the $i$-th row is equal to
\[
\begin{cases}
(-2)^{k'}\,\dfrac{k'!}{k!!}&\text{if $k$ is odd},\\[3pt]
\dfrac{(-1)^{k'+1}}{2^{k'}(k'+1-2i)}\cdot\dfrac{(k-1)!!}{(k'+1)!}&\text{if $k$ is even}.
\end{cases}
\]
\end{thm}

\begin{proof}
We keep notation from the proof of Proposition \ref{prop:solinfty}. For odd $k$, the entries of $\DRpairing_k^\rmid$ are
\begin{align*}
\res_{w=0}\langle\wh{m}_{i,\infty},\omega_j\rangle_\alg
&= -\sum_{a,b=0}^k \binom{k}{a}\hspace{-3pt}\binom{k}{b}\res_{w=0}
\Bigl\langle\xi_{i,a} e_0^{k-a}\ov e_1^a,
\frac{2^kI_0^bK_0^{k-b}}{w^{j+1}}\,
e_0^{k-b}\ov e_1^b\,\de w\Bigr\rangle_\alg \\
&=-\frac{1}{2^k} \sum_{a,b=0}^k \binom{k}{a}\hspace{-3pt}\binom{k}{b}\res_{w=0}
\Bigl\langle\xi_{i,a} e_0^{k-a} e_1^a,
\frac{2^kI_0^bK_0^{k-b}}{w^{j+1}}\,
e_0^{k-b} e_1^b\,\de w\Bigr\rangle_\topo \\
&= -\frac{1}{2^k} \sum_{a=0}^k (-1)^a\binom{k}{a}
\res_{w=0} \Bigl(\xi_{i,a} \frac{2^kI_0^{k-a}K_0^a}{w^{j+1}} \,\de w \Bigr) \\
&= -\frac{1}{2^k} \sum_{a=0}^k \frac{(-1)^{a}\binom{k}{a}}{k-2a}
\res_{w=0} \bigl( w^{k'-i-j}F_{k-a}G_{i,a} \,\de w \bigr)
\end{align*}
where $F_{k-a}, -G_{i,a} \in 1+\sqrt{w}\,\QQ\lcr \sqrt{w}\,\rcr$.
Clearly, the last residue vanishes if $i+j\!\leq\! k'$.
If $i+j\!=\!k'+1$,
we find:
\[
\res_{w=0}\gen{\wh{m}_{i,\infty},\omega_j}
= 2^{-k} \sum_{a=0}^k \frac{(-1)^a\binom{k}{a}}{k-2a}.
\]
We conclude the case where $k$ is odd with the next lemma.

\begin{lemma}
For any $k\geq1$, we have
\[
\sum_{\substack{0\leq a\leq k\\a\neq k/2}} \frac{(-1)^a\binom{k}{a}}{k-2a}=
\begin{cases}
2^k(-2)^{k'}\dfrac{k'!}{k!!}&\text{if $k$ is odd},\\
0&\text{if $k$ is even}.
\end{cases}
\]
\end{lemma}

\begin{proof}
If $k$ is even, replacing $a$ with $k-a$ in the sum shows that the sum is equal to its opposite, and hence vanishes. We thus assume that $k$ is odd and set
\[
f_k(x) = \sum_{a=0}^k \binom{k}{a} \frac{x^{k-2a}}{k-2a}.
\]
Then
$f_k(x) = \int_1^x \big(x+\sfrac{1}{x})^k\sfrac{\de x}{x}$.
Besides, one has
\begin{align*}
f_k(x)
&= \int_0^{\log x} (e^t + e^{-t})^k \,\de t= 2^k \int_0^{\log x} \cosh^k t \,\de t
	\qquad (x = e^t) \\
&= 2^k \biggl[ \frac{\sinh t \cosh^{k-1} t}{k}\bigg|_{t=\log x}
	+ \frac{k-1}{k} \int_0^{\log x} \cosh^{k-2} t \,\de t \biggr] \\
&= \frac{1}{k}\Big(x-\frac{1}{x}\Big)\Big(x+\frac{1}{x}\Big)^{k-1}
	+ 4\,\frac{k-1}{k} f_{k-2}(x).
\end{align*}
By evaluating $f_k(\sfi)$ inductively, one obtains the desired equality.\end{proof}

Assume $\ktwofour$, so that $k/2=k'+1$ is odd. Then the entries of $\DRpairing_k^\rmid$ are
\begin{align*}
\res_{w=0}\langle\wh{m}_{i,\infty},\omega_j\rangle
&= -\frac{1}{2^k} \biggl[\sum_{a\neq k/2} (-1)^a\binom{k}{a}
\res_{w=0} \Bigl(\xi_{i,a} \frac{2^kI_0^{k-a}K_0^a}{w^{j+1}} \,\de w \Bigr)\\[-5pt]
&\hspace{2.7cm} -\binom{k}{k/2}
\res_{w=0} \Bigl(\xi_{i,k/2} \frac{2^k(I_0K_0)^{k/2}}{w^{j+1}} \,\de w \Bigr) \biggr] \\[5pt]
&= -\frac{1}{2^k} \biggl[\sum_{a\neq k/2} \frac{(-1)^a\binom{k}{a}}{(k-2a)}
\res_{w=0} \Bigl( w^{k'-i-j+1/2}F_{k-a}G_{i,a} \,\de w \Bigr) \\[-5pt]
&\hspace{2.7cm} -\frac{2\binom{k}{k/2}}{(k'+1-2i)}
\res_{w=0} \Bigl( w^{k'-i-j}FG_i \,\de w \Bigr) \biggr],
\end{align*}
where
\[
F(w) = \biggl(\sum_{n=0}^\infty \frac{((2n-1)!!)^3}{2^{5n}n!}\,w^n\biggr)^{k/2}.
\]
Again, the residue is zero if $i+j\leq k'$. On the other hand, if $i+j=k'+1$, the first term in the above expression is zero,
according to the lemma above,
and since $F(0)=1$ and $G_i(0)=-1$, we have
\[
\res_{w=0}\langle\wh{m}_{i,\infty},\omega_j\rangle
= -\binom{k}{k/2}[2^{k-1}(k'+1-2i)]^{-1}.
\]

The computation in the case where $\kfour$ is similar.
\end{proof}

\begin{example}[$k=5$]
We have $k'=2$ and
\[
\arraycolsep3.5pt
\DRpairing_5^{\rmid}=\begin{pmatrix}
0&8/15\\[3pt]
8/15&\mu_{5,2,2}
\end{pmatrix}.
\]
From the proof of the theorem above, one has
\[ \mu_{5,2,2} = \frac{-1}{2^5}\sum_{a=0}^5 \frac{(-1)^{a}\binom{5}{a}}{5-2a}
\res_{w=0} \Bigl( F_{5-a}G_{2,a} \,\frac{\de w}{w^2} \Bigr). \]
A direct computation yields
\[ -\sum_{a=0}^5 \frac{(-1)^{a}\binom{5}{a}}{5-2a}F_{5-a}G_{2,a}
\sim \frac{256}{15} + \frac{2^9\cdot 13}{3^3\cdot 5^5}w
\mod{w^2\QQ\lcr w\,\rcr}. \]
Therefore,
\[\arraycolsep3.5pt
\DRpairing_5^{\rmid}=\begin{pmatrix}
0&8/15\\[3pt]
8/15&2^4\cdot13/3^3\cdot5^3
\end{pmatrix}.
\]
\end{example}

\begin{example}[$k=6$]
We have $k'=2$ and $i=1,2$. Due to skew-symmetry, the anti-diagonal entries are enough to determine $\DRpairing_6^\rmid$. Theorem \ref{th:Smid} gives
\[
\DRpairing_6^{\rmid}=\begin{pmatrix}0&-5/8\\5/8&0\end{pmatrix}.
\]
\end{example}

\begin{cor}\label{cor:compatibilityHodge}
The ordered basis $\Basis_{k,\rmid}$
of $\coH^1_{\dR,\rmid}(\Gm,\Sym^k\Kl_2)$ is adapted to the Hodge filtration.\end{cor}

\begin{proof}
For odd $k$,
\cite[Prop.\,4.21(2) \& Th.\,1.8(1)]{F-S-Y18} imply the claim.
For even $k$, \loccit\ only gives compatibility for half of the basis $\Basis_{k,\rmid}$.
However, since the Poincaré pairing respects the Hodge filtration,
compatibility holds for the whole $\Basis_{k,\rmid}$
by the above theorem.
\end{proof}

\begin{remark}[The matrix $\DRpairing_k$]\label{rem:detDRpairing}
The matrix of the pairing
\[
\DRpairing_k:\coH^1_{\dR,\rc}(\Gm,\Sym^k\Kl_2)\otimes\coH^1_{\dR}(\Gm,\Sym^k\Kl_2)\to\CC
\]
in the basis $\Basis_{k,\rc}\otimes\Basis_k$ is obtained with similar residue formulas as for $\DRpairing_k^\rmid$, due to \cite[Prop.\,3.12]{F-S-Y20}.
\begin{enumerate}
\item
If $\knotfour$,
\[
\DRpairing_{k;i,j}=
\begin{cases}
\DRpairing_{k;i,j}^\rmid&\text{if }1\leq i,j\leq k',\\
0&\text{if }i=0\text{ and }j=1,\dots,k'.
\end{cases}
\]
On the other hand,
we have
$\DRpairing_{k;0,0}
=\res_{z=0}\langle e_0^k,u_0\de z/z\rangle_\alg =1$.
Last, for $i\geq1$,
we use the expressions \eqref{eq:baree}
and obtain
\begin{align*}
\DRpairing_{k;i,0} &= \res_{w=0}\Bgen{\wh{m}_{i,\infty},u_0\frac{\de z}{z}}_\alg \\[-5pt]
&= \sum_{a=0}^k\binom{k}{a}\res_{w=0}\Bgen{\xi_{i,a}e_0^{k-a}\ol{e}_1^a,u_0\frac{\de z}{z}}_\alg \\
&= \sum_{a=0}^k\binom{k}{a}\res_{w=0}\Bgen{
\xi_{i,a}\Big[\frac{t}{2}I_0'v_0-I_0v_1\Big]^{k-a}\Big[-\frac{t}{2}K_0'v_0+K_0v_1\Big]^a,
u_0\frac{\de z}{z}}_\alg \\
&= \sum_{a=0}^k(-1)^a\binom{k}{a}\res_{w=0}\xi_{i,a}I_0^{k-a}K_0^a\frac{\de z}{z}.
\end{align*}
Since
\[ \xi_{i,a}I_0^{k-a}K_0^a \sim
\begin{cases}
\frac{1}{2^k(k-2a)}w^{(k+1)/2-i}G_{i,a}, & a\neq k/2, \\
\frac{1}{2^k(k/4-i)}w^{k/2-i}G_i, & a = k/2
\end{cases} \]
with $G_{i,a},G_i \in \QQ\lcr \sqrt{w}\,\rcr$,
we conclude that $\DRpairing_{k;i,0}=0$.
In other words, $\DRpairing_k$ takes the form
\[
\begin{pmatrix}1&0\\0&\DRpairing_k^\rmid\end{pmatrix}.
\]
In particular, we get
\[ \det\DRpairing_k=\det\DRpairing_k^\rmid
=
\begin{cases}
(-1)^{k'(k'+1)/2}\bigl(2^{k'}k'!/k!!)^{k'}&\text{if $k$ is odd},\\
\bigl((k-1)!!\bigr)^{k'}\big/\bigl[\bigl(2^{k'}(k'+1)!\bigr)^{k'}\bigl((k'-1)!!\bigr)^2\bigr]&\text{if }\ktwofour.
\end{cases} \]
\item
If $\kfour$, the matrix $\DRpairing_k$ is obtained from the matrix $\begin{smallpmatrix}1&0\\0&\DRpairing_k^\rmid\end{smallpmatrix}$ by adding a row $i=k/4$ and a column $j=k/4$ that we compute now. For the row $i=k/4$, we note that
\[
\bgen{\wh m_{k/4,\infty},v_0^k}_\alg=2^k(I_0K_0)^{k/2}
=w^{k/4}\sum_{j\geq0}\gamma_{k,k/4+j}w^j,
\]
so that
\[
\DRpairing_{k;k/4,k/4+j}=
\begin{cases}
0&\text{if $j<0$},\\
-\gamma_{k,k/4+j}&\text{if $j\geq 0$}.
\end{cases}
\]
For the column $j=k/4$, we compute as above that $\DRpairing_{k;0,k/4}=0$ and, for $i\geq1$ and $\neq k/4$,
\[
\DRpairing_{k;i,k/4}=\res_{z=\infty}\gen{\wh m_{i,\infty},\omega_{k/4}}_\alg.
\]
In particular,
by setting $k'' = \flr{\psfrac{k-1}{4}}$
and recalling $\gamma_{k,k/4}=1$, we get
\begin{equation}\label{eqn:detSkmultiple4}
\det\DRpairing_k
=-\det\DRpairing_k^{\rmid}
=-\bigl[2^{-k'-1}(k-1)!!\,((k'+1)!)^{-1}\bigr]^{k'-1}(k''!)^{-2}.
\end{equation}
\end{enumerate}
\end{remark}

\section{Betti intersection pairing for \texorpdfstring{$\Sym^k\Kl_2$}{SymKL}}\label{sec:Bpairing}

In this section,
we exhibit natural bases of the rapid decay,
the moderate growth,
and the middle homology spaces
denoted respectively by
\[ 
\coH_1^\rrd(\Gm,\Sym^k\Kl_2),
\quad
\coH_1^\rmod(\Gm,\Sym^k\Kl_2),
\quad\text{and}\quad
\coH_1^\rmid(\Gm,\Sym^k\Kl_2). \]
We then compute the Betti pairing $\Bpairing_k$ as introduced in \cite[\S2.d]{F-S-Y20}. Bear in mind that the~notation there keeps track of the degree of the homology spaces; as this degree is always equal to $1$ here, we omit~it,
but we remember
the exponent $k$ of the symmetric power.
We keep the setting of Section~\ref{subsec:QKl2}. While we used $\gen{\cbbullet,\cbbullet}_\alg$ on $\Sym^k\Kl_2$ to compute the de~Rham intersection matrix~$\DRpairing_k$, we will use the topological pairing $\gen{\cbbullet,\cbbullet}_\topo$ on $\Sym^k\Kl_2^\nabla$ to compute the Betti pairing $\Bpairing_k$, which is thus defined over~$\QQ$.

We consider the following $C^\infty$ chains on $\PP^1_z$ diffeomorphic to their images:
\begin{align*}
\RR_+ &= \text{$[0,\infty]$, oriented from $0$ to $+\infty$,} \\
c_0 &= \text{unit circle, starting at $1$ and oriented counterclockwise,} \\
c_+ &= \text{$[1,\infty]$, oriented from $+1$ to $+\infty$.}
\end{align*}

According to Lemma \ref{lem:coefev}, the $\flr{\sfrac{k}{2}}+1$ twisted chains
\[
\beta_j = \RR_+ \otimes e_0^je_1^{k-j},
\quad
0\leq j\leq \flr{\sfrac{k}{2}},
\]
have moderate growth and define $\flr{\sfrac{k}{2}}+1$ elements of $\coH^{\rmod}_1(\Gm, \Sym^k\Kl_2)$, still denoted by~$\beta_j$.

Besides, the twisted chain $\alpha_0=c_0\otimes e_0^k$ is a twisted cycle with compact support, and hence has rapid decay, since $e_0$ is invariant by monodromy. We obtain other twisted chains with compact support as follows. For each integer $n \geq 1$, let us set
\[
C_n(a)= \frac{(-1)^{a-1}}{na}\,\binom{n}{a},
\quad
1\leq a\leq n.
\]

\begin{lemma}\label{lemma:relation:c_n}
Let $n \geq 1$ be an integer.
\begin{enumerate}
\item\label{lemma:relation:c_n1}
The sequence $(C_n(a))_a$ is the unique solution to the linear relations
\[
\sum_{a=1}^n C_n(a)\, a = \frac{1}{n}
\quad\text{and}\quad
\sum_{a=1}^n C_n(a)\, a^r = 0
\quad
\text{if $2\leq r\leq n$}.
\]
\item\label{lemma:relation:c_n2}
The sequence $(C_n(a))_a$ is the unique solution to the linear relations
\[
\sum_{a=1}^n C_n(a)\, a^{n+1} = (-1)^{n-1}(n-1)!
\quad\text{and}\quad
\sum_{a=1}^n C_n(a)\, a^r = 0
\quad
\text{if $2\leq r\leq n$}.
\]
\end{enumerate}
\end{lemma}

\begin{proof}
Direct simplification of Cramer's rule in solving the systems of linear equations.
\end{proof}

\begin{lemma}\label{lemma:c_tau_bernoulli}
For integers $n\geq1$ and $r\geq0$,
one has
\[
\sum_{a=1}^{n+r} C_{n+r}(a)\sum_{b=1}^ab^n =
\begin{cases}
\dfrac{(-1)^n}{n+r}\Bern{n}&\text{if $r\geq1$},\\[8pt]
\dfrac{(-1)^n}n\Bern{n}+(-1)^{n-1}\dfrac{(n-1)!}{n+1}&\text{if $r=0$}.
\end{cases}
\]
\end{lemma}

\begin{proof}
Replacing $\sum_{b=1}^ab^n$ with Bernoulli's formula
\begin{equation}\label{eq:BernoulliFormula}
\sum_{b=1}^ab^n = \frac{1}{n+1}\sum_{\ell=0}^n (-1)^\ell\binom{n+1}{\ell}\,\Bern{\ell}\, a^{n+1-\ell}, 
\end{equation}
we can rewrite the left-hand side as
\[
\sum_{a=1}^{n+r} C_{n+r}(a)\sum_{b=1}^ab^n
=\frac1{n+1}\sum_{\ell=0}^n(-1)^\ell\binom{n+1}{\ell}
\,\Bern{\ell}\sum_{a=1}^{n+r} C_{n+r}(a)\,a^{n+1-\ell}. 
\]
Then Lemma \ref{lemma:relation:c_n}\eqref{lemma:relation:c_n1} gives the assertion for $r\geq1$ and, for $r=0$, we use Lemma \ref{lemma:relation:c_n}\eqref{lemma:relation:c_n2} instead.
\end{proof}

Since $T^ae_1=e_1+ae_0$ for all $a\geq1$, it follows from Lemma \ref{lemma:relation:c_n}\eqref{lemma:relation:c_n1} that, for each $1\leq i\leq k$, the twisted chain with compact support
\[
\sum_{a=1}^{k-i+1} C_{k-i+1}(a)\,c_0^a \otimes e_0^{i-1}e_1^{k-i+1}
\]
has boundary $\{1\}\otimes e_0^ie_1^{k-i}$.
As a consequence, the following $(k'+1)$ twisted chains are twisted cycles, whose classes are elements in $\coH^\rrd_1(\Gm, \Sym^k\Kl_2)$:
\begin{equation}\label{eq:alphai}
\left\{
\begin{aligned}
\alpha_0 &= c_0 \otimes e_0^k \\
\alpha_i &= -\dfrac{(-1)^{k-i}(k-i)!}{k-i+2}\,\alpha_0
+c_+ \otimes e_0^ie_1^{k-i}
+\sum_{a=1}^{k-i+1}\hspace*{-1mm} C_{k-i+1}(a)
c_0^a \otimes e_0^{i-1}e_1^{k-i+1}
\quad
(1\leq i\leq k').
\end{aligned}
\right.
\end{equation}
The natural map
\[
\coH^\rrd_1(\Gm, \Sym^k\Kl_2) \to \coH^{\rmod}_1(\Gm, \Sym^k\Kl_2)
\]
is given on these twisted cycles by
\begin{equation}\label{eqn:imagesinmiddle}
\alpha_i \mto \begin{cases}
0 & \text{if $i = 0$}, \\
\beta_i & \text{if $1\leq i\leq k'$}
\end{cases}
\end{equation}
by shrinking the circle $c_0$ and extending the half-line $c_+$.

\begin{prop}\label{prop:Bk}
If $0\leq j\leq\flr{k/2}$, we have $\Bpgen{\alpha_0,\beta_j}=-\delta_{0,j}$ and, if $1\leq i\leq k'$,
\[
\Bpgen{\alpha_i,\beta_j} =
(-1)^{k-i}\,\dfrac{(k-i)!(k-j)!}{k!}\, \dfrac{\Bern{k-i-j+1}}{(k-i-j+1)!}.
\]
\end{prop}

With a small abuse of notation, we also denote by $\Bpairing_k$ the matrix $(\Bpairing_{k;i,j})_{1\leq i,j,\leq k'}$ with
\begin{equation*}%\label{eq:Bij}
\Bpairing_{k;i,j}=\Bpgen{\alpha_i,\beta_j}.
\end{equation*}

\begin{proof}
For the first assertion, since the intersection index $(c_0,\RR_+)$ is equal to $-1$, we have by~\eqref{eq:selfdualityQ}:
\[
\Bpgen{\alpha_0,\beta_j}=\begin{cases}
-\langle e_0^k,e_0^je_1^{k-j}\rangle_\topo=0&\text{if }j\neq0,\\
-\langle e_0^k,e_1^k\rangle_\topo=-1&\text{if }j=0.
\end{cases}
\]
Let us compute $\Bpgen{\alpha_i,\beta_j}$ if $1\!\leq\!i\!\leq\!k'$ and $0\!\leq\!j\!\leq\!\flr{k/2}$. Fix some $0<\theta_o<\pi$ and let $x_o = \exp(\sfi\theta_o)$. To~achieve the computation, we move the ray $c_+$ by adding the scalar $(x_o-1)$ and let the circle $c_0$ start at~$x_o$.
Then the component $c_0^a\otimes e_0^{i-1}e_1^{k-i+1}$ in the deformed $\alpha_i$ meets~$\beta_j$ physically $a$ times
at the same point $1 \in \CC^\times$ with intersection index $-1$.
At the $b$-th intersection ($1\leq b\leq a$), the factor $e_0^{i-1}e_1^{k-i+1}$ becomes $e_0^{i-1}(e_1+be_0)^{k-i+1}$
and \eqref{eq:selfdualityQ} gives
\[
\bgen{e_0^{i-1}(e_1+be_0)^{k-i+1},e_0^je_1^{k-j}}_\topo
=(-1)^j\,\binom{k-j+1}{j}\binom{k}{j}^{-1} b^{k-i-j+1}.
\]
For $j\geq1$, since $\Bpairing_k(\alpha_0,\beta_j)=0$, we obtain, by adding these contributions and taking into account the intersection indices,
\[
\Bpgen{\alpha_i,\beta_j} =
(-1)^{j+1}\,\binom{k-i+1}{j}\binom{k}{j}^{-1}
\sum_{a=1}^{k-i+1} C_{k-i+1}(a)\sum_{b=1}^ab^{k-i-j+1}.
\]
The asserted equality follows by applying Lemma \ref{lemma:c_tau_bernoulli} with $r\geq1$.

If $j=0$, then $\Bpgen{\alpha_i,\beta_j}$ writes
\[
\dfrac{(-1)^{k-i}(k-i)!}{k-i+2}
-\sum_{a=1}^{k-i+1} C_{k-i+1}(a)\sum_{b=1}^ab^{k-i+1}
=(-1)^{k-i}\frac{\Bern{k-i+1}}{k-i+1},
\]
after Lemma \ref{lemma:c_tau_bernoulli} with $r=0$.
\end{proof}

\begin{thm}\mbox{}\label{th:Bettibasis}
\begin{enumerate}
\item\label{th:Bettibasis1}
The family $(\alpha_i)_{0\leq i\leq k'}$ is a basis of~$\coH_1^\rrd(\Gm,\Sym^k\Kl_2)$.
\item\label{th:Bettibasis2}
The family $(\beta_j)$ with $0\leq j\leq k'$ (\resp with $0\leq j\leq k'+1$ and $j\neq1$) is a basis of~$\coH_1^\rmod(\Gm,\Sym^k\Kl_2)$ if $\knotfour$ (\resp if $\kfour$).
\item\label{th:Bettibasis3}
The family $(\beta_j)$ with $1\!\leq\! j\!\leq\! k'$ (\resp with $2\!\leq \!j\!\leq\! k'$) is a basis of $\coH_1^\rmid(\Gm,\Sym^k\Kl_2)$ if~$\knotfour$ (\resp if $\kfour$).
\end{enumerate}
\end{thm}

\begin{notation}\label{nota:Bpairingmid}
If $\kfour$, we shift the indices of the bases $\alpha$ and $\beta$ as follows:
\[
\text{for $i\in\lcr1,k'\rcr$,}\quad\alpha'_i=\begin{cases}
\alpha_{i+1}&\text{if }i<k/4,\\
\alpha_i&\text{if }i>k/4,
\end{cases}
\]
and similarly for $\beta'$. We set $\Bpairing_k^\rmid=(\Bpairing_{k;i,j}^\rmid)_{i,j\in\lcr1,k'\rcr}$ with
\[
\Bpairing_{k;i,j}^\rmid=\begin{cases}
\Bpgen{\alpha_i,\beta_{j}}&\text{if }\knotfour,\\
\Bpgen{\alpha'_i,\beta'_j}&\text{if }\kfour.
\end{cases}
\]
In particular, $\Bpairing_k^\rmid=\Bpairing_k$ if $\knotfour$. Theorem \ref{th:Bettibasis}\eqref{th:Bettibasis3} implies that $\Bpairing_k^\rmid$ is an invertible matrix.
\end{notation}

Theorem \ref{th:Bettibasis} is a straightforward consequence of Propositions \ref{prop:Bk} and \ref{prop:Bernoullidet} below.

\begin{prop}\label{prop:Bernoullidet}\mbox{}
\begin{enumerate}
\item
Let $B_k$ denote the matrix of size $k'$ having entries $\Bpgen{\alpha_i,\beta_j}$ with
\begin{itemize}
\item\label{prop:Bernoullidet1}
$1\leq i,j\leq k'$ if $\knotfour$,
\item
$1\leq i\leq k'$ and $2\leq j\leq k/2$ if $\kfour$.
\end{itemize}
Then
\[
\det B_k=
\begin{cases}
\dpl\biggl[k!\prod_{a=1}^{k'}\binom ka\biggr]^{-1}&\text{if $k$ is odd},\\
\dpl\biggl[k!(k'+1)\prod_{a=1}^{k'}\binom ka\biggr]^{-1}&\text{if $\ktwofour$},\\
\dpl\biggl[k!\prod_{a=2}^{k'}\binom ka\biggr]^{-1}&\text{if $\kfour$}.
\end{cases}
\]

\item\label{prop:Bernoullidet2}
If $\kfour$, let $B'_k$ denote the matrix of size $k'-1$ having entries $\Bpgen{\alpha_i,\beta_j}$ with $2\leq i,j\leq k'$. Then
\[
\det B'_k=\biggl[\frac k4(k'!)^2\prod_{a=2}^{k'}\binom ka\biggr]^{-1}.
\]
\end{enumerate}
(Note that $\Bpairing_k^\rmid=B_k$ if $\knotfour$ and $\Bpairing_k^\rmid=B_k'$ if $\kfour$).
\end{prop}

\subsubsection*{Some determinants of Bernoulli numbers} We will make use of the following lemma:

\begin{lemma}\label{lem:Bernoullidet}
The following identities hold:
\begin{align}
\label{eq:bern_det_1}
\det \begin{pmatrix}
\dfrac{\Bern{i+j}}{(i+j)!}
\end{pmatrix}_{1\leq i,j\leq n}
&= \frac{(-1)^{n(n-1)/2}(2n+1)!!}{2^{n(n+1)}[3!!\,5!!\cdots (2n+1)!!]^2},\\
\label{eq:bern_det_2}
\qquad\det \begin{pmatrix}
\dfrac{\Bern{i+j+1}}{(i+j+1)!}
\end{pmatrix}_{1\leq i,j\leq n}
&= \begin{cases}
\dfrac{(-1)^{n/2}}{2^{n(n+2)}[3!!\,5!!\cdots (2n+1)!!]^2}
& \text{if $n$ is even}, \\
0 & \text{if $n$ is odd}.
\end{cases}
\end{align}
\end{lemma}

\begin{proof}
We follow the principle in \cite[\S 5.4]{Krattenthaler05}. Recall that, for each $n \geq 0$, the Lommel polynomials $h_{n,\nu}(x)$ satisfy
\[
J_{\nu+n}(z) =
h_{n,\nu}(z^{-1})J_\nu(z) - h_{n-1,\nu+1}(z^{-1})J_{\nu-1}(z), 
\]
where $J_\nu(z)$ is the Bessel function of the first kind of order $\nu$.
Let
\[
f_n(x) = \frac{1}{(2n+1)!!}\, h_{n,1/2}(x) \in \QQ[x]
\quad
(n\geq 0).
\]
The following hold:
\begin{itemize}
\item
$f_n(x)$ is monic of degree $n$
with $f_0(x) = 1, f_1(x) = x$,
and $\{ f_n(x) \}$ satisfies the recursive relation
\[
f_{n+1}(x) = xf_n(x) - b_nf_{n-1}(x),
\quad
b_n = \frac{1}{(2n+1)(2n+3)},
\quad
n \geq1.
\]
\item
$\{ f_n(x) \}$ forms an orthogonal family with respect to the linear functional
\[
L(f(x)) = \sum_{m=1}^\infty \frac{1}{(m\pi)^2}
\left( f(\sfrac{1}{m\pi}) + f(\sfrac{-1}{m\pi}) \right).
\]
\end{itemize}
For the above, see \cite[\S 1]{K-V95} for references.
We have the moments $\mu_r = L(x^r)$ of the linear functional
\[
\mu_r = \frac{1}{\pi^{r+2}}\sum_{m=1}^\infty
\frac{1}{m^{r+2}} + \frac{1}{(-m)^{r+2}}
= \frac{(-1)^{r/2}2^{r+2}\Bern{r+2}}{(r+2)!} \geq 0,
\quad
r\geq 0
\]
(in particular, $\mu_0 = \sfrac{1}{3}$).
With these data and by applying \cite[Th.\,29]{Krattenthaler05}, one readily obtains
\begin{align*}
\det (\mu_{i+j})_{0\leq i,j\leq n-1}
&= \frac{(2n+1)!!}{[3!!\,5!!\cdots (2n+1)!!]^2}, \\
\det (\mu_{i+j+1})_{0\leq i,j\leq n-1}
&= \begin{cases}
\dfrac{(-1)^{n/2}}{[3!!\,5!!\cdots (2n+1)!!]^2} & \text{if $n$ is even}, \\
0 &\text{if $n$ is odd}.
\end{cases}
\end{align*}
The asserted formulas follow immediately by a simple matrix manipulation.
\end{proof}

\begin{remark}
Let
\[
\Theta_n = \begin{pmatrix}
\dfrac{\Bern{2a+2b-2}}{(2a+2b-2)!}
\end{pmatrix}_{1\leq a,b\leq n},
\quad
\Theta'_n = \begin{pmatrix}
\dfrac{\Bern{2a+2b}}{(2a+2b)!}
\end{pmatrix}_{1\leq a,b\leq n}.
\]
By rearranging columns and rows,
one has
\begin{align*}
\begin{pmatrix}
\dfrac{\Bern{i+j}}{(i+j)!}
\end{pmatrix}_{1\leq i,j\leq n}
&\sim \begin{cases}
\Theta_{\sfrac{n}{2}}\oplus \Theta'_{\sfrac{n}{2}} & \text{if $n$ is even},\\
\Theta_{\sfrac{n+1}{2}} \oplus \Theta'_{\psfrac{n-1}{2}} & \text{if $n$ is odd},
\end{cases} \\
\begin{pmatrix}
\dfrac{\Bern{i+j+1}}{(i+j+1)!}
\end{pmatrix}_{1\leq i,j\leq n}
&\sim \begin{cases}
\begin{pmatrix}
0 & \Theta'_{\sfrac{n}{2}} \\
\Theta'_{\sfrac{n}{2}} & 0
\end{pmatrix} & \text{if $n$ is even},\\
\text{singular} & \text{if $n$ is odd}.
\end{cases}
\end{align*}
Therefore, \eqref{eq:bern_det_1} inductively implies the equalities
\[
\det\Theta_n = \frac{1}{2^{2n^2}3!!\,5!!\cdots (4n-1)!!},
\quad
\det\Theta'_n = \frac{(-1)^n}{2^{2n(n+1)}3!!\,5!!\cdots (4n+1)!!}
\]
and \eqref{eq:bern_det_2} is a consequence of \eqref{eq:bern_det_1}.
The evaluation of a variant of $\det\Theta_n$ is also considered in~\hbox{\cite[Cor.\,2]{Z-C14}}, although their formula does not seem to be correct. Both this approach and that of \loccit use the orthogonal family of Lommel polynomials.
\end{remark}

\begin{proof}[Proof of Proposition \ref{prop:Bernoullidet}] Let $\Delta_n = \begin{pmatrix} \delta_{n+1}(i+j) \end{pmatrix}_{1\leq i,j\leq n}$.
For integers $m,n$ with $m\geq n\geq 0$,
let
\[
D^\pm_{m,n} =
\diag((\pm 1)^{m}m!, (\pm 1)^{m-1}(m-1)!, \dots, (\pm 1)^nn!).
\]

\eqref{prop:Bernoullidet1}
For $k$ odd, we have
\[
B_k = \frac{1}{k!}D^-_{k-1,k'+1}\Delta_{k'}\begin{pmatrix}
\dfrac{\Bern{i+j}}{(i+j)!}
\end{pmatrix}_{1\leq i,j\leq k'}\Delta_{k'}D^+_{k-1,k'+1}
\]
and by \eqref{eq:bern_det_1}, one obtains
\[
\det B_k = \frac{k!!}{2^{k'(k'+1)}(k!)^{k'}}
\prod_{a=1}^{k'} \left[ \frac{(k'+a)!}{(2a+1)!!} \right]^2
= \biggl[ k! \prod_{a=1}^{k'} \binom{k}{a} \biggr]^{-1}.
\]
For $\ktwofour$, we have
\[
B_k = \frac{1}{k!}D^-_{k-1,k'+2}\Delta_{k'}\begin{pmatrix}
\dfrac{\Bern{i+j+1}}{(i+j+1)!}
\end{pmatrix}_{1\leq i,j\leq k'}\Delta_{k'}D^+_{k-1,k'+2}
\]
and
\[
\det B_k = \frac{1}{2^{k'(k'+2)}(k!)^{k'}}
\prod_{a=1}^{k'} \left[ \frac{(k'+1+a)!}{(2a+1)!!} \right]^2
= \biggl[ k!\,(k'+1)\prod_{a=2}^{k'}\binom{k}{a} \biggr]^{-1}
\]
by \eqref{eq:bern_det_2}.

For $\kfour$, we have
\[
B_k = \frac{1}{k!}D^-_{k-2,k'+1}\Delta_{k'}\begin{pmatrix}
\dfrac{\Bern{i+j}}{(i+j)!}
\end{pmatrix}_{1\leq i,j\leq k'}\Delta_{k'}D^+_{k-1,k'+2}
\]
and
\[
\det B_k = \frac{(k-1)!(k-1)!!}{2^{k'(k'+1)}(k!)^{k'}(k'+1)!}
\prod_{a=1}^{k'}\left[ \frac{(k'+a)!}{(2a+1)!!} \right]^2
= \biggl[ k!\prod_{a=2}^{k'}\binom{k}{a} \biggr]^{-1}.
\]

\eqref{prop:Bernoullidet2}
We have
\[
B'_k = \frac{1}{k!}D^-_{k-2,k'+2}\Delta_{k'-1}\begin{pmatrix}
\dfrac{\Bern{i+j+1}}{(i+j+1)!}
\end{pmatrix}_{1\leq i,j\leq k'-1}\Delta_{k'-1}D^+_{k-2,k'+2},
\]
and
\[
\det B'_k = \frac{1}{2^{k'^2-1}(k!)^{k'-1}}
\prod_{a=1}^{k'-1} \left[ \frac{(k'+1+a)!}{(2a+1)!!} \right]^2
= \biggl[ \frac{k}{4}\,(k'!)^2\prod_{a=2}^{k'}\binom{k}{a} \biggr]^{-1}.\qedhere
\]
\end{proof}

\section{Quadratic relations between periods and Bessel moments}\label{sec:Ppairing}

In this section, we express the period pairing between rapid decay homology and de Rham cohomology of $\Sym^k \Kl_2$ in terms of Bessel moments and we obtain quadratic relations between them by specializing to our setting the general results from \cite{F-S-Y20}.

\Subsection{Quadratic relations between periods}
We use the topological pairing~$\gen{\cbbullet,\cbbullet}_\topo$ on $\Kl_2$,
which is compatible with the $\QQ$-structure of $\Kl_2^\nabla$ from Section \ref{subsec:QKl2}. Recall that the induced pairing $\gen{\cbbullet,\cbbullet}_\topo$ on $\Sym^k\Kl_2$
is $(-1)^k$-symmetric. The period pairing $\Ppairing^{\rrd,\rmod}_1$ was defined in \cite{F-S-Y20},
where the index $1$ referred to the fact that we were pairing rapid decay homology and moderate growth de Rham cohomology in degree one. Here we denote it by $\Ppairing^{\rrd,\rmod}_k$, in order to emphasize that we are dealing with the $k$-th symmetric power and since there are no other non-trivial (co)homological degrees at play.
Recall the de Rham cohomology classes $\omega_i$ from~\eqref{eq:omegai} and the rapid decay cycles~$\alpha_i$ from \eqref{eq:alphai}. Since~$(\alpha_i)_{0\leq i\leq k'}$ is a basis of~$\coH^\rrd_1(\Gm,\Sym^k\Kl_2)$ by Theorem \ref{th:Bettibasis}\eqref{th:Bettibasis1} and $(\omega_i)_{0\leq i\leq k'}$ is a basis of $\coH_\dR^1(\Gm, \Sym^k\Kl_2)$ by~\cite[Prop.\,4.14]{F-S-Y18},
we deduce from the perfectness of the pairing $\Ppairing^{\rrd,\rmod}_k$(\cf\cite[Cor.\,2.11]{F-S-Y20}) that the period matrix $\bigl(\Ppairing^{\rrd,\rmod}_{k;i,j}\bigr)_{0\leq i,j\leq k'}$ defined by
\[
\Ppairing^{\rrd,\rmod}_{k;i,j}=\Ppairing^{\rrd,\rmod}_k(\alpha_i,\omega_j)
\]
is invertible. Thanks to the identity \eqref{eq:ve} relating $v_0$ to $e_0$ and the change of variables $t=2\sqrt{z}$, the first row of this matrix reads
\begin{equation}\label{eq:P0j}
\Ppairing^{\rrd,\rmod}_k(\alpha_0,\omega_j)
= \int_{c_0} \bgen{e_0^k,v_0^k}_\topo z^j\,\frac{\de z}{z}
= \int_{c_0} (2\pii)^k I_0(2\sqrt{z})^{k} z^j\,\frac{\de z}{z}
= (2\pii)^{k+1}\delta_{0,j},
\end{equation}
from which we immediately derive:

\begin{prop}\label{prop:Pk}
The $k'\times k'$
period matrix $\Ppairing_k=\bigl(\Ppairing^{\rrd,\rmod}_{k;i,j}\bigr)_{1\leq i,j\leq k'}$ is invertible.\qed
\end{prop}

The pure part of the pairing $\Ppairing^{\rrd,\rmod}_k$
arises from the pairing between middle homology and middle de Rham cohomology. According to \eqref{eqn:imagesinmiddle} and Theorem \ref{th:Bettibasis}\eqref{th:Bettibasis3}, the elements $(\alpha_i)_{i\in\lcr1,k'\rcr}$ map to a basis of~$\coH_1^\rmid(\Gm,\Sym^k\Kl_2)$
whenever $\knotfour$. If~$\kfour$, we instead consider the images of the shifted elements $(\alpha'_i)_{i\in\lcr1,k'\rcr}$, as introduced in Notation \ref{nota:Bpairingmid}. Regarding $\coH^1_{\dR,\rmid}(\Gm,\Sym^k\Kl_2)$, Corollary \ref{cor:BasisdRmid} gives the basis $(\omega_i)_{i\in\lcr1,k'\rcr}$ (\resp $(\omega_i')_{i\in\lcr1,k'\rcr}$) if $\knotfour$ (\resp if $\kfour$), where $\omega_i'$ is modified as in \eqref{eq:omegaprimei}. With this notation, the \emph{middle period matrix} is defined as follows:
\[
\Ppairing^\rmid_k=\bigl(\Ppairing^\rmid_{k;i,j}\bigr)_{i,j\in\lcr1,k'\rcr}
=\begin{cases}
\bigl(\Ppairing^{\rrd,\rmod}_k(\alpha_i,\omega_j)\bigr)_{i,j\in\lcr1,k'\rcr}&\text{if }\knotfour,\\
\bigl(\Ppairing^{\rrd,\rmod}_k(\alpha'_i,\omega'_j)\bigr)_{i,j\in\lcr1,k'\rcr}&\text{if }\kfour.
\end{cases}
\] In particular, $\Ppairing^\rmid_k=\Ppairing_k$ if $\knotfour$.

Recall the matrices $\DRpairing_k^\rmid$ from Section \ref{subsec:DRpairing} and $\Bpairing_k^\rmid$ from Notation \ref{nota:Bpairingmid}.
In the current setting,
the general method to express middle quadratic relations
explained in \cite[\S 3.f]{F-S-Y20},
namely in (3.21) therein, yields the following result:

\begin{thm}[Middle quadratic relations for $\Sym^k\Kl_2$]\label{th:middlequad}
The middle periods of $\Sym^k\Kl_2$ satisfy the following quadratic relations:
\[
(-2\pii)^{k+1}\,\Bpairing_k^\rmid=\Ppairing_k^\rmid\cdot(\DRpairing_k^\rmid)^{-1}\cdot{}^t\Ppairing_k^\rmid.
\]
In particular, the matrix $\Ppairing_k^\rmid$ is invertible.\qed
\end{thm}

\subsection{Bessel moments as periods}

We consider the power moments of the modified Bessel functions
$I_0$ and $K_0$ defined by
\[
\IKM_k(i,j) =
(-1)^{k-i}\,2^{k-j}(\pii)^i\int_0^\infty \hspace*{-2mm}I_0(t)^i K_0(t)^{k-i}t^j \,\de t,
\]
where the indices $i$ and $j$ are subject to the constraints
\[
0\leq i\leq k' \text{ and }j\geq0 \text{ or, if $k$ is even, } i=\sfrac{k}{2} \text{ and }0\leq j\leq k'-1.
\]
For this range of indices, it results from the asymptotic expansion at infinity \eqref{eq:BesselIKinfty} that the improper integral $\IKM_k(i,j)$ converges. In what follows, such moments will occur only for odd $j$.

\begin{prop}\label{prop:Pij}
For all $1 \leq i,j \leq k'$, the following equality holds:
\[
\Ppairing^{\rrd,\rmod}_{k;i,j}=\IKM_k(i,2j-1).
\]
\end{prop}

\begin{proof}
In view of the definition \eqref{eq:alphai} of the twisted cycles $\alpha_i$, we need to compute the integrals along $c_0^a$ of $\bgen{e_0^{a-1}e_1^{k-a+1},v_0^k}_\topo z^j\sfrac{\rd z}{z}$ for $a=1,\dots,k-i+1$ and the integral along $c_+$ of~$\bgen{e_0^ie_1^{k-i},v_0^k}_\topo z^j\sfrac{\rd z}{z}$.

Let us first remark that, for $\epsilon\in(0,1]$,
we can replace $c_0$ and $c_+$ with
the scalings $c_{0,\epsilon}$ and $c_{+,\epsilon}$ by $\epsilon$
defined with the base point $\epsilon$ instead of $1$, leading to a twisted cycle $\alpha_{i,\epsilon}$ equivalent to $\alpha_i$ ($1\leq i\leq k'$). We will show
\begin{align}
\label{eq:limfirstline}
\hspace*{7mm}\lim_{\epsilon\to0}\int_{c^a_{0,\epsilon}}\bgen{e_0^{a-1}e_1^{k-a+1},v_0^k}_\topo z^j\,\frac{\rd z}{z}&=0,\quad a=1,\dots,k-i+1,\\
\label{eq:limsecondline}
\lim_{\epsilon\to0}\int_{c_{+,\epsilon}}\bgen{e_0^ie_1^{k-i},v_0^k}_\topo z^j\,\frac{\rd z}{z}&=\IKM_k(i,2j-1)\quad\text{for }1 \leq i,j \leq k'.
\end{align}

We first show that the limits \eqref{eq:limfirstline} are zero.
According to \eqref{eq:selfdualityQ},
we only need to compute the coefficient of $v_0^k$ in $e_0^{k-a+1}e_1^{a-1}$,
which is equal to $I_0^{a-1}K_0^{k-a+1}$ up to a constant (\cf \eqref{eq:ve}).
It~remains to check that
\[
\int_{\arg t=0}^{a\pii}I_0(t)^{a-1}K_0(t)^{k-a+1}t^{2j-1}\rd t\to0\quad\text{when }|t|=2\sqrt\epsilon\text{ and }\epsilon\to0.
\]
We can use the estimate \eqref{eq:BesselIKzero} to compute the integral, and the assumption $j\geq1$ implies that the absolute value of the integral tends to zero with $\epsilon$.

For \eqref{eq:limsecondline},
recall that the coefficient of $v_0^k$ in $e_0^{k-i}e_1^i$
is equal to $2^k(\pii)^i\binom{k}{i}I_0^iK_0^{k-i}$,
and thus (\cf \eqref{eq:selfdualityQ})
\[
\bgen{e_0^ie_1^{k-i},v_0^k}_\topo=(-1)^{k-i}\binom{k}{i}^{-1}2^k(\pii)^i\binom{k}{i}I_0^iK_0^{k-i}=(-1)^{k-i}2^k(\pii)^iI_0^iK_0^{k-i}.
\]
The assertion \eqref{eq:limsecondline} then follows from the relation $z^j\rd z/z=2^{-2j+1}t^{2j-1}\rd t$.
\end{proof}

\begin{cor}\label{cor:PmidasBessel}
The matrix $\Ppairing_k^\rmid$ satisfies, for $i,j\in\lcr1,k'\rcr$,
\[
\Ppairing^\rmid_{k;i,j}=
\begin{cases}
\IKM_k(i,2j-1)&\text{if $\knotfour$},\\
\IKM_k(i+1,2j-1)-\gamma_{k,j}\IKM_k(i+1,k')&\text{if $\kfour$ and $i<k/4$,}\\
\IKM_k(i,2j-1)-\gamma_{k,j}\IKM_k(i,k')&\text{if $\kfour$ and $i>k/4$}.
\end{cases}
\]
\end{cor}

\begin{example}
Consider the case $k=8$.
Taking the determinant of the quadratic relations from Theorem \ref{th:middlequad} and using the computation \eqref{eqn:detSkmultiple4} and Proposition \ref{prop:Bernoullidet}\eqref{prop:Bernoullidet2} yield
\begin{displaymath}
(\det\Ppairing^\rmid_8)^2 =
(2\pii)^{18}\det \Ipairing^\rmid_8\det\DRpairing^\rmid_8= \sfrac{-5^2\pi^{18}}{2^43^2}.
\end{displaymath}
On the other hand, by Corollary \ref{cor:PmidasBessel} and the equalities $\gamma_{8, 1}=0$ and $\gamma_{8, 3}=\sfrac{1}{8}$ from \eqref{eq:valuesofgammaki}, we explicitly have:
\[
\Ppairing_8^\rmid
= \begin{pmatrix}
(\pii)^2 \\
& -(\pii)^3 \end{pmatrix}
A
\begin{pmatrix}
2^7 \\
& 2^2 \end{pmatrix},
\]
with
\[
A = \int_0^\infty
\begin{pmatrix}
I_0(t)^2 K_0(t)^6 t &&& 2I_0(t)^2 K_0(t)^6t^5 - I_0(t)^2K_0(t)^6t^3 \\[5pt]
I_0(t)^3 K_0(t)^5 t &&& 2I_0(t)^3K_0(t)^5t^5 - I_0(t)^3K_0(t)^5t^3
\end{pmatrix}\de t.
\]
One can check numerically that $\det A$ is positive,
and hence $\det A = \sfrac{5\pi^4}{2^{11}\cdot 3}$.
Using the linear relations
$\IKM^\cp_8(1,r) + \IKM^\cp_8(3,r) = 0$ from Corollary \ref{cor:lin_rel_IKM} below for $r=1,5$,
this verifies the numerical evaluation made in \cite[(2.5)]{B-R18}.
\end{example}

\subsection{Relation to the conjecture of Broadhurst and Roberts}\label{subsec:BR}

In \cite[\S5]{B-R18}, Broadhurst and Roberts use different normalizations to state their conjectural quadratic relations. Instead of the above matrices $\Bpairing_k$ and $\Ppairing_k$, they consider matrices that we shall denote by $\Bpairing^\BR_k$ and $\Ppairing_k^\BR$ (the~notation~$B_N$ and~$F_N$ is used in \loccit, the index $k$ being occupied by what is $k'$ here).
To~compare their matrices with ours, we introduce auxiliary square matrices
\[
\sfU_{k'}=\antidiag(1, \dots, 1), \quad \sfR_{k'}=\antidiag(\sfi, \sfi^2, \dots, \sfi^{k'}),
\quad \sfT_{k'}=\diag(-4, (-4)^2, \dots, (-4)^{k'})
\]
of size $k'$, where the anti-diagonal entries are listed down from the top corner. By Propositions~\ref{prop:Bk} and \ref{prop:Pij}, the matrices $\Bpairing^\BR_k$ and $\Ppairing_k^\BR$ relate to ours as
\[ \sfU_{k'}\Bpairing^\BR_k \sfU_{k'}=\dfrac{k!}{2^{k+1}}\begin{pmatrix}
\sfi^{(k+i+j-1)}\cdot\Bpairing_{k;i,j}
\end{pmatrix}_{1\leq i,j\leq k'}
\quad\text{and}\quad
\sfU_{k'}\Ppairing_k^\BR=\frac{1}{(-2\sqrt{\pi})^{k+1}}
	\cdot \begin{pmatrix} (-4)^j\,\sfi^i\,\Ppairing^{\rrd,\rmod}_{k;i,j}\end{pmatrix}_{1\leq i,j \leq k'}, \]
whence the identities
\begin{equation}\label{eqn:BRandours}
\Bpairing^\BR_k= -\frac{\sfi^{(k+1)}k!}{2^{k+1}}\,{}^t\sfR_{k'}\cdot \Bpairing_k \cdot \sfR_{k'}
\quad\text{and}\quad
\Ppairing_k^\BR= \frac{1}{(-2\sqrt{\pi})^{k+1}}\, {}^t\sfR_{k'}\cdot \Ppairing_k \cdot \sfT_{k'}.
\end{equation}

Besides, Broadhurst and Roberts \cite[p.\,7]{B-R18} define matrices $\Dpairing_k^\BR=\bigl(\Dpairing^{\BR}_{k;i,j}\bigr)_{1\leq i,j\leq k'}$ with rational coefficients and conjecture that, for all integers $k \geq 1$, the quadratic relation 
\[
\Ppairing_k^\BR\cdot \Dpairing_k^\BR\cdot {}^t\Ppairing_k^\BR = \Bpairing_k^\BR.
\]
holds. In the direction of this conjecture, we obtain:
\begin{cor}
If $\knotfour$, then the matrix $\Dpairing_k$ defined as
\[
\Dpairing_k=(-1)^kk!\,(\sfT_{k'}\cdot\Spairing_{k}^{\rmid}\cdot\sfT_{k'})^{-1}
\]
satisfies
\[
\Ppairing_k^\BR\cdot \Dpairing_k\cdot {}^t\Ppairing_k^\BR = \Bpairing_k^\BR.
\]
\end{cor}

\begin{proof}
Under the assumption on $k$, we have $\Ppairing_k=\Ppairing_k^\rmid$ and $\Bpairing_k=\Bpairing_k^\rmid$. The statement then follows from the quadratic relations of Theorem \ref{th:middlequad} and the equalities \eqref{eqn:BRandours}.
\end{proof}

If $\kfour$, we set
\[
\Bpairing_k^{\prime\BR}=-\frac{\sfi^{-(k+1)}k!}{2^{k+1}}\,{\sfR'_{k'}}\cdot \Bpairing^\rmid_k \cdot {\sfR'_{k'}}
\quad\text{and}\quad
\Ppairing_k^{\prime\BR}=\frac{1}{(-2\sqrt{\pi})^{k+1}}\, {\sfR'_{k'}}
	\cdot \Ppairing^\rmid_k \cdot \sfT'_{k'},
\]
where we denote by $\sfR'_{k'}$ (\resp $\sfT'_{k'}$) the matrix obtained from
$\sfR_{k'}$ (\resp $\sfT_{k'}$)
by deleting the row and the column of index~$k/4$, that we consider as indexed by $\lcr1,k'\rcr^2$. We define the matrix $\Dpairing'_k$ indexed by $\lcr1,k'\rcr^2$ so that it satisfies the relation (\cf\eqref{eq:Sprimekmidfour})
\[
\Dpairing'_k=(-1)^kk!\,(\sfT'_{k'}\cdot\Spairing_{k}^{\rmid}\cdot\sfT'_{k'})^{-1}.
\]

\begin{cor}
If $\kfour$, then the matrix $\Dpairing'_k$ satisfies
\[
\Ppairing^{\prime\BR}_k\cdot \Dpairing'_k\cdot {}^t\Ppairing^{\prime\BR}_k = \Bpairing^{\prime\BR}_k.
\]
\end{cor}

\begin{remark}
Numerical computations show that, for all integers $k\leq22$ that are not multiples of $4$, the equality $\Dpairing_k=\Dpairing_k^\BR$ holds. On the other hand, for $k=8,12,16,20$, the matrix $\Dpairing'_k$ coincides with the matrix $\Dpairing_k^{\prime\BR}$ obtained from $\Dpairing_k^\BR$ by deleting the row and the column of index $k/4$. \begin{comment}
Thanks to Theorem \ref{th:Smid}, it would follow from these equalities that $\Dpairing_k^\BR$ is an upper-left triangular matrix (\ie the entries $(i,j)$ are zero if $i+j\geq k'+2$) with anti-diagonal entries
\[
\Dpairing_{k; i, j}^\BR=\begin{cases}
\dfrac{(k!!)^2}{2^{k+1}} &\text{if $k$ is odd},\\[5pt]
\dfrac{(i-j)((\sfrac{k}{2})!)^2}{2} &\text{if $k$ is even}.
\end{cases}
\]
\end{comment}
These computations also seem to suggest that
$k!(\Spairing_{k}^{\rmid})^{-1}$ has integral coefficients.
Is it true for all $k$?
\end{remark}

\section{The full period matrices}
\label{sect:complete_periods}

In Theorem \ref{th:middlequad}, we emphasized quadratic relations for the middle periods to make the link with the conjecture of Broadhurst and Roberts. However, quadratic relations hold between the rapid-decay versus moderate periods
\[
\Ppairing^{\rrd,\rmod}_k:\coH_1^\rrd(\Gm,\Sym^k\Kl_2)\otimes\coH^1_{\dR}(\Gm,\Sym^k\Kl_2)\to\CC
\]
and the moderate versus rapid-decay periods
\[
\Ppairing^{\rmod,\rrd}_k:\coH_1^\rmod(\Gm,\Sym^k\Kl_2)\otimes\coH^1_{\dR,\rc}(\Gm,\Sym^k\Kl_2)\to\CC.
\]
Namely, it follows from \cite[Rem.\,2.15]{F-S-Y20} that
\[
(-2\pii)^{k+1}\,\Bpairing_k^{\rrd,\rmod}=\Ppairing_k^{\rrd,\rmod}\cdot(\DRpairing_k)^{-1}\cdot{}^t\Ppairing_k^{\rmod,\rrd},
\]
where $\Bpairing^{\rrd,\rmod}_k$ and $\DRpairing_k$ stand for the complete Betti and de~Rham intersection pairings.
In this section,
we consider the bases $\Basis_{k,\rc},\Basis_k$ and $(\alpha_i)_i,(\beta_i)_i$ as defined in Section~\ref{subsec:bases} and Theorem~\ref{th:Bettibasis}, and we compute the entries of the matrices of~$\Ppairing^{\rrd,\rmod}_k$ and $\Ppairing^{\rmod,\rrd}_k$ which do not appear in $\Ppairing^{\rmid}_k$.

\subsection{The complete period matrix \texorpdfstring{$\Ppairing^{\rrd,\rmod}$}{Prdmod}}
In order to finish the computation of the complete period matrix $\bigl(\Ppairing^{\rrd,\rmod}_{k;i,j}\bigr)_{0\leq i,j\leq k'}$,
we are left, according to \eqref{eq:P0j} and Proposition \ref{prop:Pij}, with computing the terms $\Ppairing^{\rrd,\rmod}_{k;i,0}$ for $i=1,\dots,k'$. In such a range, we consider regularized Bessel moments defined as follows, according to the expansions \eqref{eq:BesselIKzero} and~\eqref{eq:BesselIKinfty}.

\begin{defi}[Regularized Bessel moments]\label{defi:IKMreg}
For all $i$ such that $0\leq i \leq k'$, the functions
\[
G_{k,i}(\epsilon) = \int_\epsilon^\infty I_0(t)^i K_0(t)^{k-i}\,\frac{\de t}{t}
+ \frac{(-1)^{k-i}}{k-i+1}\left(\gamma + \log\Psfrac{\epsilon}{2} \right)^{k-i+1}
\]
are holomorphic on small sectors containing $\epsilon \in \RR_{>0}$
and have finite limit as $\epsilon\to\nobreak 0^+$. The regularized Bessel moments are defined as
\[
\IKM_k^\reg(i,-1)
= (-1)^{k-i}2^{k+1}(\pii)^i \lim_{\epsilon\to 0^+} G_{k,i}(\epsilon).
\]
\end{defi}

\begin{prop}\label{prop:shrinking-reg}
We have
\[
\Ppairing^{\rrd,\rmod}_k(\alpha_i,\omega_0)
=\IKM_k^\reg(i,-1),
\quad \text{for }1\leq i\leq k'.
\]
\end{prop}

\begin{proof}
We argue as for \eqref{eq:limfirstline} and \eqref{eq:limsecondline}.
If $1\leq i\leq k'$,
let
\[
\Ppairing'_k(\alpha_i,\omega_0)
= \Ppairing^{\rrd,\rmod}_k\Bigl(\alpha_i +\frac{(-1)^{k-i}(k-i)!}{k-i+2}\,\alpha_0,\omega_0\Bigr),
\]
so that
\begin{equation}\label{eq:Pkio}
\Ppairing^{\rrd,\rmod}_k(\alpha_i,\omega_0)=\Ppairing'_k(\alpha_i,\omega_0)-\frac{(-1)^{k-i}(k-i)!}{k-i+2}\,\Ppairing^{\rrd,\rmod}(\alpha_0,\omega_0).
\end{equation}
On the one hand, by scaling the chains $c_+$ and $c_0$ by $\epsilon \in \RR_{>0}$ and letting $\epsilon' = 2\sqrt\epsilon$, we find
\begin{align*}
\Ppairing'_k(\alpha_i,\omega_0)
&= \int_{c_{+,\epsilon}} \langle e_0^ie_1^{k-i},v_0^k\rangle_\topo\frac{\de z}{z}+ \sum_{a=1}^{k-i+1} C_{k-i+1}(a)\int_{c_{0,\epsilon}^a}
\langle e_0^{i-1}e_1^{k-i+1},v_0^k\rangle_\topo\frac{\de z}{z} \\
&= (-1)^{k-i}2^k(\pii)^i\int_\epsilon^\infty I_0(2\sqrt z))^iK_0(2\sqrt z)^{k-i}\frac{\de z}{z} \\
&\hspace{15mm}
+(-1)^{k-i+1}2^k(\pii)^{i-1}\sum_{a=1}^{k-i+1}C_{k-i+1}(a)
\int_{c_{0,\epsilon}^a} I_0(2\sqrt z))^{i-1}K_0(2\sqrt z)^{k-i+1} \frac{\de z}{z} \\
&= (-1)^{k-i}2^{k+1}(\pii)^i
\int_{\epsilon'}^\infty I_0(t)^i K_0(t)^{k-i}\frac{\de t}{t}\\
&\hspace{15mm}
+(-1)^{k-i+1}2^{k+1}(\pii)^{i-1}\sum_{a=1}^{k-i+1}C_{k-i+1}(a)
\int_{c_{0,\epsilon'}^{a/2}}
I_0(t)^{i-1}K_0(t)^{k-i+1}\, \frac{\de t}{t}\\
&= (-1)^{k-i}2^{k+1}(\pii)^i
\int_{\epsilon'}^\infty I_0(t)^i K_0(t)^{k-i}\,\frac{\de t}{t}\\
&\hspace{15mm}
+2^{k+1}(\pii)^{i-1}\sum_{a=1}^{k-i+1}C_{k-i+1}(a)
\int_{c_{0,\epsilon'}^{a/2}}
\Bigl[\bigl(\gamma + \log\Psfrac{t}{2}\bigr)^{k-i+1}+ O(t^2\log^{k-i+1}t)\Bigr]\, \frac{\de t}{t}.
\end{align*}
Lemma \ref{lemma:relation:c_n} yields
\begin{align*}
\sum_{a=1}^{k-i+1}C_{k-i+1}(a)&
\int_{c_{0,\epsilon'}^{a/2}}
\bigl(\gamma + \log\Psfrac{t}{2} \bigr)^{k-i+1} \frac{\de t}{t} \\
&= \frac{1}{k-i+2}\sum_{a=1}^{k-i+1}C_{k-i+1}(a)
\Bigl[ \bigl(\gamma + \log\Psfrac{\epsilon'}{2} + \pii a \bigr)^{k-i+2}- \bigl(\gamma + \log\Psfrac{\epsilon'}{2} \bigr)^{k-i+2} \Bigr] \\
&= \frac{\pii}{k-i+1} \bigl(\gamma + \log\Psfrac{\epsilon'}{2} \bigr)^{k-i+1}
+ \frac{(-1)^{k-i}(k-i)!}{k-i+2}(\pii)^{k-i+2}.
\end{align*}
Letting $\epsilon\to 0^+$, one obtains
\[
\Ppairing'_k(\alpha_i,\omega_0)
=\IKM_k^\reg(i,-1) + \dfrac{(-1)^{k-i}(k-i)!}{k-i+2}\,(2\pii)^{k+1}.
\]
On the other hand, $\Ppairing^{\rrd,\rmod}_k(\alpha_0,\omega_0)=(2\pii)^{k+1}$
as computed in \eqref{eq:P0j},
and \eqref{eq:Pkio} gives the desired formula.
\end{proof}

\subsection{The complete period matrix \texorpdfstring{$\Ppairing^{\rmod,\rrd}_k$}{Pmodrd}}

The period matrix $\Ppairing^{\rmod,\rrd}_k$ is defined~by
\[
\Ppairing^{\rmod,\rrd}_{k;i,j}=\Ppairing^{\rmod,\rrd}_k(\beta_i,(\wh m_j,\omega'_j)),\quad
\begin{cases}
0\leq i,j\leq k' &\text{if }\knotfour,\\
\begin{cases}
0\leq i\leq k/2,\ i\neq1\\
0\leq j\leq k'
\end{cases}
&\text{if }\kfour,
\end{cases}
\]
(\cf the notation of \eqref{eq:omegaprimei} for $\omega'_j$) and is non-degenerate (argument similar to that for $\Ppairing^{\rrd,\rmod}_k$). According to \cite[\S3.f]{F-S-Y20}, its middle part is equal to $\Ppairing^\rmid_k$ already computed.
\begin{itemize}
\item
If $\knotfour$, we are left with computing $\Ppairing^{\rmod,\rrd}_{k;i,j}$ for $i,j\in[0,k']$ and $i\text{ or }j=0$.
\item
If $\kfour$, we are left with computing $\Ppairing^{\rmod,\rrd}_{k;i,j}$ with
\[
\begin{cases}
i=0,2,\dots,k/2,\ j=0,k/4,\\
i=0,k/2,\ j=1,\dots,k',\ j\neq k/4.
\end{cases}
\]
\end{itemize}

\begin{defi}[Regularized Bessel moments, continued]
\label{defi:regIKM_conti}
If $\kfour$ and $k/4< j\leq k'$, the function $\epsilon\mto H_{k,j}(\epsilon)$ defined by (\cf \eqref{eq:gammaki})
\[
H_{k,j}(\epsilon) =\int_0^{1/\epsilon} (I_0(t)K_0(t))^{k/2} \left(t^{2j} - \frac{\gamma_{k,j}}{2^{k/2-2j}}t^{k/2}\right)\frac{\de t}{t}- \frac{1}{2^{k-2j+1}} \sum_{n=1}^{j} \frac{\gamma_{k,j-n}}{n(4\epsilon^2)^n}
\]
is holomorphic near $\RR_{>0}$ with finite limit when $\epsilon\to0^+$.
We set
\[
\IKM^\reg_k(k/2,2j-1)= 2^{k-2j+1}(\pii)^{k/2}\lim_{\epsilon\to 0^+}H_{k,j}(\epsilon).
\]
\end{defi}

\begin{prop}\label{prop:B4}\mbox{}
\begin{enumerate}
\item\label{prop:B41}
If $\knotfour$, we have
\begin{enumerate}
\item\label{prop:B41a}
$\Ppairing^{\rmod,\rrd}_{k;i,0}=(-1)^k\delta_{i,0},\quad i=0,\dots,k'$,\item\label{prop:B41b}
$\Ppairing^{\rmod,\rrd}_{k;0,j}=\IKM_k(0,2j-1),\quad j=1,\dots,k'$.
\end{enumerate}
\item\label{prop:B42}
If $\kfour$, we have
\begin{enumerate}
\item\label{prop:B42a}
$\Ppairing^{\rmod,\rrd}_{k;i,0}=\delta_{i,0},\quad i=0,\dots,k/2$,
\item\label{prop:B42b}
$\Ppairing^{\rmod,\rrd}_{k;i,k/4}=-(2\pii)^{k/2}2^{k/2}\binom{k}{k/2}^{-1}\delta_{i,k/2},\quad i=0,\dots,k/2$,
\item\label{prop:B42c}
$\Ppairing^{\rmod,\rrd}_{k;0,j}=\IKM_k(0,2j-1)-\gamma_{k,j}\IKM_k(0,k'),
\quad j=1,\dots,k',\ j\neq k/4$,
\item\label{prop:B42d}
$\Ppairing^{\rmod,\rrd}_{k;k/2,j}=
\begin{cases}
\IKM_k(k/2,2j-1), & 1\leq j<k/4,\\
\IKM^\reg_k(k/2,2j-1), & k/4<j\leq k'.
\end{cases}
$
\end{enumerate}
\end{enumerate}
\end{prop}

\begin{proof}
For \eqref{prop:B41a} and \eqref{prop:B42a}, we note that
\[
\Ppairing^{\rmod,\rrd}_k(\beta_i,(\wh{m}_0,0))
= \bgen{e_0^ie_1^{k-i},e_0^k}_\topo
= (-1)^k\delta_{i,0}.
\]
Similarly, for \eqref{prop:B42b} (\cf Proposition \ref{prop:solinfty} and \cite[Prop.\,3.18]{F-S-Y20}),
\[ \Ppairing_k^{\rmod,\rrd}(\beta_i, (\wh{m}_{k/4},0))
= -\bgen{e_0^ie_1^{k-i}, 2^k(\pii)^{k/2}(e_0e_1)^{k/2}}_\topo
=-2^k(\pii)^{k/2}\binom{k}{k/2}^{-1}\delta_{i,k/2}. \]
For the rest, since the coefficients of the cycle $\beta_i$ in terms of $\{u_a\}$, have logarithmic growth at $0$, and $\wh{m}_{j,0}$ is holomorphic and vanishes (Lemma \ref{lem:sol0}), the contribution $\lim_{z\to 0}\bgen{e_0^ie_1^{k-i}, \wh{m}_{j,0}}_\topo(z)$ of $\gen{\beta_i,\wh{m}_{j,0}}_\topo$ in \cite[Prop.\,3.18]{F-S-Y20} in the period pairing is zero. For $i\neq k/2$ or $i=k/2, 1\leq j<k/4$, the coefficient $(-1)^{k-i}(\pii)^i\xi_{j,i}$ of $\gen{\beta_i,\wh{m}_{j,\infty}}_\topo$ (\cf \eqref{eq:xiia}) is holomorphic near $\RR_{>0}$ and vanishes at $\infty$, so the contribution of $\wh{m}_{j,\infty}$ is zero too. In this case, one therefore has
\[ \Ppairing_k^{\rmod,\rrd}(\beta_i, (\wh{m}_j,\omega_j')_\topo)
= \int_{\RR_+} \bgen{e_0^ie_1^{k-i},\omega_j'}_\topo
= \IKM_k(i,2j-1)-\gamma_{k,j}\IKM_k(i,k') \]
as in the proof of Proposition \ref{prop:Pij}. This completes the cases \eqref{prop:B41b}, \eqref{prop:B42c} and the first part of \eqref{prop:B42d}.

It remains to check the second part of \eqref{prop:B42d}, with $k/4<j\leq k'$. Let
\[
\xi = (-4\pii I_0K_0)^{k/2} z^{j-1}\de z\quad\text{and}\quad \eta = (-4\pii I_0K_0)^{k/2} z^{k/4-1}\de z.
\]
We have
\[ \bgen{(e_0e_1)^{k/2},\omega_j'}_\topo
= \bgen{(e_0e_1)^{k/2},
\left(\omega_j + \res_{z=\infty}(\omega_j)\omega_{k/4}\right)}_\topo
= \xi + \res_{z=\infty}(\xi)\eta. \]
Note that $\wh{m}_{j,\infty}$ is defined so that $\de\gen{(e_0e_1)^{k/2},\wh{m}_{j,\infty}}_\topo = \xi + \res_{z=\infty}(\xi)\eta$ and $\gen{(e_0e_1)^{k/2},\wh{m}_{j,\infty}}_\topo$ has no constant term in the fractional Laurent series expansion in $1/z$. We~then obtain
\[ \Ppairing_k^{\rmod,\rrd}(\alpha_{k/2}, (\wh{m}_j,\omega_j'))
= \lim_{\tau\to\infty} \int_0^\tau \xi + \res_{z=\infty}(\xi)\eta
- \bgen{(e_0e_1)^{k/2},\wh{m}_{j,\infty}}_\topo(\tau)
= \IKM^\reg_k(k/2,2j-1).\qedhere \]
\end{proof}

\begin{example}
Respectively for $k=3,4$,
the complete quadratic relations lead to the equalities
of the regularized Bessel moments
\begin{align*}
\lim_{\epsilon\to 0^+} \int_\epsilon^\infty I_0(t)K_0(t)^2 \frac{\de t}{t}
	+\frac{1}{3}\left(\gamma+\log\frac{\epsilon}{2}\right)^3
&= \frac{3}{2}\int_0^\infty K_0(t)^3t\,\de t\cdot\int_0^\infty I_0(t)K_0(t)^2t\,\de t, \\
\lim_{\epsilon\to 0^+} \int_\epsilon^\infty I_0(t)K_0(t)^3 \frac{\de t}{t}
	-\frac{1}{4}\left(\gamma+\log\frac{\epsilon}{2}\right)^4
&= \frac{\pi^4}{120}.
\end{align*}
\end{example}

\begin{remark}[Determinant of $\Ppairing_k$]
A formula for the determinant of the matrix $\Ppairing_k$
in Proposition \ref{prop:Pk},
which implies in particular its non-vanishing, was conjectured by Broadhurst and Mellit in \cite{B-M16} and \cite[Conj.\,4 \& 7]{Broadhurst16} and proved by Zhou \cite{Zhou18}.
Since $\det \DRpairing_k$ is computed in Remark~\ref{rem:detDRpairing} and the determinant of the Betti intersection matrix is computed in Proposition~\ref{prop:Bernoullidet}
in order to prove its non-vanishing,
the quadratic relations also lead to a computation of
$\det\Ppairing_k$ up to sign.
\end{remark}

\begin{remark}[Dimension of the linear span of Bessel moments]
\label{rem:Zhou_dep}
In \cite[Th.\,1.2]{ZhouQind},
Zhou shows that, for each $i$ such that $1 \leq i \leq k'$ (\resp for $i=0$),
the $\QQ$-vector space $E_i$ spanned by the Bessel moments~$\IKM_k(i,2j-1)$ for all~$j \geq 1$
has dimension at most~$k'$ (\resp at most $k'+1$).
From the point of view of the present paper,
the upper bound $\dim_\QQ E_i\leq k'$ when $1\leq i\leq k'$
is a direct consequence of the fact that the~$\IKM_k(i,2j-1)$ result from pairing the fixed rapid decay cycle $\alpha_i$ with the de Rham cohomology classes $\omega_j=[z^ju_0\rd z/z]$. Provided $j \geq 1$ and $k$ is not a multiple of $4$, all these~$\omega_j$ lie in the middle de Rham cohomology~$\coH^1_{\dR, \rmid}(\Gm, \Sym^k \Kl_2)$, which has dimension $k'$. If $k$ is a multiple of $4$,
the periods $\Ppairing^\rmid_{k}(\alpha_i,\cbbullet)$
are generated by
$\IKM_k(i,2j-1)-\gamma_{k,j}\IKM_k(i,k')$ for $j\in[1,k']$,
with~$\gamma_{k,k/4}=1$.
For each $j\geq 1$,
the class $(z^j - \gamma_{k,j}z^{k/4})u_0\de z/z$
lies in $\coH_{\dR,\rmid}^1(\Gm, \Sym^k\Kl_2)$
and has period
$\IKM_k(i,2j-1)-\gamma_{k,j}\IKM_k(i,k')$ when paired with $\alpha_i$.
Therefore, $\dim_\QQ E_i \leq k'$ holds. In the case $i=0$,
the periods $\Ppairing^{\rmod,\rrd}_{k}(\alpha_0,\cbbullet)$
are generated by
\[
1 \quad \text{and}\quad \IKM_k(0,2j-1)-\gamma_{k,j}\IKM_k(0,k')\quad (1\leq j\leq k').
\]
Again, for each $j\geq 1$,
there exists $\wh{m}_j$ such that
the class
\[
(\wh{m}_j,(z^j - \gamma_{k,j}z^{k/4})u_0\de z/z)
\]
lies in~$\coH_{\dR,\cp}^1(\Gm, \Sym^k\Kl_2)$
and has period
\[
\IKM_k(0,2j-1)-\gamma_{k,j}\IKM_k(0,k')
\]
when paired with~$\alpha_0$.
Hence, $\dim_\QQ E_0 \leq k'+1$, as this is the dimension of de Rham cohomology with compact support.
\end{remark}

\subsection{Linear relations}
When $k$ is even,
there is a unique non-trivial linear relation
among the~\hbox{$k'+2$} classes $(\beta_j)_{0 \leq j \leq k/2}$
in $\coH^{\rmod}_1(\Gm, \Sym^k\Kl_2)$.
We determine explicitly the moderate 2-chain that yields this relation.

\begin{lemma}\label{lemma:2-chain_boundary}
For two integers $n,r$ with $0\leq r\leq 2n$,
\[ \sum_{i=0}^{\min\{n,2n-r\}} (-1)^i\binom{n}{i}\binom{2n-i}{r} =
\begin{cases}
0 & \text{if $0\leq r\leq n-1$}, \\
\dpl\binom{n}{r-n} & \text{if $n\leq r\leq 2n$}.
\end{cases} \]
\end{lemma}

\begin{proof}
We provide two approaches;
the second one has been provided by Hao-Yun Yao of NTU.

Let $a_r$ be the sum in the left hand side.
Then the polynomial
$f(x) = \sum_{r=0}^{2n} a_rx^r$ is the coefficient of $y^n$
in the expansion of
\[
g(x,y) = (y-1)^n\sum_{r=0}^n (x+1)^{2n-r}y^r= (x+1)^n(y-1)^n \sum_{j=0}^n (x+1)^{n-j}y^j.
\]
Thus one obtains
\[ f(x) = (x+1)^n((x+1)-1)^n = x^n(x+1)^n \]
and the assertion follows.

Alternatively,
let $S$ be the set of subsets of $\{ 1,2,\dots, 2n\}$
consisting of $r$ elements
and define~\hbox{$T = \{ A \in S \,|\, \{ 1,2, \dots, n\} \subset A \}$.}
Members of $T$ are obtained by
choosing $(r-n)$ elements from
$\{ n+1, \dots, 2n \}$
so that the cardinality $|T|$ equals
$0$ if $0\leq r < n$
or $\binom{n}{r-n}$ if $n\leq r\leq 2n$.
Let $S_c = \{ A \in S \,|\, c \not\in A \}$.
One has $S = T \sqcup (S_1\cup\cdots\cup S_n)$
and
$|S_{c_1} \cap\cdots\cap S_{c_i}| = \binom{2n-i}{r}$
for~$1\leq c_1 < \cdots < c_i \leq 2n$.
The inclusion-exclusion principle then gives the formula.
\end{proof}

\begin{prop}\label{prop:2-chain}
Let $\wt\PP^1$ be the real oriented blow-up of $\PP^1$ along $\{0,\infty\}$ as in Section \ref{subsec:QKl2}, and
consider the simplicial $2$-chain
\begin{gather*}
\rho: \{ (x,y) \in \RR^2 \,|\, 0\leq x,y,x+y \leq 1 \} \to \wt\PP^1,\\ \rho(x,y) = \tan\frac{\pi (x+y)}{2} \exp\left( 4\sfi\tan^{-1}\Psfrac{y}{x} \right),
\end{gather*}
which covers $\wt\PP^1$ once.
If $k$ is even,
the singular chain
\[ \Delta = \rho \otimes \sum_{i=0}^{k/2} (-1)^i\binom{k/2}{i} e_0^ie_1^{k-i} \]
is of moderate growth, and the relation
\[ \sum_{i=0}^{\flr{(k-2)/4}} \binom{k/2}{2i+1}\beta_{2i+1} = 0 \] holds in $\coH^{\rmod}_1(\Gm, \Sym^k\Kl_2)$.
\end{prop}

\begin{proof}
It follows as direct consequences of
the monodromy relation and Lemma \ref{lemma:2-chain_boundary},
which also imply that
\[ \frac{-1}{2}\pt\Delta
= \sum_{i=0}^{\flr{(k-2)/4}} \binom{k/2}{2i+1}\beta_{2i+1}.
\qedhere\]
\end{proof}

If $\ktwofour$, choosing the principal determination of $w^{1/2}$
near $\RR_{>0}$, we write, in a way similar to \eqref{eq:I0K0}, the asymptotic expansion
\[
2^k(I_0(t)K_0(t))^{k/2} \sim w^{k/4}\sum_{n=0}^\infty \gamma'_{k,n+k/4}w^n.
\]
As in Definition \ref{defi:regIKM_conti}, one checks that, for $j\in[(k+2)/4,k']$, the difference
\[ H_{k,j}(\epsilon) =
\int_0^{1/\epsilon} (I_0(t)K_0(t))^{k/2}t^{2j}\frac{\de t}{t}
- \frac{1}{2^{k-2j+1}\epsilon}
\sum_{n=0}^{j} \frac{\gamma'_{k,j-n-1/2}}{(2n+1)(4\epsilon^2)^n} \]
is holomorphic near $\RR_{>0}$ with finite limit as $\epsilon\to0$.
Set
\[ \IKM^\reg_k(k/2,2j-1) =
-2^{k-2j+1}(\pii)^{k/2}\lim_{\sigma\to 0^+} H_{k,j}(\sigma). \]
In this case, one has
\[ \Ppairing^{\rmod,\rrd}_k(\beta_{k/2},(\wh{m}_j,\omega'_j)) =
\begin{cases}
0 & \text{if $j=0$}, \\
\IKM_k(k/2,2j-1) & \text{if $1\leq j\leq (k-2)/4$}, \\
\IKM^\reg_k(k/2,2j-1) & \text{if $(k+2)/4 \leq j\leq k'$}
\end{cases} \]
as in Proposition \ref{prop:B4}(\ref{prop:B42}).
To unify the situations,
we introduce the compactly supported version
of the Bessel moments.

\begin{defi}
Assume $k$ is even.
We set
\[ \IKM^\cp_k(i,2j-1) = \begin{cases}
\IKM^\reg_k(k/2, 2j-1)
&\text{if $\begin{cases}i = \sfrac{k}{2},\\ \flr{\sfrac{k}{4}}+1\leq j\leq k',\end{cases}$} \\
\IKM_k(i,2j-1) - \gamma_{k,j}\IKM_k(i,k')
& \text{if $\kfour,\ 0\leq i\leq k'$}, \\
\IKM_k(i,2j-1) & \text{otherwise}.
\end{cases}\]
\end{defi}

\begin{cor}[Sum rule identities]\label{cor:lin_rel_IKM} Assume $k$ is even and set $k'' = \flr{\psfrac{k-1}{4}}$. The linear relations among Bessel moments
\begin{equation}\label{eq:linear-IKM}
\sum_{i=0}^{k''} \binom{k/2}{2i+1} \IKM_k^\cp(2i+1,2j-1) = 0
\end{equation} hold for all $j$ such that
\[
\begin{cases}
1\leq j\leq 2k''&\text{if }\ktwofour,\\
1\leq j\leq 2k''+1,\ j\neq k/4&\text{if }\kfour.
\end{cases}
\]
\end{cor}

\begin{remark}
For $j\in[1,k'']$, these \textit{sum rule identities} were proved by Zhou by analytic means
(see \cite[(1.3)]{Zhou19} for $\ktwofour$ and \cite[(1.5)]{Zhou19} for $\kfour$, where the second argument of the Bessel moments is also allowed to be even). Our proof, closer to the spirit of the Kontsevich-Zagier period conjecture, produces the relation \eqref{eq:linear-IKM} simply from the Stokes formula.
\end{remark}

\section{Comparison of period structures}
\label{sect:Lfunctions}

In \cite{F-S-Y18}, we have introduced the Nori motives $\Motive_k$ defined over $\QQ$, whose definition we recall in Section~\ref{subsec:Mk}, and we have explained how the Hodge filtration of their Hodge realization can be computed in terms of symmetric powers of the Kloosterman connection on $\Gm$ and their irregular Hodge filtration. We aim at making Deligne's conjecture on critical values explicit for them in Section \ref{subsec:Deligne}.
In this section as a preparation,
we focus on the comparisons of various $\QQ$-structures.
For that purpose, we only retain from each Nori motive $\Motive_k$
\begin{itemize}
\item
its \emph{period realization} (over $\Spec\QQ$, \cf Section \ref{subsec:PerSpecQ}), that is, the $\QQ$-vector spaces $(\Motive_k)_\dR$ (de~Rham realization) and $(\Motive_k)_\Betti$ (Betti realization), together with the period isomorphism $\per:\CC\otimes_\QQ(\Motive_k)_\Betti\simeq\CC\otimes_\QQ(\Motive_k)_\dR$,
and
\item
the action $F_\infty$ on $(\Motive_k)_\Betti$ induced by geometric complex conjugation.
\end{itemize}

On the other hand, the computations of the previous sections lead us to define a (transposed) period structure (over $\Spec\QQ$) as follows:
\begin{itemize}
\item
The de~Rham $\QQ$-vector space $\coH_{\dR,\rmid}^1(\Gm,\Sym^k\Kl_2)_\QQ$ is the $\QQ$-vector space generated by the family $\Basis_{k,\rmid}$ of Corollary \ref{cor:BasisdRmid}.

\item
The (dual) Betti $\QQ$-vector space $\coH_1^\rmid(\Gm,\Sym^k\Kl_2)_\QQ$ is the $\QQ$-vector space generated by the family $(\beta_j)$ of twisted cycles of Theorem \ref{th:Bettibasis}\eqref{th:Bettibasis3}.

\item
The period pairing is the pairing $\Ppairing_k^\rmid$.
\end{itemize}

There is an obvious notion of morphism of period structures.

\begin{thm}\label{th:periodstructures}
The period structure $((\Motive_k)_\dR,(\Motive_k)_\Betti,\per)$
of the motive $\Motive_k$
is isomorphic to the transpose of $(\coH_{\dR,\rmid}^1(\Gm,\Sym^k\Kl_2)_\QQ,\coH_1^\rmid(\Gm,\Sym^k\Kl_2)_\QQ,\Ppairing_k^\rmid)$.
\end{thm}

Furthermore, we will give an explicit description under this correspondence of the action $F_\infty$ on $\coH_1^\rmid(\Gm,\Sym^k\Kl_2)_\QQ$.

We proceed in two steps.
Firstly in Section \ref{subsec:Mk}, in an analogous way to the method used in~\cite{F-S-Y18}, we realize the mixed Hodge structure corresponding to $\Motive_k$ together with its associated period structure as the exponential mixed Hodge structure and the associated period structure attached to a function on a smooth variety equipped with the action of a finite group. We also analyze the automorphism
of the Betti fiber of this exponential mixed Hodge structure
induced by the complex conjugation on the variety underlying $\Motive_k$.
In particular, we avoid computing periods directly on the variety underlying the motive $\Motive_k$. The tools for this part are developed in the appendix in which the period realization
of exponential mixed Hodge structures is investigated,
extending the focus on the de Rham realization
in the appendix of \cite{F-S-Y18}.
Then in Section \ref{subsec:Uf-SymKl},
we compare these objects to those of the previous sections
by making explicit differential forms of higher degree and higher dimensional twisted cycles that correspond to the bases obtained for $\Sym^k\Kl_2$ there. The period matrix computed in Proposition \ref{prop:Pij} is thus a period matrix for $\Motive_k$, and the action of the conjugation is explicit since the Bessel moments we consider are either real or purely imaginary.

\subsection{The motives \texorpdfstring{$\Motive_k$}{Mk}}\label{subsec:Mk}
Let $y=(y_1,\cdots,y_k)$ be the Cartesian coordinates
of the torus $\Gm^k$ defined over $\QQ$.
Upon $\Gm^k$,
consider the action of $\symgp_k\times\mu_2$
where the symmetric group $\symgp_k$ acts by permuting the variables $y$
and the group $\mu_2=\{\pm1\}$ acts as $y_i\mto\pm y_i$.
Denote by $g_k\colon\Gm^k\to\Afu$ the regular function
given by the Laurent polynomial
\[
g_k(y_1,\dots,y_k)=\sum_{i=1}^k \Bigl(y_i+\frac{1}{y_i}\Bigr)
\]
and let $\KM_0 = (g_k)$
be the associated closed subvariety of $\Gm^k$.
Then $\KM_0$ is invariant under the action of $\symgp_k\times\mu_2$
and hence the latter acts on various cohomology spaces,
or on the (Nori) motives $\coH^{k-1}(\KM_0)$ and $\coH_\cp^{k-1}(\KM_0)$
of degree $k-1$ of $\KM_0$.
Let $\sgn\colon\symgp_k\times\mu_2\to \QQ^\times$
be the product of the sign character on $\symgp_k$
and the trivial one on $\mu_2$.
We define the pure motive $\Motive_k$ of weight $k+1$ over $\QQ$ by
taking the $\sgn$-isotypic part
\[ \Motive_k =
\gr^W_{k+1}\left[\coH_\cp^{k-1}(\KM_0)(-1)\right]^{\symgp_k\times\mu_2,\sgn}, \]
where $W$ indicates the weight filtration (\cf\cite[(3.1)]{F-S-Y18}).
The Betti realization $(\Motive_k)_\Betti$ and the de~Rham realization $(\Motive_k)_\dR$
of $\Motive_k$ are the $\QQ$-vector spaces of dimension
$\flr{\psfrac{k-1}{2}}-\delta_{4\ZZ}(k)$
obtained by replacing $\coH_\cp$
by the corresponding cohomology functors (with compact support).

Let $\KM$ be the complex variety defined by the base change of $\KM_0$ to $\CC$.
It is shown in \loccit\ that the base change
\[ \CC\otimes (\Motive_k)_\dR
= \gr^W_{k+1}\left[\coH_{\dR,\cp}^{k-1}(\KM)(-1)\right]^{\symgp_k\times\mu_2,\sgn} \]
of the de Rham realization $(\Motive_k)_\dR$
is identified with $\coH_{\dR,\rmid}^1(\Gm,\Sym^k\Kl_2)$
discussed in the previous sections.

\subsubsection*{The pair $(U,\wt{f}_k)$}
We consider the torus $U_0=\Gmt\times\Gm^k$ over $\QQ$ with its $\CC$-exten\-sion $U$, endowed with
\begin{itemize}
\item
the action of $\symgp_k$ permuting the coordinates on $\Gm^k$,
\item
the action of $\mu_2$ sending $(t,y)$ to $\pm(t,y)$,
\item
the involution $\iota$ sending $t$ to $t$ and $y_i$ to $-y_i$ for each $i$,
\item
the anti-linear involution $\conj$ on the analytic manifold $U(\CC)$
induced by the conjugation of coordinates, which commutes with all the above actions.
\end{itemize}
Let $\wt f_k:U_0=\Gmt\times\Gm^k\to\Afu$ denote the Laurent polynomial $(\sfrac t2)\cdot g_k$ where $g_k$ is defined above. Then $\wt f_k$ is left invariant by the action of $\symgp_k$, $\mu_2$ and $\conj$, and satisfies $\iota^*\wt{f}_k = -\wt{f}_k$.

Conjugation acts on $\Omega^p(U)$ by changing a $p$-form $\omega(t,y)$ to $\ov{\omega(\ov t,\ov y)}$. Therefore, $\conj^*(\rd\wt f_k)=\rd\wt f_k$ and $\conj^*$
induces an involution on
$\coH^{k+1}_{\dR,\rc}(U,\wt f_k)$, $\coH_\dR^{k+1}(U,\wt f_k)$ and $\coH^{k+1}_{\dR,\rmid}(U,\wt f_k)$ which commutes with the action induced by $\symgp_k\times\mu_2$ and $\iota^*$.

According to \cite[Th.\,3.8]{F-S-Y18},
the associated mixed Hodge structure $(\Motive_k)_\rH$
is identified with the exponential mixed Hodge structure $\coH^{k+1}_\rmid(U,\wt f_k)^{\symgp_k \times \mu_2,\,\sgn}$.
In particular, they have isomorphic period realizations (\cf Appendix).

\subsubsection*{The de Rham realization $(\Motive_k)_{\dR}$}
The pair $(U,\wt{f}_k)$, together with the action of $\symgp_k\!\times\!\mu_2$, is the $\CC$-extension of the pair $(U_0,\wt{f}_k)$ defined over~$\QQ$. As explained in Section \ref{subsec:PerSpecQ}, the de~Rham cohomologies $\coH^{k+1}_{\dR,?}(U,\wt f_k)$ ($?=\emptyset,\rc,\rmid$) are also endowed with the $\QQ$-structure
$\coH^{k+1}_{\dR,?}(U_0,\wt f_k)$,
and therefore so are the $\sgn$-isotypic components $\coH^{k+1}_{\dR,?}(U,\wt f_k)^{\symgp_k \times \mu_2,\,\sgn}$
with the $\QQ$-structure
$\coH^{k+1}_{\dR,?}(U_0,\wt f_k)^{\symgp_k \times \mu_2,\,\sgn}$.

\begin{lemma}
The isomorphism
$\gr^W_{k+1}\coH_{\dR,\cp}^{k-1}(\KM)(-1)^{\symgp_k\times\mu_2,\sgn} \simeq
\coH^{k+1}_{\dR,\rmid}(U,\wt f_k)^{\symgp_k \times \mu_2,\,\sgn}$
of \cite[Th.\,3.8]{F-S-Y18} identifies the $\QQ$-vector spaces $(\Motive_k)_{\dR}$ and $\coH^{k+1}_{\dR,\rmid}(U_0,\wt f_k)^{\symgp_k \times \mu_2,\,\sgn}$.
\end{lemma}

\begin{proof}
We start from Example \ref{ex:dRQiso}, on noting that $\wt f_k=tg_k$.
We thus obtain isomorphisms (setting $Z_0=\Afu_t\times\cK_0$)
\begin{align*}
\coH^{k+1}_{\dR,\rc}(Z_0)^{\symgp_k \times \mu_2,\,\sgn}&\simeq \coH^{k+1}_{\dR,\rc}(\Afu_t\times\Gm^k/\QQ,\wt f_k)^{\symgp_k \times \mu_2,\,\sgn} \\
&\simeq\coH^{k+1}_{\dR,\rc}(U_0,\wt f_k)^{\symgp_k \times \mu_2,\,\sgn},
\end{align*}
where the second isomorphism is obtained as in the proof of \cite[Th.\,3.8]{F-S-Y18}. After extension of scalars from $\QQ$ to $\CC$, these isomorphisms become the de~Rham part of the isomorphisms considered in the proof of \loccit\ Since the left-hand sides are defined from Nori motives, the weight filtration is defined on them. We note that $(\Motive_k)_\dR\simeq\coH^{k+1}_{\dR,\rc}(Z_0)^{\symgp_k \times \mu_2,\,\sgn}/W_k\coH^{k+1}_{\dR,\rc}(Z_0)^{\symgp_k \times \mu_2,\,\sgn}$ because this holds after extension of scalars.
Furthermore, the composition
\[
\coH^{k+1}_{\dR,\rc}(Z_0)^{\symgp_k \times \mu_2,\,\sgn}\isom\coH^{k+1}_{\dR,\rc}(U_0,\wt f_k)^{\symgp_k \times \mu_2,\,\sgn}\to\coH^{k+1}_{\dR}(U_0,\wt f_k)^{\symgp_k \times \mu_2,\,\sgn}
\]
whose image is
$\coH^{k+1}_{\dR,\rmid}(U_0,\wt f_k)^{\symgp_k \times \mu_2,\,\sgn}$
by definition, has kernel equal to $W_k\coH^{k+1}_{\dR,\rc}(Z_0)^{\symgp_k \times \mu_2,\,\sgn}$, because this holds after extension of scalars. This proves the lemma.
\end{proof}

\subsubsection*{The Betti realization $(\Motive_k)_{\Betti}$}
We use the notation and apply the results of Appendix \ref{subsec:fiberperiodfunction},
more specifically Formulas \eqref{eq:BettiUfD}, \eqref{eq:BettiUfDPhi} and \eqref{eq:BettiUfDPhimid},
to the pair $(U,\wt f_k)$.
Since $\wt f_k$ is invariant under the action of $\symgp_k \times \mu_2$ on $U$, as well as under the action of complex conjugation $\conj:U(\CC)\to U(\CC)$, the families of supports~$\Phi_\rrd$ and~$\Phi_\rmod$ are also invariant under these actions. There is thus a natural action of $\symgp_k \times \mu_2$ on $\coH^{k+1}_\rc(\wt U_\rrd(D),\QQ)$ and $\coH^{k+1}_\rc(\wt U_\rmod(D),\QQ)$. We thus have
\begin{equation}
\label{eq:Bettirealcohomology}
\begin{aligned}
(\Motive_k)_{\Betti}&\simeq\coH^{k+1}_{\Betti,\rmid}(U,\wt f_k)^{\symgp_k \times \mu_2,\,\sgn}\\
&=\image\bigl[\coH^{k+1}_{\Betti,\rc}(U,\wt f_k)^{\symgp_k \times \mu_2,\,\sgn}\ra\coH^{k+1}_\Betti(U,\wt f_k)^{\symgp_k \times \mu_2,\,\sgn}\bigr]\\
&\simeq\image\bigl[\coH^{k+1}_\rc(\wt U_\rrd(D),\QQ)^{\symgp_k \times \mu_2,\,\sgn}\ra\coH^{k+1}_\rc(\wt U_\rmod(D),\QQ)^{\symgp_k \times \mu_2,\,\sgn}\bigr].
\end{aligned}
\end{equation}

\begin{lemma}
Under the identification \eqref{eq:Bettirealcohomology}, the action of the conjugation on $(\Motive_k)_\Betti$
coincides with that on
$\coH^{k+1}_{\Betti,\rmid}(U,\wt f_k)^{\symgp_k \times \mu_2,\,\sgn}$.
\end{lemma}

\begin{proof}
We start from the identification of Example \ref{ex:conj}. Since the action of $\symgp_k\times\mu_2$ on $\Afu_t\times\Gm^k$ commutes with $\conj$, we obtain an identification
\[
(\Motive_k)_\Betti\simeq\coH^{k+1}_{\Betti,\rmid}(\Afu_t\times\Gm^k,\wt f_k)^{\symgp_k \times \mu_2,\,\sgn}
\]
compatible with $\conj^*$. Since the isomorphism (\cf\cite[Proof of Th.\,3.8]{F-S-Y18})
\[
\coH^{k+1}_{\Betti,\rmid}(\Afu_t\times\Gm^k,\wt f_k)^{\symgp_k \times \mu_2,\,\sgn}\simeq\coH^{k+1}_{\Betti,\rmid}(U,\wt f_k)^{\symgp_k \times \mu_2,\,\sgn}
\]
is clearly compatible with $\conj^*$
the desired result follows.
\end{proof}

The Betti realization is also given~by (\cf \eqref{eq:Bettihomology})
\begin{equation}\label{eqn:Bettireal}
(\Motive_k)_{\Betti}\simeq\image\bigl[\coH^\rrd_{k+1}(U,\wt f_k)^{\symgp_k \times \mu_2,\,\sgn}\ra\coH^\rmod_{k+1}(U,\wt f_k)^{\symgp_k \times \mu_2,\,\sgn}\bigr]
\end{equation}
and carries the involution $F_\infty=\conj_*$.
Indeed,
the expression \eqref{eq:UrdR}
makes explicit the interpretation of the elements of $\coH^\rrd_{k+1}(U,\wt f_k)$ in terms of twisted cycles. It is moreover clear that, in this expression, the closed subspace $\partial_{\rrd,R}U$ of $U(\CC)$, upon which $\exp(-\wt{f}_k)$ is sufficiently small, is invariant under both the action of $\symgp_k \times \mu_2$ and the complex conjugation on $U(\CC)$, and hence provide the action of $\symgp_k \times \mu_2$ and $F_\infty$ on $\coH^\rrd_{k+1}(U,\wt f_k)$.

\subsubsection*{Self-duality and period pairing}
The involution $\iota$ yields the identifications
\[ \mu: \coH^{k+1}_?(U,\wt{f}_k) \isom \coH^{k+1}_?(U,-\wt{f}_k),
\quad
? = \emptyset, \cp,
\]
compatible with the action of $\symgp_k \times \mu_2$, and leads to the self-duality
up to a Tate twist by $\QQ(k+1)$ the mixed Hodge structure $(\Motive_k)_\rH\simeq\coH^{k+1}_\rmid(U,\wt{f}_k)^{\symgp_k \times \mu_2,\,\sgn}$.

The nondegenerate period pairing
\[
\Ppairing^{\rrd,\rmod}:\coH^\rrd_{k+1}(U,\wt f_k)^{\symgp_k \times \mu_2,\,\sgn}\otimes\coH^{k+1}_\dR(U,\wt f_k)^{\symgp_k \times \mu_2,\,\sgn}\to\CC,
\]
is computed as follows. Given a twisted cycle
\[
\alpha\in\coH^\rrd_{k+1}(U,\wt f_k)=\coH_{k+1}(\wt U_\rrd,\partial\wt U_\rrd,\QQ)^{\symgp_k \times \mu_2,\,\sgn},
\]
we choose for each large enough $R\!>\!0$ a representative $\alpha_R\!\in\!\coH_{k+1}(U,\partial_{\rrd,R}U,\QQ)$ $\sgn$-invariant under $\symgp_k \times \mu_2$. We also represent a de~Rham class in the space $\coH^{k+1}_\dR(U,\wt f_k)^{\symgp_k \times \mu_2,\,\sgn}$ by a top differential form $\omega$ on $U$ which is $\sgn$-invariant under $\symgp_k \times \mu_2$. Then
\[
\Ppairing^{\rrd,\rmod}(\alpha,[\omega])=\lim_{R\to\infty}\int_{\alpha_R}e^{-\wt f_k}\iota^*\omega.
\]

\subsection{Proof of Theorem \ref{th:periodstructures}}\label{subsec:Uf-SymKl}
The goal of this section is to identify the period structure
\[
(\coH^{k+1}_{\dR,\rmid}(U_0,\wt f_k)^{\symgp_k \times \mu_2,\,\sgn},\coH^{k+1}_{\Betti,\rmid}(U,\wt f_k)^{\symgp_k \times \mu_2,\,\sgn},\per)
\]
as defined from Sections \ref{subsec:computGM} and \ref{subsec:fiberperiodfunction} with the transpose of
\[
(\coH_{\dR,\rmid}^1(\Gm,\Sym^k\Kl_2)_\QQ,\coH_1^\rmid(\Gm,\Sym^k\Kl_2)_\QQ,\Ppairing_k^\rmid),
\]
ending thereby the proof of Theorem \ref{th:periodstructures}, and to make explicit the action of $F_\infty$ on $\coH_1^\rmid(\Gm,\Sym^k\Kl_2)_\QQ$.

\subsubsection*{Summary of the proof}\mbox{}
\begin{enumerate}\renewcommand{\labelenumi}{Step~\theenumi.}
\item\label{summary:0}
We consider the family $(w_j)_{0\leq j\leq k'}$
in $\coH^{k+1}_{\dR}(U_0,\wt f_k)^{\symgp_k \times \mu_2,\,\sgn}$
as defined in \cite[Proof of Prop.\,4.21]{F-S-Y18}, which has been shown to form a $\QQ$-basis of this space.
\item\label{summary:1}
We construct a family of rapid decay cycles $(\tau_i)$ in
$\coH^\rrd_{k+1}(U,\wt f_k)^{\symgp_k \times \mu_2,\,\sgn}$
such that the matrix $(\Ppairing^{\rrd,\rmod}(\tau_i,w_j))_{0\leq i,j\leq k'}$ is equal to $(\Ppairing^{\rrd,\rmod}_{k;i,j})_{0\leq i,j\leq k'}$ obtained from Propositions \ref{prop:Pij} and \ref{prop:shrinking-reg}. Since this matrix is nondegenerate and $(w_j)$ is a basis of $\coH^{k+1}_{\dR}(U,\wt f_k)^{\symgp_k \times \mu_2,\,\sgn}$, we conclude that $(\tau_i)$ is a basis of $\coH^{k+1}_\rrd(U,\wt f_k)^{\symgp_k \times \mu_2,\,\sgn}$.

\item\label{summary:2}
According to Remark \ref{rem:FPerP}, the period realization of $\coH^{k+1}(U_0,\wt f_k)^{\symgp_k \times \mu_2,\,\sgn}$
is isomorphic to
\[
(\QQ\cdot(w_j),\QQ\cdot(\tau_i),(\Ppairing^{\rrd,\rmod}_{i,j})_{0\leq i,j\leq k'}),
\]
hence to
\[
(\QQ\cdot(\omega_j),\QQ\cdot(\alpha_i),(\Ppairing^{\rrd,\rmod}_{k;i,j})_{0\leq i,j\leq k'}),
\]
that~is,
\[
(\coH_\dR^1(\Gm,\Sym^k\Kl_2)_\QQ,\coH_1^\rrd(\Gm,\Sym^k\Kl_2)_\QQ,\Ppairing_k^{\rrd,\rmod}).
\]

\item\label{summary:3}
Owing to the duality argument explained in \cite[Lem.\,2.27]{F-S-Y20},
the same property holds for the period realization of $\coH^{k+1}_\rc(U_0,\wt f_k)^{\symgp_k \times \mu_2,\,\sgn}$ and $(\coH_{\dR,\rc}^1(\Gm,\Sym^k\Kl_2)_\QQ,\coH_1^\rmod(\Gm,\Sym^k\Kl_2)_\QQ,\Ppairing_k^{\rmod,\rrd})$.

\item\label{summary:4}
That the same property thus holds true for the period realization of $\coH^{k+1}_\rmid(U_0,\wt f_k)^{\symgp_k \times \mu_2,\,\sgn}$ and
\[
(\coH_{\dR,\rmid}^1(\Gm,\Sym^k\Kl_2)_\QQ,\coH_1^\rmid(\Gm,\Sym^k\Kl_2)_\QQ,\Ppairing_k^\rmid),
\]
which needs a supplementary argument when $k\equiv0\bmod4$.

\item\label{summary:5}
Lastly, that conjugation is compatible with these identifications amounts to checking whether the entries of the period matrix are real or purely imaginary.
\end{enumerate}

\subsubsection*{A basis of $\coH^{k+1}_{\dR}(U_0,\wt f_k)^{\symgp_k \times \mu_2,\,\sgn}$}
Let us consider the class $[\wt\omega_j]$ of
\[
\wt\omega_j=2^{1-2j}t^{2j-1}
\rd t \wedge \dfrac{\rd y_1}{y_1} \wedge \cdots \wedge \dfrac{\rd y_k}{y_k}
\]
in $\coH^{k+1}_{\dR}(U_0,\wt f_k)$ and set
\begin{equation}\label{eq:wj}
w_j=\frac{1}{2k!}\sum_{\sigma \in \symgp_k\times \mu_2}
	\sgn(\sigma) \cdot \sigma([\wt\omega_j])
	\in\coH^{k+1}_{\dR}(U_0,\wt f_k)^{\symgp_k \times \mu_2,\,\sgn}.
\end{equation}
Through the isomorphism (\cf\cite[Prop.\,2.13]{F-S-Y18})
\[
\coH^1_\dR(\Gm,\Sym^k\Kl_2)\simeq\coH^{k+1}_{\dR}(U,\wt f_k)^{\symgp_k \times \mu_2,\,\sgn},
\]
each class $[z^jv_0^k\rd z/z]$ in the basis $\Basis_k$ corresponds to the class~$w_j$ (\cf \cite[Proof of Prop.\,4.21]{F-S-Y18}). As a consequence, $(w_j)_{0\leq j\leq k'}$ forms a basis of $\coH^{k+1}_{\dR}(U_0,\wt f_k)^{\symgp_k \times \mu_2,\,\sgn}$ satisfying
$\iota^*w_j=w_j$.

\subsubsection*{A family of rapid decay cycles}
For each $i$ such that $0\leq i\leq k'$, we define rapid decay cycles $\wt\alpha_i$
as follows.
First, let $\wt\alpha_0$ be the bounded cycle
\[\textstyle \wt\alpha_0 = \prod_{p=0}^k c_p, \]
where $c_0$ is the circle $|t|=1$ and $c_p$ is the circle $|y_p|=1$ for $p=1, \ldots, k$, both oriented counterclockwise. For the remaining unbounded cycles,
we start with the chain
\[\textstyle \wt\alpha'_i = [1, +\infty) \times (\prod_{p=1}^i c_p) \times (\RR_{>0})^{k-i}
\]
on $U(\CC)$, where $[1, +\infty)$ is relative to the variable~$t$ and the open octant $(\RR_{>0})^{k-i}$ is relative to the variables $y_p$ for $p=i+1, \dots, k$. The orientation is the natural one on each half-line, and given by the order of the variables for $\wt\alpha'_i$.
Since $\mathrm{Re}(y_p+1/y_p) \geq -2$ on $c_p$ and $\geq 2$ on $\RR_{>0}$,
by noting that $k-i>i$ for $1 \leq i \leq k'$,
one sees that $e^{-\wt{f}_k}$ has rapid decay along $\wt \alpha_i'$. The rapid decay chain~$\wt\alpha_i'$ has boundary
\[\textstyle
\pt\wt\alpha_i'= -\{1\} \times (\prod_{p=1}^i c_p) \times (\RR_{>0})^{k-i},
\] and hence is not a cycle. To kill the boundary, we mimic the construction of the cycles $\alpha_i$ from~\eqref{eq:alphai} by introducing, for each $r$ such that $1 \leq r \leq k'$, the chain $\wt\gamma_r$ defined as the image of the map~\hbox{$[0,1]^r \times (0,1)^{k-r+1} \to U(\CC)$} given~by
\[
(s_0, \ldots, s_k) \mto
\begin{cases}
t=e^{2\pii s_0}, & \\
y_p=e^{2\pii s_p}, & 1\leq p\leq r-1, \\
y_p=\begin{cases}
e^{2\pii s_0} \cdot 4s_p, & 0<s_p\leq 1/4, \\
e^{2\pii s_0(2-4s_p)}, & 1/4\leq s_p\leq 3/4, \\
e^{-2\pii s_0}\big(1+\tan\frac{\pi}{2}(4s_p-3)\big), & 3/4\leq s_p<1,
\end{cases} & r\leq p\leq k. \\
\end{cases}\]
(See Figure \ref{fig:gamma_r}.)
The chain~$\wt\gamma_r$ is rapidly decay for $e^{-\wt{f}_k}$
with boundary
\[\textstyle
\pt\wt\gamma_r= \{1\} \times (\prod_{p=1}^{r-1} c_p)
	\times \prod_{p=r}^k (\RR_{>0} -2c_p)
	- \{1\} \times (\prod_{p=1}^{r-1} c_p) \times (\RR_{>0})^{k-r}.
\]

\begin{figure}[htb]
\begin{center}
\includegraphics[scale=1]{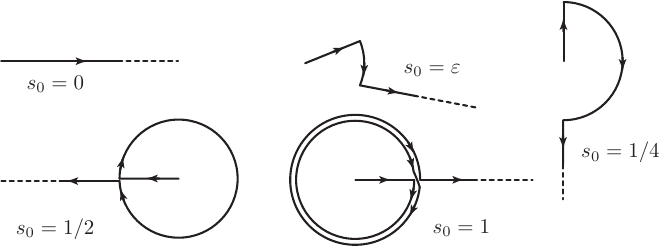}
\caption{The path $s_p\mto y_p(s_0,s_p)$ for $s_0$ fixed and $r\leq p\leq k$.}\label{fig:gamma_r}
\end{center}
\end{figure}

For each integer $m \geq 1$,
there is a unique sequence $(\theta_m(r))_{0 \leq r \leq m-1}$
of rational numbers
satisfying the identity
\begin{equation}\label{eq:theta-coeff}
\sum_{r=0}^{m-1} \theta_m(r)a^r\big[ (a+b)^{m-r}-b^{m-r} \big]
= ab^{m-1}
\end{equation}
in the polynomial ring $\QQ[a,b]$.
Explicitly, the first two values are
$\theta_m(0) = 1/m$ and $\theta_m(1) = -1/2$
and the recursion
\begin{equation}\label{eq:binom-theta}
\sum_{r=0}^{\ell-1} \binom{m-r}{m-\ell}\theta_m(r) = 0
\end{equation} holds if $1<\ell <m$. Set
\[ \wt\theta_m = \sum_{r=0}^{m-1}\frac{\theta_m(r)}{m-r+1}. \]

For each $i$ such that $1\leq i\leq k'$, let
\[ \wt\alpha_i = -(-2)^{k-i}\wt\theta_{k-i+1}\wt\alpha_0
+ \frac{1}{(k-i+1)!}\sum_{\sigma\in\symgp_{k-i+1}}\sigma
\Big(\wt\alpha'_i +\sum_{r=i}^k
	(-2)^{r-i-1}\theta_{k-i+1}(r-i)\wt\gamma_r \Big), \]
where the symmetric group $\symgp_j$
acts by permuting the last $j$ components.
We compute the boundary of the terms involving $\wt\gamma_r$.
Since the average over the group action
makes the positions of circles and lines from $\partial\wt\gamma_r$
equidistributed in the last $k-i+1$ components,
for each fixed $j$ such that $1\leq j\leq k-i$
and any $\sigma\in\symgp_{k-i+1}$,
the coefficient in the boundary of the term
$\sigma\big(\{1\}\times \prod_{p=1}^{i+j-1} c_p \times \prod_{p=i+j}^k \RR_{>0}\big)$,
resulting from $\wt\gamma_r$ for $i\leq r\leq i+j-1$,
is equal to
\begin{align*}
\sum_{r=i}^{i+j-1}
	(-2)^{r-i-1}\theta_{k-i+1}(r-i)\binom{k+1-r}{i+j-r}(-2)^{i+j-r}
&= (-2)^{j-1}\sum_{\ell=0}^{j-1}
	\theta_{k-i+1}(\ell)\binom{k-i+1-\ell}{j-\ell} \\
&= \begin{cases}
1, & j = 1, \\
0, & 1<j< k-i+1,
\end{cases}
\end{align*}
where the last equality follows from \eqref{eq:binom-theta}.
One concludes that the $\wt\alpha_i$ define rapid decay cycles.

\begin{remark}
In fact,
$\theta_m(r) = \binom{m}{r}\sfrac{\Bern{r}}{m}$ if $0\leq r<m$,
and
$\wt\theta_m = \sfrac{-\Bern{m}}{m}$
for all $m\geq 1$.
To see this,
plugging $a=1$ into \eqref{eq:theta-coeff}
and summing the resulting equations for $b=1,\dots, n$,
one obtains
\[ \sum_{\ell=0}^{m-1} n^{m-\ell}\sum_{r=0}^\ell\binom{m-r}{m-\ell}\theta_m(r)
= \sum_{b=1}^n b^{m-1}, \]
and hence by Bernoulli's formula \eqref{eq:BernoulliFormula},
\[ \sum_{r=0}^\ell \binom{m-r}{m-\ell}\theta_m(r)
= \frac{(-1)^\ell}{m}\binom{m}{\ell}\Bern{\ell},
\quad
0\leq\ell <m. \]
Using \eqref{eq:binom-theta} and the equalities
$\theta_m(0) = \Bern{0}/m, \theta_m(1) = \Bern{1}$
and $\Bern{2p+1}=0$ for $p\geq 1$,
we obtain the identity for $\theta_m(r)$. Therefore,
\[
\wt\theta_m
= \frac{1}{m}\sum_{r=0}^{m-1}\frac{1}{m-r+1}\binom{m}{r}\Bern{r}
= \frac{1}{m(m+1)}\sum_{r=0}^{m-1} \binom{m+1}{r}\Bern{r},
\]
and the identity for $\wt\theta_m$
follows from the recursive relation $\sum_{r=0}^m \binom{m+1}{r}\Bern{r} = 0$ for all $m\geq 1$.
\end{remark}

\begin{lemma}For any positive integer $m$,
the following identity holds in $\QQ[a,b]$:
\begin{equation}\label{eq:theta-tilde}
\sum_{r=0}^{m-1}\frac{\theta_m(r)}{m-r+1}a^{r-1}\big[
	(a+b)^{m-r+1}-b^{m-r+1}\big]
= \wt\theta_ma^m + \frac{b^m}{m}.
\end{equation}
\end{lemma}

\begin{proof}
Integrating the equation \eqref{eq:theta-coeff}
with respect to $b$ from $0$ to $b$
yields the formula.
\end{proof}

\begin{prop}\label{prop:geometricpairing}
The period pairing
\[
\Ppairing^{\rrd,\rmod}:
\coH^{\rrd}_{k+1}(U, \wt{f}_k) \otimes
\coH_\dR^{k+1}(U, \wt{f}_k) \to \CC
\]
is given on the rapid decay cycles $\wt\alpha_i$ and the differential forms $\tilde \omega_j$ by
\[ \Ppairing^{\rrd,\rmod}(\wt\alpha_i,\wt\omega_j) = \begin{cases}
2(2\pii)^{k+1}\delta_{0,j} & \text{if $i=0$}, \\
(-1)^{k-i}\IKM^\reg_k(i,-1) & \text{if $j=0$ and $1\leq i\leq k'$}, \\
(-1)^{k-i}\IKM_k(i,2j-1) & \text{if $1\leq i,j\leq k'$}.
\end{cases} \]
\end{prop}

\begin{proof}
It is clear that
$\Ppairing^{\rrd,\rmod}(\wt\alpha_0,\wt\omega_j)
= 2(2\pii)^{k+1}\delta_{0,j}$.

For the rest,
as in the proof of Proposition \ref{prop:shrinking-reg},
for $\epsilon\in\RR_{>0}$,
we let
$\wt\alpha'_{i,\epsilon} = \epsilon\wt\alpha'_i$ and
$\wt\gamma_{r,\epsilon} = \epsilon\wt\gamma_r$
be the scalings of $\wt\alpha'_i$ and $\wt\gamma_r$,
which are homologous to $\wt\alpha'_i$ and $\wt\gamma_r$, respectively.
Consider the case $j=0$.
By the limiting behavior \eqref{eq:BesselIKzero}
and Definition \ref{defi:IKMreg} of the regularized Bessel moments,
we have the limiting behaviors
\begin{align*}
\int_{\wt\alpha'_{i,\epsilon}} e^{-\wt{f}_k}\wt\omega_0
&= 2^{k+1}(\pii)^i\int_\epsilon^\infty I_0(t)^iK_0(t)^{k-i}\frac{\de t}{t} \\
&%\hspace*{-2mm}
\sim_{\varepsilon \to 0^+} (-1)^{k-i}\IKM^\reg_k(i,-1)
+ \frac{(-1)^{k-i+1}2^{k+1}(\pii)^i}{k-i+1}
	\big(\gamma+\log(\sfrac{\epsilon}{2})\big)^{k-i+1}, \\
\int_{\wt\gamma_{r,\epsilon}} e^{-\wt{f}_k}\wt\omega_0
&= 2^{k+1}(\pii)^{r-1}\int_{|t|=\epsilon} I_0(t)^{r-1} K_0(t)^{k-r+1}\frac{\de t}{t} \\
&%\hspace*{-2mm}
\sim_{\varepsilon \to 0^+} \frac{(-1)^{k-r+1}2^{k+1}(\pii)^{r-1}}{k-r+2}\cdot\Big[\big(2\pii+\gamma+\log(\sfrac{\epsilon}{2})\big)^{k-r+2}
- \big(\gamma+\log(\sfrac{\epsilon}{2})\big)^{k-r+2}\Big].
\end{align*}
Substituting $m=k-i+1$, $a = 2\pii,$ and $b= \gamma+\log\sfrac{\epsilon}{2}$
into the identity \eqref{eq:theta-tilde},
one obtains
\begin{multline*}
\sum_{r=i}^k (-2)^{r-i-1}\theta_{k-i+1}(r-i)
	\int_{\wt\gamma_{r,\epsilon}} e^{-\wt{f}_k}\wt\omega_0 \\[-5pt]
\sim_{\varepsilon \to 0^+} -(-2)^{k-i+1}(2\pii)^{k+1}\wt\theta_{k-i+1}
	+ \frac{(-1)^{k-i}2^{k+1}(\pii)^i}{k-i+1}
	\big(\gamma+\log(\sfrac{\epsilon}{2})\big)^{k-i+1}.
\end{multline*}
Letting $\epsilon\to 0^+$ leads to
$\Ppairing^{\rrd,\rmod}(\wt\alpha_i,\wt\omega_0)
= (-1)^{k-i}\IKM^\reg_k(i,-1)$.

Finally, if $1\leq i,j\leq k'$,
it is straightforward to verify that
\[
\lim_{\epsilon\to 0^+}\int_{\wt\alpha'_{i,\epsilon}} e^{-\wt{f}_k}\wt\omega_j
= (-1)^{k-i}\IKM_k(i,2j-1) \quand
\lim_{\epsilon\to 0^+}\int_{\wt\gamma_{i,\epsilon}} e^{-\wt{f}_k}\wt\omega_j
= 0
\]
(similar to the proof of Proposition \ref{prop:Pij}), and the remaining case follows.
\end{proof}

Starting from the $\wt\alpha_i$, we obtain the following cycles in $\coH^\rrd_{k+1}(U,\wt f_k)^{\symgp_k \times \mu_2,\,\sgn}$:\begin{align*}
\tau_0&=\frac{1}{k!}\sum_{\sigma \in \symgp_k \times \mu_2} \sgn(\sigma) \cdot \sigma(\wt \alpha_0), \\
\tau_i&=\frac{(-1)^{k-i}}{2k!}\sum_{\sigma \in \symgp_k \times \mu_2} \sgn(\sigma) \cdot \sigma(\wt \alpha_i)
\quad (1\leq i\leq k').
\end{align*}
For each $i$ such that $1 \leq i \leq k'$,
let $\Gamma_i$ denote the image of $\tau_i$
in $\coH^\rmod_{k+1}(U,\wt f_k)^{\symgp_k \times \mu_2,\,\sgn}$.

\begin{cor}\label{thm:Bettirealization}\mbox{}
\begin{enumerate}
\item\label{Qstr:1}
The family $(\tau_i)_{0 \leq i \leq k'}$ forms a basis of $\coH^\rrd_{k+1}(U,\wt f_k)^{\symgp_k \times \mu_2,\,\sgn}$.

\item\label{Qstr:2}
The period realizations
\begin{itemize}
\item
of $\coH^1_\rc(\Gm,\Sym^k\Kl_2)$ and of $\coH_\rc^{k+1}(U_0,\wt f_k)^{\symgp_k \times \mu_2,\,\sgn}$,
\item
as well as those of $\coH^1(\Gm,\Sym^k\Kl_2)$ and $\coH^{k+1}(U_0,\wt f_k)^{\symgp_k \times \mu_2,\,\sgn}$,
\end{itemize}
are isomorphic.

\item\label{eq:Betti1} If $k$ is not a multiple of $4$, the family $(\Gamma_i)_{1 \leq i \leq k'}$ forms a basis of $(\Motive_k)_{\Betti}$. If~$k$ is a multiple of $4$, the $\QQ$-linear relation
\[
\sum_{i=0}^{\flr{(k-2)/4}} \binom{k/2}{2i+1}\Gamma_{2i+1} = 0
\]
holds and $(\Gamma_i)_{2 \leq i \leq k'}$ forms a basis of $(\Motive_k)_{\Betti}$. In particular, the period realization of the motive $\Motive_k$ is isomorphic to
\[
(\coH^1_{\dR, \rmid}(\Gm, \Sym^k \Kl_2)_\QQ, \coH_1^{\rmid}(\Gm, \Sym^k \Kl_2)_\QQ, \Ppairing^\rmid_{k}).
\]
\item\label{eq:Betti2} The involution $F_\infty$ acts as $F_\infty(\Gamma_i)\!=\!(-1)^i\Gamma_i$ for all $i$ such that \hbox{$1 \leq i \leq k'$}.
\end{enumerate}
\end{cor}

\begin{proof}
For each $i$ such that $0 \leq i \leq k'$, equation \eqref{eq:P0j}, Propositions \ref{prop:Pij}, \ref{prop:shrinking-reg} and \ref{prop:geometricpairing} yield the equality of periods $\Ppairing^{\rrd,\rmod}(\tau_i, w_j)=\Ppairing^{\rrd,\rmod}_k(\alpha_i, \omega_j)$ for all $j$ such that $0 \leq j \leq k'$ (with $w_j$ defined by \eqref{eq:wj}), from which~\eqref{Qstr:1} follows, as indicated in Step \ref{summary:1} of the summary. The proof of \eqref{Qstr:2} has been explained in Steps \ref{summary:2} and \ref{summary:3} of the summary.

In view of Theorem \ref{th:Bettibasis}\eqref{th:Bettibasis3} and Proposition \ref{prop:2-chain}, to prove \eqref{eq:Betti1}, it suffices to show that $\Ppairing^\rmid(\Gamma_i, w_j)=\Ppairing^\rmid_{k}(\beta_i, \omega_j)$ for all $i,j$ such that $1 \leq i, j \leq k'$. Since $\Gamma_i$ and $\beta_i$ are the images of the rapid decay cycles $\tau_i$ and $\alpha_i$ respectively and the pairing $\Ppairing^\rmid_{k}$ is induced by $\Ppairing^{\rrd,\rmod}$, this amounts to~$\Ppairing^{\rrd,\rmod}(\tau_i, w_j)=\Ppairing^{\rrd,\rmod}_k(\alpha_i, \omega_j)$, which we just checked.

Finally, the statement about $F_\infty$ follows from the fact that the period pairing $\Ppairing^{\rrd,\rmod}$ is compatible with complex conjugation in that (recall that $\iota^*w_j=w_j$ and $\conj^*w_j=w_j$),
\[
\Ppairing^{\rrd,\rmod}(\conj_*\tau_i, w_j)=\int_{F_\infty\tau_i}e^{-\wt f_k}w_j=\int_{\conj_*\tau_i}\conj^*(e^{-\wt f_k}w_j)
=\overline{\Ppairing^{\rrd,\rmod}(\tau_i, w_j)}\quad\text{for all $i,j$},
\]
along with the equality $\overline{\Ppairing^{\rrd,\rmod}(\tau_i, w_j)}=(-1)^i\Ppairing^{\rrd,\rmod}(\tau_i, w_j)$ if $1\leq i,j\leq k'$, since $\IKM_k(i,2j-1)$ is real for even $i$ and purely imaginary for odd $i$.
\end{proof}

\section{\texorpdfstring{$L$}{L}-functions of Kloosterman sums and Bessel moments}\label{subsec:Deligne}
In this section, we make precise the statement of Deligne's conjecture for the motive $\Motive_k$. Being a classical motive, $\Motive_k$ has an associated $L$\nobreakdash-function $L(\Motive_k, s)$, which is identified in \cite[Th.\,5.8 \& 5.17]{F-S-Y18}
with the $L$-function $L_k(s)$ of symmetric power moments of Kloosterman sums, and its Betti realization $(\Motive_k)_\Betti$ carries a pure Hodge structure of weight $k+1$. The factor at infinity was computed in \cite[Cor.\,5.30]{F-S-Y18}:
\begin{equation}\label{eqn:shapeGamma}
L_\infty(\Motive_k, s)=\pi^{-\sfrac{ms}{2}}\prod_{j=1}^m \Gamma\Bigl(\dfrac{s-j}{2}\Bigr),
\end{equation} where $m=k'$ if $k$ is not a multiple of $4$ and $m=k'-1$ otherwise. This also follows from Corollary~\ref{thm:Bettirealization}\eqref{eq:Betti2}: indeed, according to~\cite[Th.\,1.8]{F-S-Y18},
the only non-trivial Hodge numbers of~$\Motive_k$ are
\[
h^{p, q}=\begin{cases} 1 & \text{for $p=2, \ldots, k-1$ if $k$ is odd} \\
1 & \text{for $\min\{p, q\}=2, \ldots, 2\flr{\psfrac{k-1}{4}}$ if $k$ is even.}
\end{cases}
\]
The ordered basis $\Basis_{k,\rmid}$ is adapted to the Hodge filtration by Corollary~\ref{cor:compatibilityHodge}, and~$F_\infty$~acts as~$-1$ on $H^{p, p}$ for $k \equiv 3\bmod{4}$, which is what we need for the conclusion.

By \cite[Th.\,1.2 \& 1.3]{F-S-Y18}, the $L$-function $L(\Motive_k, s)$ extends meromorphically to the complex plane and the completed $L$\nobreakdash-function
\[
\Lambda(\Motive_k, s)=L(\Motive_k, s)L_\infty(\Motive_k, s)
\]
satisfies a functional equation relating its values at $s$ and~\hbox{$k+2-s$.} The \emph{critical integers} are the integral values of $s$ at which neither $L_\infty(\Motive_k, s)$ nor \hbox{$L_\infty(\Motive_k, k+2-s)$} has a pole and the \emph{critical values} are the values of $L(\Motive_k, s)$ at critical integers.

Deligne's conjecture \cite[\S1]{Deligne} predicts that critical values agree, up to a rational factor, with the determinants of certain minors of the period matrix, which are defined as follows. Let $a$ be an integer, and let $\Motive_k(a)_\Betti^+$ and $\Motive_k(a)_\Betti^-$ denote respectively the invariants and anti-invariants of~$F_\infty$ acting on~$\Motive_k(a)_\Betti$. As~$F_\infty$ exchanges the subspaces $H^{p-a, q-a}$ and $H^{q-a, p-a}$ in the Hodge decomposition and acts as~$-1$ on~$H^{p-a, p-a}$, the eigenspaces for $F_\infty$ have dimensions either $\sum_{p >q} h^{p-a, q-a}$ or $\sum_{p \geq q} h^{p-a, q-a}$, and there exists unique steps $F^{\pm}\Motive_k(a)_{\dR}$ of the Hodge filtration with $\dim F^{\pm}\Motive_k(a)_{\dR}=\dim \Motive_k(a)_\Betti ^{\pm}.$ If $n$ is a critical integer, Deligne defines
\[
c_n=\det (\Ppairing^\rmid_{k}(\sigma_i, \nu_i) ) \in \CC^\times \slash \QQ^\times
\] where $\sigma_i$ runs through any basis of the $\QQ$-linear dual of $\Motive_k(k+1-n)^+_\Betti$ and $\nu_i$ runs through any basis of $F^+\Motive_k(k+1-n)_{\dR}$, \cf the paragraph before \cite[Conj.\,1.8]{Deligne} taking the duality of pure Hodge structures~$(\Motive_k)_\rH^\vee \cong (\Motive_k)_\rH(k+1)$ from \cite[(3.4)]{F-S-Y18} into account.
He then conjectures that $L(\Motive_k, n)$ is a rational multiple of $c_n$.

\begin{notation} For each $k \geq 3$, the determinants of Bessel moments $D_{k,\odd}$ and $D_{k,\even}$ are defined by the following formulas:
\begin{equation}\label{eqn:Deligne}
\begin{aligned}
D_{k,\odd}&=\det\begin{pmatrix} \dpl\int_0^\infty I_0(t)^{2i-1} K_0(t)^{k+1-2i} t^{2j-1} \,\de t \end{pmatrix}_{1 \leq i, j \leq \flr{\psfrac{k+1}{4}}} \\
D_{k,\even}&=\det\begin{pmatrix} \dpl\int_0^\infty I_0(t)^{2i} K_0(t)^{k-2i} t^{2j-1} \,\de t
\end{pmatrix}_{1 \leq i, j \leq \flr{\sfrac{k}{4}}}.
\end{aligned}
\end{equation}
In other words, $D_{k,\odd}$, \resp $D_{k,\even}$, is obtained by extracting from $\Ppairing_k^\rmid$ the entries that belong to odd, \resp even, lines and to the first $\flr{\psfrac{k+1}{4}}$, \resp $\flr{\sfrac{k}{4}}$, columns.
If $k$ is a multiple of $4$, we let $D_{k,\odd}'$ and $D_{k,\even}'$ be the determinants of the same matrices except that we remove the last row and column, \ie those indexed by $i=j=\sfrac{k}{4}$.
\end{notation}

\begin{thm}
For $k=3$ or $k\geq 5$,
the critical integers $n$ for $L(\Motive_k, s)$ and the corresponding values $c_n$ are given by Table \ref{table:critical} below.
\begin{table}[htb]
\renewcommand{\arraystretch}{1.2}
\begin{center}
\begin{tabular}{|c|c|c|c|}
\hline
$k$ & rank & critical integers $n$ & $c_n$ \\
\hline\rule{0pt}{3ex}
\multirow{2}{*}{$3$} & \multirow{2}{*}{$1$} & $2-2a$ & $1$ \\
& & $2a+3$ & $\pi^{2a} D_{3,\odd}$
\\
\hline\rule{0pt}{3ex}
\multirow{2}{*}{$4r + 1$} & \multirow{2}{*}{$2r$}
& $2r + 1$ & $\pi^{-r(r+1)}D_{k,\odd}$ \\
&& $2r + 2$
& $\pi^{-r(r-1)}D_{k,\even}$ \\
\hline\rule{0pt}{3ex}
\multirow{3}{*}{$4r+2$} & \multirow{3}{*}{$2r$}
& $2r + 1$ & $\pi^{-r(r+1)}D_{k,\even}$ \\
&& $2r + 2$
& $\pi^{-r(r+1)}D_{k,\odd}$ \\
&& $2r + 3$
& $\pi^{-r(r-1)}D_{k,\even}$ \\
\hline\rule{0pt}{3ex}
\multirow{2}{*}{$4r + 3$} & \multirow{2}{*}{$2r+1$}
& $2r + 2$ & $\pi^{-r(r+1)}D_{k,\even}$ \\
&& $2r + 3$
& $\pi^{-r(r+1)}D_{k,\odd}$ \\
\hline\rule{0pt}{3ex}
\multirow{5}{*}{$4r+4$} & \multirow{5}{*}{$2r$}
& $2r + 1$ & $\pi^{-r(r+3)}D'_{k,\even}$ \\
&& $2r + 2$
& $\pi^{-r(r+3)}D'_{k,\odd}$ \\
&& $2r + 3$
& $\pi^{-r(r+1)}D'_{k,\even}$ \\
&& $2r + 4$
& $\pi^{-r(r+1)}D'_{k,\odd}$ \\
&& $2r + 5$
& $\pi^{-r(r-1)}D'_{k,\even}$ \\
\hline
\end{tabular}
\vspace{5pt}
\caption{\label{table:critical}Critical integers $n$ and values of $c_n$ ($r\geq 1$)}
\end{center}
\end{table}
\end{thm}

\begin{proof}
The case $k=3$ is exceptional in that there exist infinitely many critical integers, namely all even integers $\leq 2$ and all odd integers $\geq 3$. For each $k \geq 5$, a straightforward computation using~\eqref{eqn:shapeGamma} yields the list of critical integers. To compute~$c_n$, we first exhibit a basis of the $\QQ$\nobreakdash-linear dual of $\Motive_k(a)_\Betti ^+$ by means of Corollary \ref{thm:Bettirealization}. If $k$ is not a multiple of $4$, a basis is given by~$\{(2\pi \sfi)^{-a} \Gamma_{2i}\}_{1 \leq i \leq \flr{\sfrac{k'}{2}}}$ if~$a$ is even and by~$\{(2\pi \sfi)^{-a} \Gamma_{2i-1}\}_{1 \leq i\leq \flr{\psfrac{k'+1}{2}}}$ if $a$ is odd. If $k$ is a multiple of $4$, say $k=4r+4$, then the~$\QQ$\nobreakdash-linear dual of~$\Motive_k(a)_\Betti ^{+}$ has basis $\{(2\pi \sfi)^{-a} \Gamma_{2i}\}_{1 \leq i \leq r} $ if $a$ is even and $\{(2\pi \sfi)^{-a} \Gamma_{2i+1}\}_{1 \leq i \leq r}$ if $a$ is odd, which shows that $\Motive_k(a)^+_\Betti $ always has dimension $r$.

Let us treat the case $k\!=\!4r\!+\!3$ in detail. For $n\!=\!2r\!+\!2$, the eigenspace~\hbox{$\Motive_k(2r\!+\!2)_\Betti ^+$} has dimension $r$, and hence $F^+\Motive_k(2r+2)_{\dR}$ is spanned by $\{w_j\}_{1 \leq j \leq r}$. Therefore, $c_{2r+2}$ is the determinant of the matrix with entries
\[
(2\pi \sfi)^{-(2r+2)}\Ppairing^{\rrd,\rmod}(\Gamma_{2i},w_j) \sim_{\QQ^\times} \pi^{2(i-r-1)} \int_0^\infty\hspace{-2mm} I_0(t)^{2i} K_0(t)^{k-2i} t^{2j-1} \,\de t,
\] from which we get $c_{2r+2}=\pi^{-r(r+1)}D_{k,\even}$. For $n=2r+3$, the eigenspace \hbox{$\Motive_k(2r+1)_\Betti ^+$} has dimension~$r+1$, and hence $F^+\Motive_k(2r+1)_{\dR}$ is spanned by $\{w_j\}_{1 \leq j \leq r+1}$. With respect to these bases, the matrix defining $c_{2r+3}$ has entries
\[
(2\pi \sfi)^{-(2r+1)}\Ppairing^{\rrd,\rmod}(\Gamma_{2i-1},w_j) \sim_{\QQ^\times} \pi^{2(i-r-1)}\int_0^\infty I_0(t)^{2i-1} K_0(t)^{k+1-2i} t^{2j-1}\,\de t,
\] which gives $c_{2r+3}=\pi^{-r(r+1)}D_{k,\odd}$. The cases $k=4r+1$ and $k=4r+2$ are completely parallel.

For $k=4r+4$ and any critical value $n$, the eigenspace $\Motive_k(4r+5-n)_\Betti ^+$ has dimension $r$, and hence $F^+\Motive_k(4r+5-n)_{\dR}$ is spanned by $\{w_j\}_{1 \leq j \leq r}$ (note that $r<k/4$, so that we do not need to modify the $\omega$'s). If $n$ is odd, $c_n$ is the determinant of the matrix with entries
\[
(2\pi \sfi)^{-(4r+5-n)}\Ppairing^{\rrd,\rmod}(\Gamma_{2i},w_j) \sim_{\QQ^\times} \pi^{2i+n-4r-5}\int_0^\infty I_0(t)^{2i} K_0(t)^{k-2i} t^{2j-1}\,\de t,
\] thus yielding
\[
c_{2r+1}=\pi^{-r(r+3)}D'_{k,\even},\ \ c_{2r+3}=\pi^{-r(r+1)}D'_{k,\even},\ \text{ and }\ c_{2r+5}=\pi^{-r(r-1)}D'_{k,\even}.
\]
If $n$ is even, the matrix has entries
\[
(2\pi \sfi)^{-(4r+5-n)}\Ppairing^{\rrd,\rmod}(\Gamma_{2i+1},w_j).
\] Thanks to the linear relation from Corollary \ref{thm:Bettirealization}\eqref{eq:Betti1}, the determinant of this matrix agrees up to a rational number with that of
\[
\pi^{2i+n-4r-6} \int_0^\infty I_0(t)^{2i-1} K_0(t)^{k+1-2i} t^{2j-1}\,\de t,
\] which gives the remaining values
\[
c_{2r+2}=\pi^{-r(r+3)}D_{k,\odd}'\quad\text{and}\quad c_{2r+4}=\pi^{-r(r+1)}D_{k,\odd}'.\qedhere
\]
\end{proof}

\begin{remark} Besides the case $k=3$, in which $L_3(s)=L(\chi_3, s-2)$ is a shifted Dirichlet $L$\nobreakdash-function, Deligne's conjecture holds for $k=5, 6, 8$. In all three cases,~$L_k(s)$ is the $L$-function of a modular form evaluated at $s-2$ (\cf \cite[Table\,1]{F-S-Y18} and the references therein) and the matrices in \eqref{eqn:Deligne} have size one. The critical values have been expressed in terms of Bessel moments in~\cite{BKV, wick, ZhouQind}, thus confirming the conjecture. The first case where a true determinant is expected to occur is~$L_7(5)$, which is a critical value of the $L$-function of the symmetric square of a modular form. Other numerical confirmations of Deligne's conjecture for $k \leq 24$ appear in \cite{BKLF}.
\end{remark}

\begin{remark}
Broadhurst and Roberts conjecture that, for $k \geq 12$ a multiple of $4$, the $L$-function $L(\Motive_k, s)$ always vanishes at the central point $s=\psfrac{k+2}{2}$, even when the expected sign of the functional equation (see \cite[below Th.\,1.3]{F-S-Y18}) is positive. According to Beilinson's conjecture, this vanishing should be explained
by the existence of certain non-trivial extension of
$\Motive_k$ by $\QQ(-\sfrac{k}{2})$. It seems possible to construct such an extension by considering the quotient of the cohomology with compact support of $\Sym^k \Kl_2$ by its weight zero piece and to relate its non-splitting to the shape of the full period matrix $\Ppairing^{\rmod,\rrd}_k$ from Section~\ref{sect:complete_periods},
Indeed,
suppose $\kfour$.
The Nori motive
$\wt\Motive_k = (\coH_\cp^{k-1}(\KM_0)/W_0)^{\symgp_k\times\mu_2,\sgn}(-1)$,
where $W_0$ denotes the weight zero part,
is an extension of $\Motive_k$ by a rank-one motive $E$ of weight $k$.
It has been shown in the proof
of \cite[Th.\,5.17]{F-S-Y18}
(see especially the discussion of the terms
$(E_\infty^{1,k-2})_{\FF_p}^{G,\chi}$
and $\ker(\beta)^{G,\chi}/W_0$)
that as a representation of $\Gal(\ol\QQ_p/\QQ_p)$,
the $\ell$-adic \'etale realization %$E_{\et,\ell}$
of $E$
is isomorphic to $\QQ_\ell(-k/2)$
for all primes $p\geq 3$ and $\ell\neq p$.
Since the Frobenius elements of characteristic $p\geq 3$ are dense
in $\Gal(\ol\QQ/\QQ)$,
one obtains an isomorphism $E \simeq \QQ(-k/2)$.
From the viewpoint of the period realization,
the non-splitting of the extension $\wt\Motive_k$
amounts to the claim that
the entries of the last row of the period matrix $\Ppairing^{\rmod,\rrd}_k$
span a $\QQ$-space of dimension at least two in $\CC$
by Proposition \ref{prop:B4}\eqref{prop:B42}.
Concretely, the latter holds if and only if the moments
$\IKM_k(k/2, 2j-1)$ and $\IKM^\reg_k(k/2,k'+2j)$, for
$1\leq j< k/4$,
are not all in $(2\pii)^{k/2}\QQ$
(rational multiples of the entry $\Ppairing^{\rmod,\rrd}_{k;k/2,k/4}$).
Such non-splitting has been verified numerically
for the period matrix $\Ppairing_k$ of $\wt\Motive_k^\vee(-k-1)$
by Broadhurst and Roberts, see \cite[\S 2]{Roberts17}.
\end{remark}

\appendix
\section*{Appendix. Period realization of an exponential mixed Hodge structure, by~Claude~Sabbah}
\refstepcounter{section}
This appendix can be regarded as a complement, with respect to period structures, to the appendix in \cite{F-S-Y18}. We consider the abelian category $\catPer$ of \emph{period structures}, whose objects consist of pairs $(V^\CC,V_\QQ)$ consisting of a finite dimensional $\CC$-vector space $V^\CC$, a finite dimensional $\QQ$-vector space $V_\QQ$, together with an isomorphism
\[
\per:\CC\otimes_\QQ V_\QQ\isom V^\CC
\]
and whose morphisms are the natural ones. There is a natural forgetful functor $\Per:\MHS\to\catPer$ from the category of $\QQ$-mixed Hodge structures to that of period structures.

\subsection{The fiber period realization of an exponential mixed Hodge structure}
\label{subsec:FT_EMHS}

Let $X$ be a complex smooth quasi-projective variety
and let $N^\rH$ be an object of the abelian category $\MHM(X)$ of $\QQ$-mixed Hodge modules on $X$. It consists of a triple $((N,F^\cbbullet N),(\cF_\QQ,W_\bbullet\cF_\QQ),\per_X)$, where
\begin{itemize}
\item
$(N,F^\cbbullet N)$ is a holonomic $\cD_X$-module endowed with a coherent filtration,
\item
$(\cF_\QQ,W_\bbullet\cF_\QQ)$ is a $\QQ$-perverse sheaf on $X^\an$ endowed with an increasing filtration by $\QQ$\nobreakdash-perverse subsheaves,
\item
and the \emph{period isomorphism} $\per_X$ is an isomorphism $\CC\otimes_\QQ\cF_\QQ\isom\pDR N$, where $\pDR N$ is the shifted de~Rham complex $\DR N[\dim X]$.
\end{itemize}
These data are subject to various compatibility relations that we do not make explicit here, referring to \cite{MSaito87, MSaito17} for details. From the mixed Hodge module $N^\rH$ we only retain the triple $\Per(N^\rH):=(N,\cF_\QQ,\per_X)$ by forgetting the Hodge and weight filtrations.

For example, let us consider the pure Hodge module $\pQQ_X^\rH$, with underlying $\cD_X$-module equal to $\cO_X$ and underlying perverse sheaf $\pQQ_X=\QQ_X[\dim X]$. The period isomorphism is induced by the isomorphism $\CC_X=\cH^0\DR\cO_X\isom\DR\cO_X$.

Let $\Afu_\ts$ be the affine line with coordinate $\ts$. The $\QQ$-linear neutral Tannakian category $\EMHS$ (\emph{exponential mixed Hodge structures}), as defined in \cite[\S4]{K-S10}, is the full subcategory of $\MHM(\Afu_\ts)$ consisting of objects~$N^\rH$ whose underlying perverse sheaf has vanishing global cohomology, with tensor structure given by the additive convolution $\star$. Denoting by $j:\Gm \to \Afu$ the inclusion, one defines a projector
\[ \Pi: \MHM(\Afu) \to \EMHS,\quad N^\rH\mto N^\rH\star \Hm{j}_!\pQQ^\rH_{\Gm},
\]
consisting in neglecting constant mixed Hodge modules on $\Afu$.

For an object $N^\rH$ in $\EMHS$, its \emph{de~Rham fibre} is the $\CC$-vector space defined as
\[
\coH^1_{\dR}(\Afu_\ts,N\otimes E^\ts)
=\coH^0a_{\Afu_\ts,\bast}(N\otimes E^\ts),
\]
where $E^\ts$ is the connection $(\cO_{\Afu_\ts},\de + \de\ts)$
and $a_{\Afu_\ts}$ denote the structure morphism of $\Afu_\ts$. This vector space is endowed with a filtration, called the irregular Hodge filtration (\cf\cite[\S A.3]{F-S-Y18}).

In order to define the Betti fiber functor, we consider the real oriented blowing-up $\varpi:\wt\PP^1\to\PP^1$ of~$\PP^1$ at $\infty$ and the open subset $\wt\PP^1_\rmod=\Afuan\cup\partial_\rmod\wt\PP^1$ in the neighbourhood of which $\rme^{-\ts}$ has moderate growth, equivalently rapid decay, \ie defined by $\mathrm{Re}(\ts)>0$. Let us denote the open inclusions \hbox{$\Afuan\hto\wt\PP^1_\rmod$} and $\wt\PP^1_\rmod\hto\wt\PP^1$ respectively by $\alpha$ and~$\beta$. The \emph{Betti fiber} of $N^\rH$ is defined~as
\[
\coH^0(\wt\PP^1,\beta_!R\alpha_*\cF_\QQ)=\coH^0_\rc(\wt\PP^1_\rmod,R\alpha_*\cF_\QQ).
\]

Let us notice that these definitions can be extended to all objects $N^\rH$ of $\MHM(\Afu_\ts)$ and then the corresponding vector spaces only depend on $\Pi(N^\rH)$ since $\coH_\dR^r(\Afu_\ts,E^\ts)=0$ for all $r$.

In order to define the period isomorphism, we first recall that the de~Rham fiber $\coH^1_{\dR}(\Afu_\ts,N\otimes E^\ts)$ can be computed as the hypercohomology of the moderate de~Rham complex $\DR^\rmod(N\otimes E^\ts)$ on~$\wt\PP^1$ and that there exists a canonical functorial isomorphism (\cf\eg\cite[\S2.e]{F-S-Y20} and the references therein)
\begin{equation}\label{eq:blup}
\coH^1_{\dR}(\Afu_\ts,N\otimes E^\ts)\simeq\coH^0(\wt\PP^1,\pDR^\rmod(N\otimes E^\ts)).
\end{equation}

Furthermore:

\begin{lemma}\label{lemma:RH_N_tw}
There exists a \emph{unique} isomorphism
\[
\beta_!R\alpha_*\DR^\an(N)\simeq\DR^\rmod(N\otimes E^\ts)
\]
extending the identity on the analytic complex line $\Afuan$.
\end{lemma}

\begin{proof}
Indeed, let us recall (\cf\loccit) that $\DR^\rmod(N\otimes E^\ts)$ has cohomology in degree zero only.
Then one can easily show that its $\cH^0$ is zero on the complement of $\wt{\PP}^1_\rmod$,
so that the natural morphism
\[
\beta_!\beta^{-1}\DR^\rmod(N\otimes E^\ts)\to\DR^\rmod(N\otimes E^\ts)
\]
is an isomorphism. One then shows in the same way that
\[
\beta^{-1}\DR^\rmod(N\otimes E^\ts)=R\alpha_*\alpha^{-1}\beta^{-1}\DR^\rmod(N\otimes E^\ts).
\]
Uniqueness follows from the adjunction formulas \cite[(2.3.6) \& (3.0.1)]{K-S90}.
\end{proof}

On the other hand, termwise multiplication by $\rme^{-\ts}$ induces an isomorphism
\[
\pDR^\an(N)\isom\pDR^\an(N\otimes E^\ts).
\]
As a consequence, there exists a unique isomorphism
\[
\per_{\wt\PP^1}:\beta_!R\alpha_*\cF_\CC\to\pDR^\rmod(N\otimes E^\ts)
\]
extending $\rme^{-\ts}\circ\per_{\Afuan}$. This morphism is thus functorial with respect to $N^\rH$.

\pagebreak[2]
\begin{defi}[Fiber period structure $\FPer(N^\rH)$]
Let $N^\rH = (N,\cF_\QQ,\per_{\Afuan})$ be an object of $\EMHS$.
\begin{enumerate}
\item
The \emph{fiber period isomorphism}
\[ \per: \CC\otimes\coH^0(\wt\PP^1,\beta_!R\alpha_*\cF_\QQ)
\to \coH^1_{\dR}(\Afu_\ts,N\otimes E^\ts) \]
is the composition of
$\coH^0(\wt\PP^1,\per_{\wt\PP^1})$ with that given by \eqref{eq:blup}.
\item
The \emph{fiber period structure} $\FPer(N^\rH)$ of an exponential mixed Hodge structure~$N^\rH$ is the following object of the category $\catPer$:
\[
\FPer(N^\rH)=(\coH^1_{\dR}(\Afu_\ts,N\otimes E^\ts),\coH^0(\wt\PP^1,\beta_!R\alpha_*\cF_\QQ),\per).
\]
\end{enumerate}
\end{defi}

In the following, we also regard $\FPer$ as a functor defined on $\MHM(\Afu_\ts)$ by factorization though $\EMHS$ by $\Pi$. The following lemma is then obvious.

\begin{lemma}\label{lem:modelalphabeta}
The assignment $\FPer$ defines an exact functor $\MHM(\Afu_\ts)\to\catPer$ factoring through~$\Pi$,
and there is an isomorphism of functors to the category $\catPer$:
\begin{equation*}
\FPer(N^\rH)\simeq\Bigl(\coH^0\bigl(\wt\PP^1,\beta_!R\alpha_*\pDR^\an(N\otimes E^\ts)\bigr),
\coH^0(\wt\PP^1,\beta_!R\alpha_*\cF_\QQ),\coH^0\bigl((\wt\PP^1,\beta_!R\alpha_*(\rme^{-\ts}\circ\per_{\Afuan})\bigr)\Bigr).
\end{equation*}
\end{lemma}

On the other hand,
let \hbox{$i_0: \{0\} \hto \Afu$} be the closed embedding.
The abelian category $\MHS = \MHM(\{0\})$
of mixed Hodge structures
is identified with a full subcategory of $\MHM(\Afu)$
and of $\EMHS$ via
the pushforward~$\Hm{i_0}_!$ of mixed Hodge modules
and the composition $\Pi\circ \Hm{i_0}_!$,
respectively.

\begin{prop}\label{prop:perMHS}
There is a commutative diagram of functors
\[
\xymatrix@C=1.5cm@R=1.2cm{
\MHS\ar[d]_\Per\ar[r]^-{\Pi\circ\Hm i_{0!}}&\EMHS\ar[dl]^(.35){\FPer}\\
\catPer
}
\]
\end{prop}

\begin{proof}
This follows from the tautological identification for a vector space $V$ over either $\QQ$ or $\CC$:
\[
\coH^0_\rc(\wt\PP^1_\rmod,R\alpha_*Ri_{0!}V)=V,
\]
which commutes with change of base field.
\end{proof}

\subsection{Computation of the fiber period isomorphism for a Gauss-Manin exponential mixed Hodge structure}\label{subsec:computGM}

Let $f:X\to\Afu_\ts$ be a projective morphism from a smooth quasi-projective variety $X$ to the affine line $\Afu_\ts$. We denote the pushforward functor $\catD^\rb(\MHM(X))\to\catD^\rb(\MHM(\Afu_\ts))$ by $\Hm f_*$ (\cf \cite[Th.\,4.3]{MSaito87}). For each object $N^\rH$ of $\MHM(X)$ and each integer~$r$, we consider the mixed Hodge module $\cH^r\Hm f_*N^\rH$ that we simply denote by $\Hm f_*^rN^\rH$. We will compute its fiber period structure by means of cohomology on a suitable real blow-up space.

For that purpose, consider a smooth projective compactification $\ov f:\ov X\to\PP^1$ of~$X$ and $f$ giving rise to a commutative diagram
\[
\xymatrix{
X\ar[d]_f\ar@{^{ (}->}[r]^-j&\ov X\ar[d]^{\ov f} \\
\Afu_\ts\ar@{^{ (}->}[r]&\PP^1
}
\]
such that the pole divisor $P=\ov f^{-1}(\infty)$ is a strict normal crossing divisor in~$\ov X$. We denote by $\varpi:\wt X(P)\to\ov X$ the real-oriented blowing-up of~$\ov X$ along the irreducible components of~$P$. Then $\ov f$ lifts as a real-analytic morphism $\wt f:\wt X(P)\to\wt\PP^1$. In this section, we set $\wt X=\wt X(P)$. We define the open subset $\partial_\rmod\wt X$ of $\partial\wt X=\varpi^{-1}(P)$ as consisting of points of $\varpi^{-1}(P)$ in the neighbourhood of which $\rme^{-f}$ has moderate growth---equivalently, rapid decay---so~that $\partial_\rmod\wt X=\wt f{}^{-1}(\partial_\rmod\wt\PP^1)$. The subset
\[
\wt X_\rmod:=X\cup\partial_\rmod\wt X=\wt f{}^{-1}(\wt\PP^1_\rmod)
\]
is open in $\wt X$. We also denote by $\alpha,\beta$ the respective open inclusions
\[
X\Hto{\alpha}\wt X_\rmod\Hto{\beta}\wt X.
\]

As in Section \ref{subsec:Kl2connection}, we set $E^f=(\cO_X,\rd+\rd f)$ and, for a $\cD_X$-module $N$ regarded as an $\cO_X$-module with flat connection $\nabla$, we set $N\otimes E^f=(N,\nabla+\rd f)$.
To a mixed Hodge module $N^\rH$ on $X$ we associate a period structure in $\catPer$ as follows.
The vector spaces are (the perverse convention is used here)
\begin{align*}
\pcoH^r_\dR(X,N\otimes E^f)&:=\coH^r(X,\pDR_X(N\otimes E^f)),\\
\pcoH^r_\Betti(X,\cF_\QQ\otimes E^f)&:=\coH^r(\wt X,\beta_!R\alpha_*\cF_\QQ)=\coH^r_\rc(\wt X_\rmod,R\alpha_*\cF_\QQ).
\end{align*}
In order to make explicit the period isomorphism, we need the following proposition. We note that termwise multiplication by the holomorphic function $\rme^{-f}$ induces an isomorphism $\pDR_X^\an N\isom\pDR_X^\an(N\otimes E^f)$. On the other hand, we endow~$\wt X$ with the sheaf $\cA_{\wt X}^\rmod$ of holomorphic functions on $X$ having moderate growth along $\partial\wt X$. Then any $\cD_X$-module $M$ has a (shifted) moderate de~Rham complex $\pDR_{\wt X}^\rmod M$ on $\wt X$ (\cf\eg\cite[\S2.e]{F-S-Y20} and the references therein). We also denote by $j: X\to\wt{X}$ the inclusion.

\begin{prop}\label{prop:DRmod}
For a regular holonomic $\cD_X$-module $N$, the following two natural morphisms
\begin{align*}
\beta_!\beta^{-1}\,\pDR^\rmod_{\wt X}j_+(N\otimes E^f)&\to\pDR^\rmod_{\wt X}j_+(N\otimes E^f),\\
\beta^{-1}\,\pDR^\rmod_{\wt X}j_+(N\otimes E^f)&\to R\alpha_*\alpha^{-1}\beta^{-1}\,\pDR^\rmod_{\wt X}j_+(N\otimes E^f)
\end{align*}
are quasi-isomorphisms.
\end{prop}

\begin{proof}[Sketch of proof]
A similar result is proved in \cite{Bibi94} (proof of Th.\,5.1) but the definition of the sheaf of functions with moderate growth used there is more restrictive than the one needed for our purposes. Instead, one uses computations similar to those of \cite[Prop.\,3.3]{Hien07} (\cf also \cite[Prop.\,1]{Hien09}) together with \cite[Cor.\,4.7.3]{Mochizuki10}.
\end{proof}

We conclude that there is a natural isomorphism
\begin{equation}\label{eq:uniqueiso}
\beta_!R\alpha_*\pDR^\an_X(N\otimes E^f)\isom\pDR^\rmod_{\wt X}j_+(N\otimes E^f)
\end{equation}
in the derived category $\catD^\rb(\CC_X)$, which is functorial with respect to $N$. Moreover, by the arguments of adjunction already used in the proof of Lemma \ref{lemma:RH_N_tw}, this is the unique isomorphism extending the identity on $X$.

\begin{cor}
For a mixed Hodge module $N^\rH\in\MHM(X)$ there exists a unique isomorphism
\[
\per_{\wt X}:\beta_!R\alpha_*\cF_\CC\to\pDR^\rmod_{\wt X}j_+(N\otimes E^f)
\]
which extends $\rme^{-f}\circ\per_X:\cF_\CC\to\pDR_X^\an(N\otimes E^f)$. We have
\[
\per_{\wt X}=\eqref{eq:uniqueiso}\circ\beta_!R\alpha_*(\rme^{-f}\circ\per_X).
\]
This morphism is functorial on $\MHM(X)$.\qed
\end{cor}

In a way similar to \eqref{eq:blup},
one shows that, for any $r$, there exists a canonical isomorphism
\begin{align*}
\can_r:\coH^r(\wt X,\pDR_{\wt X}^\rmod j_+(N\otimes E^f))&\simeq \coH^r(\ov X,\pDR_{\ov X}j_+(N\otimes E^f))\\
&=\coH^r(X,\pDR_X(N\otimes E^f))=:\pcoH^r_\dR(X,N\otimes E^f).
\end{align*}

\begin{defi}
Let $N^\rH$ be an object of $\MHM(X)$. For each $r$ (the perverse convention is used here), the fiber period structure $\FPer^r(N^\rH\otimes E^f)$ is defined as
\[
\FPer^r(N^\rH\otimes E^f)=\bigl(\pcoH^r_\dR(X,N\otimes E^f),\coH^r_\Betti(X,\cF_\QQ\otimes E^f),\can_r\circ\pcoH^r(\wt X,\per_{\wt X})\bigr).
\]
\end{defi}

\begin{lemma}[Expression of $\FPer^r(N^\rH\otimes E^f)$ on the real blow-up]\label{lem:modelalphabetaX}
For each $r$, there is an isomorphism in $\catPer$:
\begin{equation*}
\FPer^r(N^\rH\otimes E^f)
\simeq\Bigl(\coH^r\bigl(\wt X,\beta_!R\alpha_*\pDR_X^\an j_+(N\otimes E^f)v),
\coH^r(\wt X,\beta_!R\alpha_*\cF_\QQ),\coH^r\bigl(\wt X,\beta_!R\alpha_*(\rme^{-f}\circ\per_X)\bigr)\Bigr).
\end{equation*}
\end{lemma}

\begin{proof}
The isomorphism is the identity on the middle term and the isomorphism on the left term is given by $\can_r\circ\coH^r(\wt X,\eqref{eq:uniqueiso})$.
\end{proof}

We then deduce that the fiber period structure $\FPer^r(N^\rH\otimes E^f)$ is isomorphic to that attached to the exponential mixed Hodge structure $\Pi(\Hm f_*^rN^\rH)$:

\begin{prop}\label{prop:FPFP}
For each $r$, there is a functorial isomorphism of fiber period structures:
\[
\FPer(\Hm f_*^rN^\rH)\simeq\FPer^r(N^\rH\otimes E^f).
\]
\end{prop}

\begin{proof}
By definition, the period isomorphism $\per_{\Afu_\ts}$ for $\Hm f_*^jN^\rH$ is obtained by taking the $j$th perverse cohomology of the composed isomorphism
\[
Rf_*\cF_\CC\To{Rf_*\per_X}Rf_*\pDR^\an N\isom\pDR^\an f_+N,
\]
where the second morphism is the standard functorial isomorphism from $\cD$-module theory. We also have a commutative diagram
\[
\xymatrix{
Rf_*\pDR^\an N\ar[d]_{Rf_*\rme^{-f}}\ar[r]&\pDR^\an f_+N\ar[d]^{\rme^{-\ts}}\\
Rf_*\pDR^\an(N\otimes E^f)\ar[r]&\pDR^\an ((f_+N)\otimes E^\ts).
}
\]
We now take the models of Lemmas \ref{lem:modelalphabeta} and \ref{lem:modelalphabetaX} and the desired isomorphism is obtained (after applying $R\Gamma(\wt\PP^1,\cbbullet)$ and taking cohomology) from the canonical isomorphism of functors
\[
R\wt f_*\circ\beta_!R\alpha_*\simeq\beta_!R\alpha_*Rf_*
\]
from $\catD^\rb(\CC_X)$ to $\catD^\rb(\CC_{\wt\PP^1})$, that we recall now. We consider the two commutative squares
\[
\xymatrix@C=1.5cm{
X\ar@{^{ (}->}[r]^-\alpha\ar[d]_f&\wt X_\rmod\ar@{^{ (}->}[r]^-\beta\ar[d]^{\wt f_\rmod}&\wt X\ar[d]^{\wt f}\\
\Afu_\ts\ar@{^{ (}->}[r]^-\alpha&\wt\PP^1_\rmod\ar@{^{ (}->}[r]^-\beta&\wt\PP^1
}
\]
Since $\wt f$ is proper, we have a canonical isomorphism of functors $R\wt f_*\circ\beta_!\simeq\beta_!\circ R\wt f_{\rmod*}$, and on the other hand we have $R\wt f_{\rmod*}\circ R\alpha_*\simeq R\alpha_*\circ R\wt f_*$.
\end{proof}

\Subsection{Fiber period realization of a pair $(U,f)$}\label{subsec:fiberperiodfunction}
\subsubsection*{Application of the results of Section \ref{subsec:computGM}}
Let~$U$ be a smooth complex quasi-projective variety of dimension~$d$ and let $f: U\to\Afu_\ts$ be a regular function. Starting with the mixed Hodge module $\pQQ_U^\rH$, we aim at giving a formula for the fiber period structure of the exponential mixed Hodge structures associated with the pushforward mixed Hodge modules $\Hm f_*^r\,\pQQ_U^\rH$ and $\Hm f_!^r\,\pQQ_U$ ($r\in\ZZ$). For that purpose, it is convenient to choose a partial completion $\kappa:U\hto X$ of $U$ as a smooth quasi-projective variety $X$ so that $H=X\moins U$ is a strict normal crossing divisor and that~$f$ extends as a projective morphism $f:X\to\Afu_\ts$. The result will be independent of such a choice. The commutative diagram used in Section \ref{subsec:computGM} is thus completed on the left as follows:
\[
\xymatrix{
U\ar[d]_f\ar@{^{ (}->}[r]^-\kappa& X\ar[d]^f\ar@{^{ (}->}[r]^-j&\ov X\ar[d]^{\ov f} \\
\Afu_\ts\ar@{=}[r]&\Afu_\ts\ar@{^{ (}->}[r]&\PP^1
}
\]
We consider the objects $N^\rH$ of $\MHM(X)$ defined as $N^\rH=\Hm \kappa_*\pQQ_U^\rH$ or $N^\rH=\Hm \kappa_!\pQQ_U^\rH$. Correspondingly, we write $N\otimes E^f$ as $E^f(*H)$ or $E^f(!H)$ and we have $\cF_\QQ=R\kappa_*\pQQ_U$ or $R\kappa_!\pQQ_U$. In this way, we are in the setting of Section \ref{subsec:computGM}.

We denote respectively by $\coH^r(U,f)$ and $\coH^r_\rc(U,f)$ the exponential mixed Hodge structure $\Pi(\Hm f_*^{r-d}\,\pQQ_U^\rH)$ and $\Pi(\Hm f_!^{r-d}\,\pQQ_U^\rH)$, and
\begin{equation}\label{eq:Hrmid}
\coH^r_\rmid(U,f)=\image[\coH^r_\rc(U,f)\ra\coH^r(U,f)]
\end{equation}
We denote respectively the associated de Rham fibers by
$\coH_\dR^r(U,f)$ and $\coH_{\dR,\cp}^r(U,f)$,
and the Betti fibers by
$\coH^r_\Betti(U,f)$ and $\coH^r_{\Betti,\cp}(U,f)$. We then have
\[
\coH_\dR^r(U,f)\simeq \coH^r\bigl(X,\DR(E^f(*H))\bigr),\quad \coH_{\dR,\rc}^r(U,f)\simeq \coH^r\bigl(X,\DR(E^f(!H))\bigr).
\]

Keeping the notation of Section \ref{subsec:computGM} for $\alpha,\beta$ (and emphasizing now the divisor~$P$), we have by Proposition \ref{prop:FPFP}:
\begin{align*}
\coH^r_\Betti(U,f)&=\coH^r(\wt X(P),\beta_!R\alpha_*R\kappa_*\QQ_U),\\
\coH^r_{\Betti,\rc}(U,f)&=\coH^r(\wt X(P),\beta_!R\alpha_*R\kappa_!\QQ_U).
\end{align*}

We set $\wt U_\rmod(P)=\wt X_\rmod(P)\moins\varpi^{-1}(H)$ and we denote by $\Phi$ the family of closed subsets of $\wt U_\rmod(P)$ whose closure in $\wt X$ is contained in the open subset $\wt X_\rmod(P)$.

\begin{prop}\label{prop:BettiD}
We have
\[
\coH^r_\Betti(U,f)\simeq\coH^r_\Phi(\wt U_\rmod(P),\QQ)\quand\coH^r_{\Betti,\rc}(U,f)\simeq\coH^r_\rc(\wt U_\rmod(P),\QQ),
\]
and the natural morphism $\coH^r_{\Betti,\rc}(U,f)\to\coH^r_\Betti(U,f)$ is induced by the inclusion of the families of supports.
\end{prop}

\begin{remark}\label{rem:relcohomology}
Setting $\wt U(P)=\wt X(P)\moins\varpi^{-1}(H)$ and denoting by $\partial_{\exp}\wt X(P)$, respectively $\partial_{\exp}\wt U(P)$, the closed subset complement to $\wt X_\rmod(P)$ in $\wt X(P)$, respectively to $\wt U_\rmod(P)$ in $\wt U(P)$, the spaces $\coH^r_\Betti(U,f)$ and $\coH^r_{\Betti,\rc}(U,f)$ also read in terms of relative cohomology:
\begin{align*}
\coH^r_\Betti(U,f)&\simeq\coH^r(\wt U(P),\partial_{\exp}\wt U(P),\QQ),\\
\coH^r_{\Betti,\rc}(U,f)&\simeq\coH^r(\wt X(P),\partial_{\exp}\wt X(P)\cup\varpi^{-1}(H),\QQ),
\end{align*}
and the natural morphism between both is induced by the inclusion of pairs $(\wt U(P),\partial_{\exp}\wt U(P))\hto(\wt X(P),\partial_{\exp}\wt X(P))$, so that $\coH^r_{\Betti,\rmid}(U,f)$ is the corresponding image, according to \eqref{eq:Hrmid}.
\end{remark}

\begin{proof}[Proof of Proposition \ref{prop:BettiD}]
In the proof, we simply set $\wt X=\wt X(P)$. With the notation above, we consider the commutative diagram
\[
\xymatrix@C=1.3cm{
U\ar@{^{ (}->}[r]^-{\alpha_U}\ar@{^{ (}->}[d]^\kappa&\wt U_\rmod \ar@{^{ (}->}[r]^-{\beta_U}\ar@{^{ (}->}[d]^{\wt\kappa'}&\wt U\ar@{^{ (}->}[d]^{\wt\kappa}&\ar@{_{ (}->}[l]_-{\gamma_U}\partial_{\exp}\wt U\ar@{^{ (}->}[d]\\
X\ar@{^{ (}->}[r]^-\alpha&\wt X_\rmod\ar@{^{ (}->}[r]^-\beta&\wt X&\ar@{_{ (}->}[l]_-\gamma\partial_{\exp}\wt X
}
\]
where the first line is obtained from the second one by deleting $\varpi^{-1}(H)$.

For the identification of $\pcoH^r_\Betti(U,f)$, we need the next lemma.

\begin{lemma}
For $\star=*$ or $\star={}!$, there is an isomorphism in $\catD^\rb(\wt X,\QQ)$:
\[
R\wt\kappa_\star\beta_{U!}R\alpha_{U*}\QQ_U\simeq\beta_!R\alpha_*R\kappa_\star\QQ_U.
\]
\end{lemma}

\begin{proof}
We can replace $R\alpha_*R\kappa_\star$ with $R\wt\kappa'_\star R\alpha_{U*}$. Furthermore, a local computation shows that
\[
R\alpha_{U*}\QQ_U=\QQ_{\wt U_\rmod}=\beta_U^{-1}\QQ_{\wt U}\quand (R\beta_U\circ R\alpha_U)_*\QQ_U=\QQ_{\wt U}.
\]
We are thus reduced to finding an isomorphism $R\wt\kappa_\star\beta_{U!}\QQ_{\wt U_\rmod}\simeq\beta_!R\wt\kappa'_\star\QQ_{\wt U_\rmod}$. Let us first construct a morphism. There is a natural morphism
\[
R\wt\kappa_\star\beta_{U!}\QQ_{\wt U_\rmod}\to R\wt\kappa_\star R\beta_{U*}\QQ_{\wt U_\rmod}\simeq R\beta_*R\wt\kappa'_\star\QQ_{\wt U_\rmod}
\]
and this morphism can be lifted as a morphism to $\beta_!R\wt\kappa'_\star\QQ_{\wt U_\rmod}$ if and only if its restriction by $\gamma$ is zero. Clearly, $\gamma^{-1}R\wt\kappa_\star\beta_{U!}\QQ_{\wt U_\rmod}$ is zero on $\partial_{\exp}\wt U$ and we need to check that the same property holds true on $\varpi^{-1}(H)\cap\partial_{\exp}\wt X$. The question reduces to a local computation in the neighbourhood of each point of $P\cap H$ in $\ov X$. We thus work in an adapted coordinate neighbourhood $\Delta^d$ of such a point. We can write $\Delta^d=\Delta^\ell\times\Delta^{d-\ell}$, with $P\cap\Delta^d=P'\times\Delta^{d-\ell}$ defined by the vanishing of the product of coordinates in~$\Delta^\ell$ and $H\cap\Delta^d=\Delta^\ell\times H''$ defined by the vanishing of the product of some coordinates in $\Delta^{d-\ell}$. In this model, the real blowing-up $\varpi:\wt\Delta^\ell\times\Delta^{d-\ell}\to\Delta^\ell\times\Delta^{d-\ell}$ is induced by the real blowing-up of $\Delta^\ell$ along its coordinates hyperplanes. In restriction to this chart we have $U=\Delta^\ell\times(\Delta^{d-\ell}\moins H'')$ and
\begin{equation}\label{eq:chartblup}
\begin{aligned}
\wt X_\rmod&=(\wt\Delta^\ell)_\rmod\times\Delta^{d-\ell},&
\wt U_\rmod&=(\wt\Delta^\ell)_\rmod\times(\Delta^{d-\ell}\moins H''),\\
\partial_{\exp}\wt X&=\partial_{\exp}(\wt\Delta^\ell)\times\Delta^{d-\ell},& \partial_{\exp}\wt U&=\partial_{\exp}(\wt\Delta^\ell)\times(\Delta^{d-\ell}\moins H'').
\end{aligned}
\end{equation}
The assertion is then clear since the morphisms $\beta_U$ and $\wt\kappa$ act on disjoint sets of variables. With the same local computation, one checks that the morphism thus obtained is an isomorphism.
\end{proof}

We can now conclude the proof for $\coH^r_\Betti(U,f)$. From the previous lemma with $\star=*$ we deduce
\[
\coH^r(\wt X,\beta_!R\alpha_*R\kappa_*\QQ_U)\simeq\coH^r(\wt X,R\wt\kappa_*\beta_{U!}R\alpha_{U*}\QQ_U)=\coH^r(\wt U,\beta_{U!}\QQ_{\wt U_\rmod}),
\]
and the assertion is then clear. On the other hand, the distinguished triangle in $\catD^\rb(\wt U,\QQ)$
\[
\beta_{U!}\beta_U^{-1}\QQ_{\wt U}\to\QQ_{\wt U}\to R\gamma_{U*}\gamma^{-1}_U\QQ_{\wt U}\To{+1}
\]
gives the expression of $\coH^r_\Betti(U,f)$ in terms of relative cohomology as asserted in Remark \ref{rem:relcohomology}.

For $\coH^r_{\Betti,\rc}(U,f)$, the previous lemma with $\star={}!$ gives similarly
\[
\coH^r(\wt X,\beta_!R\alpha_*\kappa_!\QQ_U)\simeq\coH^r(\wt X,(\wt\kappa\circ\beta_U)_!R\alpha_{U*}\QQ_U)=\coH^r_\rc(\wt U_\rmod,\QQ).\qedhere
\]
\end{proof}

\begin{remark}\label{rem:exactseq}
Let $Z\subset U$ be a divisor on which $f$ vanishes, let $a_Z$ denote the structure morphism, let $i_Z\,{:}\,Z\!\hto\!U$ denote the closed inclusion and \hbox{$j_Z\,{:}\,U\!\moins\! Z\!\hto\!U$} the complementary open inclusion. We set $\pQQ_Z^\rH=\Hm a_Z^*\QQ^\rH_{\mathrm{Spec}\CC}[\dim Z]$, so that there is an isomorphism
\[
\Hm i_Z^*\,\pQQ_U^\rH=\cH^0\Hm i_Z^*\,\pQQ_U^\rH\simeq \pQQ_Z^\rH
\]
and an exact sequence
\[
0\to\Hm i_{Z*}\pQQ_Z^\rH\to\Hm j_{Z!}\,\Hm j_Z^*\pQQ_U^\rH\to\pQQ_U^\rH\to0,
\]
giving rise to an exact sequence in $\EMHS$ (\cf\cite[(A.21)]{F-S-Y18}):
\[
\cdots\to\coH^{r-1}_\rc(Z)\to\coH^r_\rc(U\moins Z,f)\to\coH^r_\rc(U,f)\to\coH^r_\rc(Z)\to\cdots
\]
If $\coH^r_\rc(U\moins Z,f)=0$ for each $r$, the exponential mixed Hodge structure $\coH^r_\rc(U,f)$ is isomorphic to the mixed Hodge structure $\coH^r_\rc(Z)$ and, correspondingly, the fiber period structure $\FPer(\coH^r_\rc(U,f))$ is isomorphic to $\Per(\coH^r_\rc(Z))$.
We will make explicit this exact sequence for the Betti fibers. Since $\ol{f}:\ov X\to\PP^1$ is a morphism, we have $\ov Z\cap P=\emptyset$.
We have a distinguished triangle
\[
\beta_!R\alpha_*Ri_{Z*}\QQ_Z\to \beta_!R\alpha_*j_{Z!}\QQ_{U\moins Z}\to\beta_!R\alpha_*\QQ_U\To{+1}
\]
and since the closure of $Z$ in $\wt X(P)$ does not intersect $\partial\wt X(P)$, we find
\[
\beta_!R\alpha_*Ri_{Z*}\QQ_Z=\beta_!R\alpha_!Ri_{Z*}\QQ_Z\quand\beta_!R\alpha_*j_{Z!}\QQ_{U\moins Z}=\beta_{Z!}\QQ_{\wt U_\rmod(P)\moins Z},
\]
where $\beta_Z$ is the inclusion $\wt U_\rmod(P)\moins Z\hto\wt X(P)$. The Betti exact sequence reduces then to
\begin{equation}\label{eq:exactseq}
\cdots\to\coH^{r-1}_\rc(Z,\QQ)\to\coH^r_\rc(\wt U_\rmod(P)\moins Z,\QQ)
\to\coH^r_\rc(\wt U_\rmod(P),\QQ)
\to\coH^r_\rc(Z,\QQ)\to\cdots
\end{equation}
\end{remark}

\subsubsection*{Computation with the total real blow-up}
In order to use results of \cite{F-S-Y20}, we consider the real blowing-up
$\pi:\wt X(D)\to\ov X$ of the irreducible components of $D=P\cup H$ in $\ov X$. There is a natural morphism $\wt\varpi:\wt X(D)\to\wt X(P)$, so that $\pi=\varpi\circ\wt\varpi$. In a local chart where formulas \eqref{eq:chartblup} hold, $\wt\varpi$ is the blowing-up map of the components of $H''$ in~$\Delta^{d-\ell}$:
\[
\wt X(D)=\wt\Delta^\ell\times\wt\Delta^{d-\ell}\to\wt\Delta^\ell\times\Delta^{d-\ell}=\wt X(P).
\]
We consider the open subsets $\wt U_\rmod(D)=U\cup\partial_\rmod\wt X(D)$ and $\wt U_\rrd(D)=U\cup\partial_\rrd\wt X(D)$, where
\begin{itemize}
\item
$\partial_\rmod\wt X(D)$ is the open subset of $\pi^{-1}(D)$ in the neighbourhood of which $\rme^{-f}$ has moderate growth (it contains $\pi^{-1}(D\moins P)$),
\item
$\partial_\rrd\wt X(D)$ is the open subset of $\pi^{-1}(P)$ in the neighbourhood of which $\rme^{-f}$ has moderate growth, equivalently, rapid decay.
\end{itemize}
In the local chart as above, these sets read
\[
\wt U_\rmod(D)=(\wt\Delta^\ell)_\rmod\times\wt\Delta^{d-\ell},\quad \wt U_\rrd(D)=(\wt\Delta^\ell)_\rmod\times(\Delta^{d-\ell}\moins H'').
\]
For the sake of simplicity, we denote by $\QQ_{\wt U_\rmod(D)}$ the sheaf on $\wt X(D)$ which is the extension by zero of the constant sheaf on $\wt U_\rmod(D)$ with stalk $\QQ$ (notation of \cite{K-S90}), and similarly with $\rrd$.
From the previous identifications with now $\alpha: X \hto \wt{X}_\rmod(D)$ and $\beta: \wt{X}_\rmod(D) \hto \wt{X}(D)$, we obtain
\begin{align*}
R\wt\varpi_*\QQ_{\wt U_\rmod(D)}&=\beta_!R\alpha_*R\kappa_*\QQ_U,\\
R\wt\varpi_*\QQ_{\wt U_\rrd(D)}&=\QQ_{\wt U_\rmod(P)}
\end{align*}
(in fact $\wt\varpi:\wt U_\rrd(D)\ra\wt U_\rmod(P)$ is an isomorphism).Therefore,
\begin{equation}\label{eq:BettiUfD}
\coH^r_\Betti(U,f)\simeq\coH^r_\rc(\wt U_\rmod(D),\QQ)\quand\coH^r_{\Betti,\rc}(U,f)\simeq\coH^r_\rc(\wt U_\rrd(D),\QQ).
\end{equation}
Let $\Phi_\rrd$ (\resp $\Phi_\rmod$) denote the family of closed subsets $F$ of $U$ whose (compact) closure $\ov F$ in $\wt X(D)$ is contained in $\wt U_\rrd(D)$ (\resp $\wt U_\rmod(D)$). A closed set $F$ of $U$ belongs to $\Phi_\rrd$ (\resp $\Phi_\rmod$)
if and only if $|\exp(-f)|_{|F}$ tends to zero
faster than any negative power of $\mathrm{dist}(x,x_o)$ (\resp is bounded by some positive power of
$\mathrm{dist}(x,x_o)$) when $x\in F$ tends to some $x_o\in D$. Then the right-hand sides in \eqref{eq:BettiUfD} read
\begin{equation}\label{eq:BettiUfDPhi}
\coH_\rc^r(\wt U_\rmod(D),\QQ)=\coH_{\Phi_\rmod}^r(U,\QQ),\qquad
\coH_\rc^r(\wt U_\rrd(D),\QQ)=\coH_{\Phi_\rrd}^r(U,\QQ),
\end{equation}
and, by considering the natural morphism induced by the inclusion of family of supports, we have
\begin{equation}\label{eq:BettiUfDPhimid}
\coH^r_{\Betti,\rmid}(U,f)=\image\bigl[\coH_{\Phi_\rrd}^r(U,\QQ)\ra\coH_{\Phi_\rmod}^r(U,\QQ)\bigr].
\end{equation}

\subsubsection*{Rapid decay and moderate growth homology spaces for the pair $(U,f)$}
If $\mathbf{1}$ denotes the generator of $E^f=(\cO_U,\rd+\rd f)$, then $\exp(-f)\cdot\mathbf{1}$ is an analytic flat section of~$E^f$. The moderate growth and the rapid decay homology spaces of the pair $(U, f)$, as defined in~\cite{F-S-Y20}, are the homology of the chain complexes consisting of singular chains in $\wt X(D)$ with boundary in $\partial\wt X(D)$ twisted by the flat section $\exp(-f)\cdot\mathbf{1}$, whose support is contained in
$\wt U_\rmod(D)$ and $\wt U_\rrd(D)$, respectively.
The flat section being fixed, we get identifications with relative homology spaces:
\begin{equation}\label{eq:Bettihomology}
\begin{split}
\coH^\rmod_r(U,f)&\simeq\coH_r(\wt U_\rmod(D),\partial\wt U_\rmod(D),\QQ),\\
\coH^\rrd_r(U,f)&\simeq\coH_r(\wt U_\rrd(D),\partial\wt U_\rrd(D),\QQ).
\end{split}
\end{equation}

\begin{notation}
For the sake of simplicity, \emph{we omit the flat section $\exp(-f)\cdot\mathbf{1}$ in the notation of such twisted chains}, that we simply call respectively rapid decay and moderate chains (we will not make use of the latter).
\end{notation}

We have a more explicit expression of the rapid decay homology as follows. By suitably lifting to $\wt X(D)$ the radial vector field of length one centered at $\infty\in\PP^1$ so that it remains tangent to
$\pi^{-1}(D\moins P)$, and by following its flow, we obtain for every large enough $R>0$ a deformation retraction of the pair $(\wt U_\rrd(D),\partial_R\wt U_\rrd(D))$ to the pair $(\wt U_\rrd(D),\partial\wt U_\rrd(D))$, where the thickened boundary $\partial_R\wt U_\rrd(D)$ is defined as $\wt U_\rrd(D)\cap\{|f|\geq R\}\cap\{\Re(f)>0\}$. Setting
\[
\partial_{\rrd,R}U=U\cap\partial_R\wt U_\rrd(D)=U\cap\{|f|\geq R\}\cap\{\Re(f)>0\},\quad R\gg0,
\]
by excision of $\partial\wt U_\rrd(D)$, we obtain
\begin{equation}\label{eq:UrdR}
\begin{split}
\coH^\rrd_r(U,f)\simeq\coH_r(\wt U_\rrd(D),\partial\wt U_\rrd(D),\QQ)&\simeq\coH_r(\wt U_\rrd(D),\partial_R\wt U_\rrd(D),\QQ)\\
&\simeq\coH_r(U,\partial_{\rrd,R}U,\QQ),\quad R\gg0.
\end{split}
\end{equation}

\begin{remark}[Period pairing and period realization]\label{rem:FPerP}
Working with the transposed period structures as in \cite[Prop.\,2.28]{F-S-Y20}, and considering rapid decay and moderate growth homology, one can show that there exist isomorphisms of (the transposes of) period structures
\begin{align*}
\FPer\coH^r(U,f)&\simeq(\coH^r_{\dR}(U,f),\coH_r^\rrd(U,f),\Ppairing_r^{\rrd,\rmod}),\\
\FPer\coH^r_\rc(U,f)&\simeq(\coH^r_{\dR,\rc}(U,f),\coH_r^\rmod(U,f),\Ppairing_r^{\rmod,\rrd})\\
\FPer\coH^r_\rmid(U,f)&\simeq(\coH^r_{\dR,\rmid}(U,f),\coH_r^\rmod(U,f),\Ppairing_r^{\rmid}).
\end{align*}
\end{remark}

\Subsection{Period structures over the category of varieties and morphisms defined over $\QQ$}\label{subsec:PerSpecQ}
In this section, we denote by $U_0$ a variety defined over $\QQ$ and by $U$ the variety defined over~$\CC$ after extension of scalars from $\QQ$ to $\CC$. When working over varieties and morphism defined over~$\QQ$, that is, smooth separated schemes of finite type over~$\QQ$ and separated morphisms (\eg $U_0$ is $\mathbb{A}^n$ or~$\Gm^n$ and $f$ is a polynomial or a Laurent polynomial with rational coefficients), we are led to consider period structures over $\Spec\QQ$.
Such a period structure consists of a pair of finite-dimensional $\QQ$-vector spaces $(V_0,V_\QQ)$ together with a period isomorphism $\per:\CC\otimes_\QQ V_\QQ\simeq\CC\otimes_\QQ V_0=V^\CC$.

\subsubsection*{$\QQ$-structure on the de~Rham cohomology}
We fix a good compactification
\[
j:(U_0,f)\hto(\ov X_0,\ov f),
\]
that~is, such that $D_0=X_0\moins U_0$ is a divisor with strict normal crossings
(\ie such that the irreducible components over $\ov\QQ$ are smooth and intersect transversally). We work with the corresponding category of $\cD$-modules and functors (\cf\eg \cite[\S\S4,\,5]{Laumon83} and the references therein). The de~Rham cohomology $\coH^r_{\dR}(U_0,f)$ is defined in a standard way as the de~Rham cohomology of the $\cD_{U_0}$-module $(\cO_{U_0},\rd+\rd f)$, and the de~Rham cohomology with compact support $\coH^r_{\dR,\rc}(U_0,f)$ is the de~Rham cohomology of the $\cD_{\ov X_0}$-module $j_\dag(\cO_{U_0},\rd+\rd f)$, with $j_\dag=\bD\circ j_+\circ\bD$ and $\bD$ is the duality functor of $\cD$-modules. We denote by $(U, f)$ the corresponding object obtained by extension of scalars from $\QQ$ to $\CC$. We have:

\begin{lemma}\label{lem:extscalars}
Extension of scalars is compatible with taking de~Rham cohomology, that is,
\[
\CC\otimes_\QQ\coH^r_{\dR}(U_0,f)\simeq\coH^r_{\dR}(U,f),\quad \CC\otimes_\QQ\coH^r_{\dR,\rc}(U_0,f)\simeq\coH^r_{\dR,\rc}(U,f).\eqno\qed
\]
\end{lemma}

As an immediate consequence we obtain that the same result holds for the middle de~Rham cohomology $\coH^r_{\dR,\rmid}(U_0,f)$.

Let us consider the setting of Remark \ref{rem:exactseq} and let us assume that the triple $(U,f,Z)$ is defined over $\QQ$. Let us set $M=(\cO_{U_0},\rd+\rd f)$. There is a natural exact sequence:
\[
0=R^0\Gamma_{Z_0}M\to M\to j_{Z_0+}j_{Z_0}^+M=M(*Z_0)\to R^1\Gamma_{Z_0}M\to0
\]
which identifies the complex $R\Gamma_{Z_0}M$ with $(\cO_{U_0}(*Z_0)/\cO_{U_0},\rd+\rd f)[-1]$. We note that
\[
(\cO_{U_0}(*Z_0)/\cO_{U_0},\rd+\rd f)\simeq(\cO_{U_0}(*Z_0)/\cO_{U_0},\rd).
\]
Indeed, this amounts to showing that the exponential function $\exp(\pm f)$ is well-defined on $\cO_{U_0}(*Z_0)/\cO_{U_0}$, and this follows from the nilpotency of multiplication by $f$ on each local section of $\cO_{U_0}(*Z_0)/\cO_{U_0}$. The long exact sequence in de~Rham cohomology thus reads
\[%\let\to\ra
\cdots\to\coH^r_{\dR,Z_0}(U_0)\to\coH^r_\dR(U_0,f)\to\coH^r_\dR(U_0\moins Z_0,f)\to\coH^{r+1}_{\dR,Z_0}(U_0)\to\cdots
\]
Dually (in the sense of $\cD_{\ov X_0}$-modules, and up to changing $f$ to $-f$, we obtain the exact sequence
\[%\let\to\ra
\cdots\to\coH^{r-1}_{\dR,\rc}(Z_0)\to\coH^r_{\dR,\rc}(U_0\moins Z_0,f)\to\coH^r_{\dR,\rc}(U_0,f)\to\coH^r_{\dR,\rc}(Z_0)\to\cdots
\]

\begin{cor}\label{cor:dRQiso}
Assume moreover that the $\QQ$-vector spaces $\coH^r_\dR(U_0\moins Z_0,f)$ and $\coH^r_{\dR,\rc}(U_0\moins Z_0,f)$ are zero for all $r$. Then the $\QQ$-de~Rham vector spaces $\coH^r_\dR(U_0,f)$ and $\coH^r_{\dR,Z_0}(U_0)$, respectively $\coH^r_{\dR,\rc}(U_0,f)$ and $\coH^r_{\dR,\rc}(Z_0)$, coincide.\qed
\end{cor}

\begin{example}\label{ex:dRQiso}
We consider the setting of \cite[Ex.\,A.27]{F-S-Y18} where the assumptions of Corollary \ref{cor:dRQiso} hold. We thus assume that $U_0=\Afu_t\times V$ for some smooth quasi-projective variety~$V_0$ and $f=tg$ for some regular function~$g$ on~$V_0$. We set $\cK_0=g^{-1}(0)$ and $Z_0=\Afu_t\times\cK_0$. Corollary \ref{cor:dRQiso} gives identifications of $\QQ$-vector spaces
\[
\coH^r_{\dR,Z_0}(U_0)\simeq\coH^r_\dR(U_0,f)\quand\coH^r_{\dR,\rc}(Z_0)\simeq\coH^r_{\dR,\rc}(U_0,f).
\]
\end{example}

\subsubsection*{Action of complex conjugation}
We denote by $(U^\RR, f^\RR)$ (or simply $U^\RR,f^\RR)$ the real-analytic space and map associated with $(U(\CC), f)$.
Then the complex conjugation endows $U^\RR$ with a real analytic involution $\conj$ which commutes with~$f^\RR$. Furthermore, one can find a compactification $(\ov X_0,D_0)$ defined over~$\QQ$ (since resolution of singularities holds in characteristic zero) so that $\conj$ extends in a unique way as a real analytic involution of $(\ov X{}^\RR,D^\RR)$. Similarly, $\wt X(D)^\RR$, etc.\ belong to the semi-analytic category and $\conj$ can then be lifted in a unique way as a semi-analytic involution $\wt\conj$ of $\wt X(D)^\RR$ that preserves $\partial\wt X(D)^\RR$. Lastly, since the moderate growth or rapid decay condition only involves $\Re(f^\RR)$, the involution $\wt\conj$ preserves the subsets $\wt U_\rmod(D)^\RR$ and $\wt U_\rrd(D)^\RR$.

\begin{cor}
If $(U,f)$ is the extension of $(U_0,f)$ defined over $\QQ$, the $\QQ$-Betti fibers $\coH_\Betti^r(U,f)$ and $\coH^r_{\Betti,\rc}(U,f)$ are naturally endowed with an involution $\conj^*$, which is compatible with the natural morphism $\coH^r_{\Betti,\rc}(U,f)\to\coH_\Betti^r(U,f)$.
\end{cor}

\begin{proof}
Indeed, $\conj^*$ is induced by
\begin{align*}
\wt\conj{}^*:\coH^r_\rc(\wt U_\rmod(D)^\RR,\QQ)&\to\coH^r_\rc(\wt U_\rmod(D)^\RR,\QQ),\quad\text{or}\\
\wt\conj{}^*:\coH^r_\rc(\wt U_\rrd(D)^\RR,\QQ)&\to\coH^r_\rc(\wt U_\rrd(D)^\RR,\QQ).\qedhere
\end{align*}
\end{proof}

\begin{cor}\label{cor:conjiso}
In the setting of Remark \ref{rem:exactseq}, assume that $(U_0,f,Z_0)$ is defined over $\QQ$. Then $\wt\conj{}^*$ is compatible with the morphisms of the Betti exact sequence \eqref{eq:exactseq}. In particular, if \hbox{$\coH^r_\rc(U\moins Z,f)=0$} for all $r$, then the involutions $\conj^*$ on $\coH^r_{\Betti,\rc}(U,f)$ and $\coH^r_\rc(Z^\RR,\QQ)$ coincide.\qed
\end{cor}

\begin{example}\label{ex:conj}
Let us keep the setting of Corollary \ref{ex:dRQiso}, that is, \cite[Ex.\,A.27]{F-S-Y18}. It is proved in \loccit that, for all $r\in\ZZ$, we have a diagram of mixed Hodge structures
\begin{equation}\label{eq:isofZ}
\begin{array}{c}
\xymatrix{
\coH^r_\rc(\Afu_t\times V,f)\ar[r]^-\sim\ar[d]&\coH^{r-2}_\rc(Z)(-1)\\
\coH^r(\Afu_t\times V,f)&\coH^r_Z(\Afu_t\times V)\ar[l]_-\sim
}
\end{array}
\end{equation}
where the vertical arrow is the natural one. Consider the case $r=d$,
the dimension of $\Afu_t\times V$. According to \cite[Prop.\,A.19]{F-S-Y18}, the upper line is mixed of weights $\leq d$ and the lower line is mixed of weights $\geq d$.
Furthermore, denoting by \hbox{$\coH^d_\rmid(\Afu_t\times V,f)$} the image of the vertical arrow, we have induced isomorphisms of pure Hodge structures
\[
\xymatrix{
\gr_d^W\coH^d_\rc(\Afu_t\times V,f)\ar[r]^-\sim\ar[d]^\wr
	& \big(\gr_{d-2}^W\coH^{d-2}_\rc(Z)\big)(-1)\\
\coH^d_\rmid(\Afu_t\times V,f)\ar[r]^-\sim&W_d\coH^d(\Afu_t\times V,f)&W_d\coH^d_Z(\Afu_t\times V)\ar[l]_-\sim
}
\]

We assume that $V_0,g$ are defined over~$\QQ$, making $U_0,f$ also defined over $\QQ$, as well as $\cK_0=g^{-1}(0)$ and $Z_0=\Afu_t\times\cK$, so that $\cK^\RR$ and $Z^\RR$ are preserved by $\conj$. It follows from Corollary \ref{cor:conjiso} that the isomorphisms of (exponential) mixed Hodge structures \eqref{eq:isofZ} induce isomorphisms of Betti fibers which are compatible with $\conj^*$.

Furthermore, the weight filtration $W_\bbullet \coH^r_\rc(Z^\RR,\QQ)$ is preserved by $\conj^*$, since it comes from a filtration defined at the level of Nori motives (\cf\cite[Th.\,10.2.5]{HMS17}). One can argue that the weight filtration of $\coH^r_{\Betti,\rc}(\Afu_t\times V,f)$ is also preserved by $\conj^*$ by analyzing first the behaviour of $\conj^*$ on $R^r\!f_!\,\pQQ_U$. Nevertheless, it is enough for our purpose to check that $\conj^*$ induces an action on $\coH^r_{\Betti,\rmid}(\Afu_t\times V,f)$, a property that follows from interpreting the latter space as the image of the Betti vertical arrow \eqref{eq:isofZ}.
\end{example}

\backmatter
\bibliographystyle{amsplain}
\bibliography{qr-Bessel2}

\providecommand{\eprint}[1]{\href{http://arxiv.org/abs/#1}{\texttt{arXiv\string:\allowbreak#1}}}
\providecommand{\bysame}{\leavevmode\hbox to3em{\hrulefill}\thinspace}
\providecommand{\MR}{\relax\ifhmode\unskip\space\fi MR }
\providecommand{\MRhref}[2]{%
  \href{http://www.ams.org/mathscinet-getitem?mr=#1}{#2}
}
\providecommand{\href}[2]{#2}
\begin{thebibliography}{10}

\bibitem{BKV}
S.~Bloch, M.~Kerr, and P.~Vanhove, \emph{A {F}eynman integral via higher normal
  functions}, Compositio Math. \textbf{151} (2015), 2329--2375.

\bibitem{Broadhurst16}
D.~Broadhurst, \emph{Feynman integrals, {L}-series and {K}loosterman moments},
  Commun. Number Theory Phys. \textbf{10} (2016), no.~3, 527--569.

\bibitem{Broadhurst17}
\bysame, \emph{Critical {$L$}-values for products of up to 20 {B}essel
  functions}, Lecture slides available at
  \url{https://www.matrix-inst.org.au/wp_Matrix2016/wp-content/uploads/2016/04/Broadhurst-2.pdf},
  2017.

\bibitem{Broadhurst17b}
\bysame, \emph{{$L$}-series from {F}eynman diagrams with up to 22 loops},
  Lecture slides available at
  \url{https://multi-loop-2017.sciencesconf.org/data/program/Broadhurst.pdf},
  2017.

\bibitem{B-M16}
D.~Broadhurst and A.~Mellit, \emph{Perturbative quantum field theory informs
  algebraic geometry}, {Loops and legs in quantum field theory (LL2016,
  Leipzig, Germany)}, PoS, 2016, \url{https://pos.sissa.it/260/079}.

\bibitem{B-R18}
D.~Broadhurst and D.~P. Roberts, \emph{Quadratic relations between {F}eynman
  integrals}, {Loops and legs in quantum field theory (LL2018, St.~Goar,
  Germany)}, PoS, 2018, \url{https://pos.sissa.it/303/053}.

\bibitem{BKLF}
\bysame, \emph{{$L$}-series and {F}eynman integrals}, 2017 {MATRIX} annals
  (D.~Wood, J.~de~Gier, C.~Praeger, and T.~Tao, eds.), MATRIX Book Ser.,
  vol.~2, Springer, Cham, 2019, pp.~401--403.

\bibitem{Deligne}
P.~Deligne, \emph{Valeurs de fonctions {$L$} et p\'{e}riodes d'int\'{e}grales},
  Automorphic forms, representations and {$L$}-functions ({C}orvallis, {O}re.,
  1977), {P}art 2, Proc. Sympos. Pure Math., vol. XXXIII, American Mathematical
  Society, Providence, RI, 1979, With an appendix by N. Koblitz and A. Ogus,
  pp.~313--346.

\bibitem{F-S-Y20}
J.~Fres{\'a}n, C.~Sabbah, and J.-D. Yu, \emph{{Quadratic relations between
  periods of connections}}, \eprint{2005.11525}, 2021.

\bibitem{F-S-Y18}
\bysame, \emph{{Hodge theory of Kloosterman connections}}, Duke Math. J. (to
  appear), \eprint{1810.06454}.

\bibitem{Hien07}
M.~Hien, \emph{{Periods for irregular singular connections on surfaces}}, Math.
  Ann. \textbf{337} (2007), 631--669.

\bibitem{Hien09}
\bysame, \emph{{Periods for flat algebraic connections}}, Invent. Math.
  \textbf{178} (2009), no.~1, 1--22.

\bibitem{HMS17}
A.~Huber and S.~M\"uller-Stach, \emph{{Periods and Nori motives}}, Ergebnisse
  der Mathematik und ihrer Grenzgebiete. 3. Folge., vol.~65, Springer, Cham,
  2017.

\bibitem{K-S90}
M.~Kashiwara and P.~Schapira, \emph{{Sheaves on Manifolds}}, Grundlehren Math.
  Wiss., vol. 292, Springer-Verlag, Berlin, Heidelberg, 1990.

\bibitem{K-V95}
H.~T. Koelink and W.~Van~Assche, \emph{Orthogonal polynomials and {L}aurent
  polynomials related to the {H}ahn-{E}xton {$q$}-{B}essel function}, Constr.
  Approx. \textbf{11} (1995), no.~4, 477--512.

\bibitem{K-S10}
M.~Kontsevich and Y.~Soibelman, \emph{{Cohomological Hall algebra, exponential
  Hodge structures and motivic Donaldson-Thomas invariants}}, Commun. Number
  Theory Phys. \textbf{5} (2011), no.~2, 231--352.

\bibitem{Krattenthaler05}
C.~Krattenthaler, \emph{Advanced determinant calculus: a complement}, Linear
  Algebra Appl. \textbf{411} (2005), 68--166.

\bibitem{Laumon83}
G.~Laumon, \emph{Sur la cat{\'e}gorie d{\'e}riv{\'e}e des
  {$\mathcal{D}$}-modules filtr{\'e}s}, {Algebraic geometry (Tokyo/Kyoto,
  1982)}, Lect. Notes in Math., vol. 1016, Springer-Verlag, 1983, pp.~151--237.

\bibitem{Malgrange91}
B.~Malgrange, \emph{{\'E}quations diff{\'e}rentielles {\`a} coefficients
  polynomiaux}, Progress in Math., vol.~96, Birkh{\"a}user, Basel, Boston,
  1991.

\bibitem{Mochizuki10}
T.~Mochizuki, \emph{{Holonomic $\mathcal D$-modules with Betti structure}},
  M{\'e}m. Soc. Math. France (N.S.), vol. 138--139, Soci{\'e}t{\'e}
  Math{\'e}matique de France, Paris, 2014.

\bibitem{Roberts17}
D.~P. Roberts, \emph{Some {F}eynman integrals and their motivic
  interpretation}, Lecture slides available at
  \url{https://www.davidproberts.net/presentations}, 2017.

\bibitem{Bibi94}
C.~Sabbah, \emph{On the comparison theorem for elementary irregular
  {$\mathcal{D}$}-modules}, Nagoya Math.~J. \textbf{141} (1996), 107--124.

\bibitem{MSaito87}
M.~Saito, \emph{{Mixed {Hodge} Modules}}, Publ. RIMS, Kyoto Univ. \textbf{26}
  (1990), 221--333.

\bibitem{MSaito17}
\bysame, \emph{A young person's guide to mixed {H}odge modules}, Hodge theory
  and {$L^2$}-analysis, Adv. Lect. Math. (ALM), vol.~39, Int. Press,
  Somerville, MA, 2017, pp.~517--553, \eprint{1605.00435}.

\bibitem{Vanhove14}
P.~Vanhove, \emph{{The physics and the mixed Hodge structure of Feynman
  integrals}}, String-Math 2013, Proc. Sympos. Pure Math., vol.~88, American
  Mathematical Society, Providence, RI, 2014, pp.~161--194.

\bibitem{Watson22}
G.~N. Watson, \emph{{A treatise on the theory of Bessel functions}}, Cambridge
  University Press, Cambridge, 1922.

\bibitem{Z-C14}
R.~Zhang and L.-C. Chen, \emph{On certain matrices of {B}ernoulli numbers},
  Abstr. Appl. Anal. (2014), 4 pp, Art. ID 536325.

\bibitem{wick}
Y.~Zhou, \emph{Wick rotations, {E}ichler integrals, and multi-loop {F}eynman
  diagrams}, Commun. Number Theory Phys. \textbf{12} (2018), no.~1, 127--192.

\bibitem{Zhou18}
\bysame, \emph{Wro\'{n}skian factorizations and {B}roadhurst-{M}ellit
  determinant formulae}, Commun. Number Theory Phys. \textbf{12} (2018), no.~2,
  355--407.

\bibitem{Zhou19}
\bysame, \emph{Hilbert transforms and sum rules of {B}essel moments},
  Ramanujan~J. \textbf{48} (2019), no.~1, 159--172.

\bibitem{Zhou20}
\bysame, \emph{Wro\'{n}skian algebra and {B}roadhurst-{R}oberts quadratic
  relations}, \eprint{2012.03523}, 2021.

\bibitem{ZhouQind}
\bysame, \emph{$\mathbf{Q}$-linear dependence of certain {B}essel moments},
  Ramanujan~J. (to appear), \eprint{1911.04141}.

\end{thebibliography}
\end{document}